\documentclass[12pt]{amsart}
\usepackage{amssymb,latexsym,amsthm, amsmath, comment, bbm, fullpage, tikz-cd}
\usepackage[all]{xy}
\usepackage{graphicx, comment, libertine, wrapfig}
\usepackage[utf8]{inputenc}
\usepackage[T2A]{fontenc}

\ifx\pdfoutput\undefined
\usepackage{hyperref}
\else
\usepackage[pdftex,colorlinks=true,linkcolor=blue,urlcolor=blue]{hyperref}
\fi

\theoremstyle{plain}

\theoremstyle{definition}
\newtheorem{definition}{Definition}[section]
\newtheorem{theorem}{Theorem}[section]
\newtheorem*{thm}{Theorem}

\newtheorem{lemma}{Lemma}[section]

\newtheorem{corollary}{Corollary}[section]
\newtheorem{proposition}{Proposition}[section]

\newtheorem{remark}{Remark}[section]
\newtheorem*{prop*}{Proposition}

\newtheorem*{ack}{Acknowledgements}

\def\tr{\mathrm{Tr}}

\newcommand{\sh}{\mbox{ш}}

\newcommand{\Jh}{\mbox{Ж}}
\newcommand{\Ch}{\mbox{Ч}}
\newcommand{\RN}{\mbox{Н}}

\newcommand{\RU}{\mbox{У}}
\newcommand{\Rd}{\mbox{д}}

\begin{document}

\title[Quantitative weak mixing for random substitution tilings]
      {Quantitative weak mixing for random substitution tilings}
      \author{Rodrigo Trevi\~no}
      \address{University of Maryland, College Park}
      \email{rodrigo@umd.edu}
      \date{\today}
      \begin{abstract}
        For $N$ compatible substitution rules on $M$ prototiles $t_1,\dots, t_M$, consider tilings and tiling spaces constructed by applying the different substitution rules at random. These give (globally) random substitution tilings. In this paper I obtain bounds for the growth on twisted ergodic integrals for the $\mathbb{R}^d$ action on the tiling space which give lower bounds on the lower local dimension of spectral measures. For functions with some extra regularity, uniform bounds on the lower local dimension are obtained. The results here extend results of Bufetov-Solomyak \cite{BS:translation} to tilings of higher dimensions.
      \end{abstract}
      \maketitle


      \tableofcontents


      \newpage

      \begin{center}
      \begin{tabular}{ |p{1.1in}|p{4.75in}|p{.45in}|  }
        \hline
        \multicolumn{3}{|c|}{Index of Notation} \\
        \hline
        Symbol & Meaning & Section \\
        \hline
        $\mathcal{T}$ & A tiling of $\mathbb{R}^d$  &  \S \ref{subsec:tilings}\\
        $\mathcal{O}_{T}^-(A)$ & Largest patch of $\mathcal{T}$ completely contained in $A$.  &  \S \ref{subsec:tilings}\\
        $\Omega$, $\Omega_\mathcal{T}, \varphi_t$&Tiling spaces and their translation action& \S \ref{subsec:tilings}\\
        $\theta, \theta_{ijk}$&Contraction factor for a substitution rule& \S \ref{subsec:tilings}\\
        $\Lambda$, $\Lambda_i$, $\Lambda_\mathcal{T}$&Set of return vectors&\S\ref{subsubsec:retVectors}\\
        $\Gamma_\mathcal{T}$&Group generated by return vectors $\Lambda$& \S \ref{subsubsec:retVectors}\\
        $\alpha:\Gamma_\mathcal{T}\rightarrow \mathbb{Z}^r$&Address map& \S \ref{subsubsec:retVectors}\\
        $\mathcal{B} = (\mathcal{V},\mathcal{E})$&Bi-infinite Bratteli diagram&\S \ref{subsubsec:biinf}\\
        $X_\mathcal{B}$&Path space of $\mathcal{B}$& \S \ref{subsubsec:biinf}\\
        $\mathcal{P}_k(\bar{e})$&$k^{th}$ approximant of a tiling defined by $\bar{e}$&\S \ref{sec:RandSub}\\
        $\Omega_x$& Tiling space constructed from the random substitutions defined by $x$&\S \ref{sec:RandSub}\\
        $\theta_{(n)_x}$& Product $\theta_{(n)_x} = \theta_{x_1}\cdots \theta_{x_n}$ of contraction factors&\S \ref{sec:RandSub}\\
        $\mathcal{T}_{\bar{e}}$&Tiling defined by $\bar{e}\in X_{\mathcal{B}}$& \S \ref{sec:RandSub}\\
        $c$, $\mathcal{C}_\mathcal{T}$&Choice function, control points &\S \ref{subsec:choices}\\
        $\mathcal{P}_{k}(v_i)$&Canonical level-$k$ supertile of type $i$&\S \ref{subsubsec:controlPts}\\
        $\mathcal{M}_k^\pm$&Multimatrix/finite dimensioncal $C^*$-algebra &\S \ref{subsec:algebras}\\
        $i_k^\pm$&Inclusion of $\mathcal{M}_{k-1}^\pm$ into $\mathcal{M}_k^\pm$ &\S \ref{subsec:algebras}\\
        $LF(\mathcal{B}^\pm)$&Locally finite $*$-algebras, inductive limit of $(\mathcal{M}^\pm_k,i^\pm_k)$&\S \ref{subsec:algebras}\\
        $\tr(\mathcal{B}^\pm)$,$\tr^*(\mathcal{B}^\pm)$&Trace and cotrace spaces of $LF(\mathcal{B}^\pm)$&\S \ref{subsec:traces}\\
        $\mathfrak{t}_{k,i}$&Basis (trace) elements of $\tr(\mathcal{M}_k)$&\S \ref{subsec:traces}\\
        $\Phi_x$&Renormalization map from $\Omega_x$ to $\Omega_{\sigma(x)}$&\S \ref{sec:renorm}\\
        $\Lambda_x^{(k)}, \Lambda_{x,i}^{(k)},\Gamma^{(k)}_x$ &Return vectors of level-$k$ supertiles and the group they generate& \S \ref{subsec:RV}\\
        $G_x$, $\bar{G}_x$&Renormalization cocycle on return vectors defined by $G_x\alpha_x(\tau) = \alpha_{\sigma(x)}(\theta_{x_{-1}}^{-1}\tau)$, $\tau\in \Lambda_x$, and its dual&\S \ref{subsec:RV}\\
        $\mathfrak{R}_x$, $\mathfrak{R}_x^*$&Real vector space spanned by the basis of $\Gamma_x$ and its dual&\S \ref{subsec:RV}\\
        $\mathfrak{G}:\mathfrak{R}_N\rightarrow \mathfrak{R}_N$&Return vector cocycle $(x,v)\mapsto (\sigma^{-1}(x),\bar{G}_xv)$&\S \ref{subsec:RV}\\
        $G_+$&Trace cocycle on the trace bundle $\tr^+(\Sigma_N)$&\S \ref{subsec:traceCocycle}\\
        $\Theta^{(k)}_x$&$\Theta^{(k)}_x = \mathcal{A}_k^x\cdots \mathcal{A}_1^x$ is the matrix product given by the trace cocycle&\S \ref{subsec:traceCocycle}\\
        $\mathcal{W}^\ell$&Set of return vectors for the substitution rule $S_\ell$&\S \ref{sec:LFreturn}\\
        $a_{x,i,\lambda}^k$&Twisted elements of the LF algebra $LF(\mathcal{B}^+_x)$&\S \ref{sec:LFreturn}\\
        $w = w^-w^+$&Positively simple word&\S \ref{sec:LFreturn}\\
        $\mathcal{S}_R^\mathcal{T}(f,\lambda)$&Twisted integral of $f$ by $\lambda$ in $[-R,R]^d$&\S \ref{sec:twisted}\\
        $t_j^{(k)}$&Level-$k$ supertile of type $j$&\S \ref{sec:twisted}\\
        $AP(\Omega), AP^k(\Omega)$&Anderson-Putnam (AP) complex of $\Omega$, and the level-$k$ AP complex&\S \ref{subsec:AP}\\
        $[\Rd]\in H^1(\Omega;\mathbb{R}^d)$&Deformation/shape vector represented by the 1-form $\Rd$&\S \ref{subsec:deform}\\
        $\mathcal{T}_{\Rd}$, $\Omega_x^{\Rd}$&Tiling and tiling space deformed by $\Rd$&\S \ref{subsec:deform}\\
        $\mathcal{D}^x_N(w,\rho,[\Rd],\lambda)$&Density of times up to $N$ for which the fiber dynamics $\Phi_x^{(n)*}[\Rd](\lambda)$ stay away from a lattice points in $H^1$&\S \ref{sec:CQVC}\\
        $\mathfrak{B}_x(\varrho,\delta,\vartheta)$&Set of bad deformation parameters $[\Rd]\in H^1(\Omega;\mathbb{R}^d)$ for which $\mathcal{D}_N^x<1-\delta$ cannot be guaranteed for arbitrarily large $N$&\S \ref{sec:dimension} \\
        \hline
      \end{tabular}
      \end{center}
      
      \newpage
      
      \section{Introduction and statement of results}
      Let me start describing what \textbf{quantitative} weak mixing is for a finite measure-preserving flow $\varphi_t:\Omega\rightarrow \Omega$ on a compact metric space $\Omega$. Weak mixing is equivalent to the non-existence of non-constant $L^2$-eigenfunctions for the flow, that is, non-existence of functions $f$ that satisfy $f\circ \varphi_t = e^{2\pi \imath t\lambda}f$ for all $t\in\mathbb{R}$ and some $\lambda\in\mathbb{R}$ (the eigenvalue). One can quantify how far away a system is from having an eigenfunction: given a function $f\in L^2$, its \textbf{spectral measure} $\mu_f$ is obtained, by Bochner's theorem, by taking the Fourier transform of the correlation function $\langle f\circ \varphi_t, f\rangle$. In the case of $f$ being an eigenfunction with eigenvalue $\lambda$, the spectral measure $\mu_f$ is a Dirac measure at $\lambda$. Recall that the \textbf{lower and upper local dimensions} of a locally finite Borel measure $\mu$ are the quantities
      $$d^-_\mu(x) := \liminf_{r\rightarrow 0^+}\frac{\log \mu (B_r(x))}{\log r}\hspace{.4in}\mbox{ and }\hspace{.4in}d^+_\mu(x) := \limsup_{r\rightarrow 0^+}\frac{\log \mu (B_r(x))}{\log r}.$$
      In the case of a spectral measure $\mu = \mu_f$, the local dimensions will be denote by $d_f^\pm(x): = d_{\mu_f}^\pm(x)$.
      
      The upper and lower dimensions of the spectral measure of an eigenfunction are zero. Quantitative weak mixing therefore involves bounding the local dimension of spectral measures from below, that is, away from the case of Dirac measures, which appear for eigenfunctions.
      
      Over the past decade, a series of works by Bufetov and Solomyak \cite{BS:substitution, BS:genus2, BS:spectralCocycle, BS:translation} have studied the issue of quantitative weak mixing for an increasingly large class of minimal and uniquely ergodic flows, culminating with a very large class of non-stationary, one-dimensional tiling spaces, a subset of which models a typical flow on flat surfaces of genus greater than or equal to two. Recently \cite{forni:twist} Forni has developed an extension from his theory of flat surfaces in \cite{forni:deviation} through the use of twisted cohomology in order to obtain results on quantitative weak mixing for translation flows on typical higher-genus surfaces. Forni's contribution \cite{forni:twist} was done after the results of Bufetov-Solomyak \cite{BS:genus2} for translation flows on genus two, but before their results for flows on surfaces of arbitrary genus. However, the results  of \cite{BS:translation} cover a very general class of one-dimensional tiling spaces as well as translation flows on flat surfaces of infinite genus.

      Much less has been done in this direction for higher rank actions, in particular for actions of $\mathbb{R}^d$. Spaces coming from tilings are natural candidates for these questions and some spectral results are known. For example, Solomyak \cite{SolomyakEigenfunctions} has characterized weak mixing for tiling spaces which are constructed from a single substitution rule. Emme \cite{Emme:spectral} has studied the behavior of spectral measures at zero for self-similar tilings of $\mathbb{R}^d$. These results are extended  in the (unpublished) Master's thesis of Marshall \cite{Marshall:master}, with results on quantitative weak mixing for a class of self-similar tilings of $\mathbb{R}^d$. Although it does not directly addresses questions related to eigenfunctions, the work \cite{BGM:correlation} brings forth methods to determine the absence of absolute continuity of spectral measures in a large class of self-similar tilings of any dimension.
      
      This paper is a contribution to the expanding list of systems for which there are results on quantitative weak mixing. In \cite{ST:random, T:TTT}, a method of studying (globally) random substitution tilings was introduced and developed. The idea is as follows: suppose there are $N$ different compatible substitution rules $S_1,\dots, S_N$ on $M$ prototiles $t_1,\dots, t_M$, and that a tiling was created from these by applying the different substitution rules at random over all the tiles. This procedure would create a tiling whose structure would not be self-similar, but the structure would be represented by the order in which the substitutions were applied to construct the tiling. In other words, some of the structure would be recorded by a point $x\in \Sigma_N$ in the $N$-shift. This is very different from the (locally) random constructions where not only are the substitutions chosen at random, but so are the tiles on which they are applied. Systems of globally random substitution similar to the one here sometimes go by the name S-adic systems or mixed substitution systems \cite{GM:mixed}.
      
      The study of (globally) random substitution tilings is made possible through the use of Bratteli diagrams (see \S \ref{subsec:brat} for definitions): they hold the information of both how substitutions are applied and in which order. More specifically, given $x = (x_1,x_2,\dots)\in\Sigma_N^+$, the level $n$ of the Bratteli diagram $B_x$ encodes the substitution $S_{x_n}$. There is a continuous map $\Delta_x:X_{B_x}\rightarrow \Omega_x$ from the path space of $B_x$ to a tiling space $\Omega_x$. Very importantly, with this construction come homeomorphisms of tiling spaces $\Phi_x:\Omega_x\rightarrow \Omega_{\sigma(x)}$. 

      As in the works of Bufetov-Solomyak and Forni, the route to proving estimates on the local dimension for spectral measures are via the study of \textbf{twisted ergodic integrals}. That is, suppose $\mathbb{R}^d$ acts on $\Omega$ minimally and preserving a Borel probability measure $\mu$. For some $f\in L^2(\Omega,\mu)$ and $\lambda\in\mathbb{R}^d$, the twisted integral of $f$ by $\lambda$ over $C_R(0)$ is
      $$\mathcal{S}_R^{x}(f,\lambda):=\int_{C_R(0)}e^{-2\pi \imath \langle \lambda, t\rangle }f\circ \varphi_t(x)\, dt$$
      where $C_R(0)$ is the $d$-cube of sidelength $2R$ centered around the origin. The relation between twisted integrals and quantitative weak mixing is given by the following fact, which goes back to Hof \cite{Hof:scaling} (see \S \ref{sec:proofs} for the proof).
      \begin{lemma}
        \label{lem:dimension}
        Suppose $\mathbb{R}^d$ acts on $\Omega$ preserving a Borel probability measure $\mu$. Suppose that for some $\lambda\in\mathbb{R}^d$, $R_0>0$, $f\in L^2(\Omega,\mu)$ and $\alpha\in(0,d)$
  $$\left\|\mathcal{S}_R(f,\lambda)\right\|_{L^2}\leq C_1R^{d-\alpha}$$
        for all $R>R_0$. Then
        $$\mu_f(B_r(\lambda))\leq Cr^{2 \alpha}$$
        for all $r<1/2R_0$ and some $C>0$. In particular, the lower local dimension of the spectral measure satisfies
        $$d^-_f(\lambda)\geq 2 \alpha.$$
      \end{lemma}

      A common tool in effective ergodic theory (that is, the study of rates of convergence of averages appearing in ergodic theory) are the so-called renormalization tools such as renormalization cocycles (see \cite{Forni:ICM}). These are linear maps defined on a vector bundle over a ``space of systems'', and the growth properties of vectors under these bundle maps determine the speed of convergence of certain averages on a typical system. In our setting here the ``space of systems'' is the collection of tiling spaces $\{\Omega_x\}$ parametrized by $x\in\Sigma_N$, and there are several bundles of interest, see below. In our case, the dynamics on the base space are the maps $\Phi_x:\Omega_x\rightarrow \Omega_{\sigma(x)}$ defined by our construction.
      
      Given the topological nature of tiling spaces, the methods used here are more in line with those of Bufetov-Solomyak than those of Forni. However, there is a crucial difference between the work here and both of those works: both \cite{BS:translation} and \cite{forni:twist} make use of a renormalization cocycle which describes \emph{both} the hierarchical structure of the foliated space as well as the structure of return vectors to a special transversal because, for $\mathbb{R}$ actions, both of these objects coincide, since all return vectors are colinear. Both of those cocycles have been used to study the rates of deviation of ergodic integrals for functions with some regularity. 

      The situation is different for $\mathbb{R}^d$ actions when $d>1$, the higher rank case. More precisely, there is a cocycle which relates the hierarchical structures of tiling spaces $\{\Omega_{\sigma^k(x)}\}$ along the orbit of $x$. This is called the \textbf{trace cocycle} in \cite{T:TTT}, where it was shown that its behavior describes the behavior of ergodic integrals of the $\mathbb{R}^d$ action on tiling spaces. But there is another cocycle at play. Given a transversal $\mho_x\subset \Omega_x$ in a tiling space, there is a set of vectors $\Lambda_x\subset\mathbb{R}^d$ which describe returns to the transversal. The group $\Gamma_x:=\mathbb{Z}[\Lambda]$ generated by integer combinations of these vectors is a finitely generated Abelian group whenever the tilings in $\Omega_x$ have finite local complexity. The renormalization maps $\Phi_{\sigma(x)}:\Omega_{x}\rightarrow \Omega_{\sigma(x)}$ induce a map from $\Gamma_x$ to $\Gamma_{\sigma(x)}$ which, under a careful choice of generators for $\Gamma_x$ and $\Gamma_{\sigma(x)}$, is represented by an integer matrix. This map is called the \textbf{return vector cocycle}, and its behavior gives the estimates necessary to bound spectral measures from below. Thus, in order to obtain quantitative weak mixing results for higher rank actions, the simultaneous understanding of two renormalization cocycles is needed.

      How can these two different renormalization cocycles be understood? A quick answer is cohomologically. It is shown in \S \ref{subsec:RetVecCoh} that every return vector $v\in\mathbb{R}^d$ gives a cohomology class $[v]\in \check H^1(\Omega;\mathbb{Z})$ (more generally, in $H^1$ of the Lie-algebra cohomology of the $\mathbb{R}^d$ action). Whereas the usual renormalization cocycle describes the action of the renormalization maps on the top level cohomology of the tiling space $H^d(\Omega;\mathbb{R})$ (more generally, $H^d$ of the Lie-algebra cohomology of the $\mathbb{R}^d$ action), the return vector cocycle describes the action of the renormalization maps on $H^1(\Omega;\mathbb{R})$. Thus, when $d>1$, these are different cocycles and both of these are needed for quantitative weak mixing results. The return vector cocycle in fact projects of the cocycle $\Phi^*_x$ induced by the renormalization maps on the first cohomology $H^1(\Omega;\mathbb{R})$ of the tiling space (see \S\ref{subsec:RetVecCoh} for details).

      One way in which one can pick tilings at random is to choose the sequence of substitutions at random, parametrized by a point $x\in\Sigma_N$ according to some shift-invariant probability measure on $\Sigma_N$. But there is another way of randomizing the construction of a tiling, and that is through deformations (see \S\ref{subsec:deform} for definitions related to deformations). For any tiling $\mathcal{T}$ of finite local complexity the set $H^1(\Omega_\mathcal{T};\mathbb{R}^d)$ is a natural space of deformations of the tiling, a subset of which can be thought of as a moduli space for the tiling space $\Omega_\mathcal{T}$. Therefore, in addition to picking the sequences of substitutions at random, one can pick the geometry of the tilings at random from $H^1(\Omega;\mathbb{R}^d)$. In dimension 1 this amounts to picking the length of the tiles in the tiling. However, for dimension greater than 1 the geometry of the tiles can be varied in much more interesting ways. The standard language to use in the deformation of orbits is that of \textbf{time changes}. For higher rank actions, and especially in tiling spaces, the language of \textbf{shape changes} is also used.    The goal of this paper is to give sufficient conditions for when random choices of hierarchical structure (i.e. order of substitutions) and geometry (i.e. deformation parameter or time change) give tilings and tiling spaces which are quantitatively weakly mixing. The results here can be seen as a quantitative, non-stationary and measure-theoretic generalization of the topological results for self-similar tilings in \cite{ClarkSadun:shape} (see (vi) below in Remark \ref{rem:1}).
      \subsection{Statement of results}
      Let me describe everything needed to understand the main result. A substitution rule $S$ is called \textbf{uniformly expanding} if there is a $\theta\in(0,1)$ such that every prototile is expanded by $\theta^{-1}$ under the substitution rule $S$. Let $S_1,\dots, S_N$ be a collection of substitution rules on $M$ prototiles $t_1,\dots, t_M$. See \S \ref{subsec:tilings} for the necessary background on tilings and tiling spaces. Here it will be assumed that the substitution rules are \textbf{compatible} in the sense of \cite{GM:mixed} (see \S \ref{sec:RandSub}): any arbitrary combination of them produces tilings in which tiles meet edge to edge and have finite local complexity. That this is a large and interesting class of tilings at least in dimension two is evident from the experimental results in \cite{GKM:computer}.
      
      As mentioned in the previous section, by picking a element of $x = (x_1,x_2,\dots)\in\Sigma_N^+$ one specifies a sequence of substitutions $S_{x_1},S_{x_2},\dots$ from which one can try to construct a tiling space $\Omega_x$. Naturally, some conditions need to be imposed so that, given a set of substitution rules, a choice of $x$ will in fact produce a tiling space $\Omega_x$. Tiling spaces of aperiodic tilings are metric spaces\footnote{In this paper all tiling spaces considered are compact.} which are foliated by the orbits of a natural action of $\mathbb{R}^d$ by translation, where $d$ is the dimension of the tiling (see \S\ref{sec:background}). The construction involves first creating a Bratteli diagram $B_x$, an infinite directed graph, which records not only the order of the substitutions given by $x$, but also some accompanying combinatorial information about the substitution. This combinatorial information can be expressed in a sequence of integer valued, non-negative matrices $\{A_k\}_k$, where $A_i$ encodes some information from the substitution $S_{x_i}$ (in fact $A_i$ is the substitution matrix $F_{x_i}$ for $S_{x_i}$). The basic necessary condition in order to construct an unambiguous tiling space $\Omega_x$ from $x$ is requiring that the sequence $\{A_k\}$ has the property that for every $m\in\mathbb{N}$ there is an $n>m$ such that the matrix $A_n\cdots A_m$ has all positive entries. Note that this property is $\sigma$-invariant: if the matrices for $B_x$ have this property then the matrices for $B_{\sigma(x)}$ have this property. A $\sigma$-invariant measure is called \textbf{minimal} if for $\mu$-almost every $x$, the matrices for $B_x$ have this property. Typical points of minimal measures give tiling spaces with minimal actions of $\mathbb{R}^d$. See \S\ref{sec:RandSub} for more details.
      
      A further condition is needed to obtain quantitative weak mixing results. A word $w  = w_1w_2\dots w_n$ is \textbf{simple} if $w_i\cdots w_n\neq w_1\cdots w_{n-i+1}$ for all $1<i\leq n$. A word $w$ is \textbf{positively simple} if it is a simple word and in addition:
      \begin{enumerate}
      \item it breaks up into two subwords $w= w^-w^+ = w_1^-\cdots w_{n^-}^-w_1^+\cdots w_{n^+}^+$, and
      \item each entry in both matrices $Q^\pm := F_{w_{n^\pm}^\pm}\cdots F_{w_1^\pm}$ is strictly greater than 1, where $F_{w_j}$ is the substitution matrix for $S_{w_j}$.
      \end{enumerate}
      Let $C([w_1.w_2])\subset \Sigma_N$ be the cylinder set with the word $w_1w_2$ around the origin. A $\sigma$-invariant probability measure $\mu$ on $\Sigma^+_N$ for which $\mu(C([w^-.w^+]))>0$ for a positively simple word $w$ is called a \textbf{positively simple measure}. Any such measure is a minimal measure. The existence of a positively simple word which occurs infinitely often in $x$ guarantees the existence of a finite set $\Lambda_w\subset \mathbb{R}^d$ of vectors which represent useful return vectors of arbitrarily large scale for the $\mathbb{R}^d$ action on $\Omega_x$. This is covered in \S\ref{sec:LFreturn}.
      
      Let $r_x$ be the rank of the group $\Gamma_x=\mathbb{Z}[\Lambda_x]$ generated by integer combinations of return vectors. Chosing a set of generators $\{v_i\}$ of $\Gamma_x$ one can define the homomorphism called the \textbf{address map} $\alpha:\Gamma_x\rightarrow \mathbb{Z}^{r_x}$ by sending the generators $\{v_i\}$ to the standard generators. As such, every element $\tau\in\Lambda_x$ has an address $\alpha(\tau)$. A positively simple measure $\mu$ (corresponding to a positively simple word $w$) is called \textbf{postal} if $\alpha(\Lambda_w)$ generates $\mathbb{Z}^{r_x}$, that is, if there are enough return vectors in $\Lambda_w$ so that any address can be found with these vectors through the address map.

      Let $\mu$ be a positively simple, postal, $\sigma$-invariant, ergodic probability measure on $\Sigma_N:= \{0,\dots, N-1\}^\mathbb{Z}$ such that the return vector cocycle is integrable (the return vector cocycle is defined in \S\ref{subsec:RV} ). For any point $x\in\Sigma_N$ consider the corresponding tiling space $\Omega_x$. As mentioned above, the geometry of the tilings in $\Omega_x$ are given by a representative $\Rd_x$ of a class $[\Rd_x]\in H^1(\Omega_x;\mathbb{R}^d)$, which is a finite dimensional vector space carrying a natural Lebesgue measure. There is an open subset $\mathcal{M}_x\subset H^1(\Omega_x;\mathbb{R}^d)$ which parametrizes classes of deformations (equivalently, time changes) of tilings in $\Omega_x$ (see \S\ref{subsec:deform}) which inherits a natural absolutely continuous measure from $H^1(\Omega_x;\mathbb{R}^d)$. Any tiling space obtained by deforming the tiling space $\Omega_x$ by deforming through a representative $\Rd$ of a class $[\Rd]\in \mathcal{M}_x$ will be denoted by $\Omega_x^{\Rd}$. The set $\mathcal{M}_x$ can be thought of as a moduli space of $\Omega_x$.
      
      The shift $\sigma:\Sigma_N\rightarrow \Sigma_N$ induces a homeomorphism of tiling spaces $\Phi_x:\Omega_x\rightarrow \Omega_{\sigma(x)}$ which generalizes the hyperbolic self-homeomorphisms which appear in self-similar tiling spaces. Such maps induce an action $\Phi_x^*:H^*(\Omega_{\sigma(x)};\mathbb{Z})\rightarrow H^*(\Omega_{x};\mathbb{Z})$ on cohomology. Under the assumption that $\log \|\Phi_x^*\|$ is $\mu$-integrable for some $\sigma$-invariant and ergodic $\mu$, the Oseledets theorem gives a decomposition of the first cohomology space $H^1(\Omega_x;\mathbb{R}) = Z^+_x\oplus Z^-_x$, where $Z^+_x$ is the Oseledets space of strictly positive Lyapunov exponents. Throughout the paper, $C_R(0) := [-R,R]^d\subset \mathbb{R}^d$.

      Let $\mu$ be a positively simple measure (for a positively simple word $w$) and denote by $k_1(x)$ the first return of $x$ to $C([w^-.w^+])$ under $\sigma^{-1}$. Let $\mu(x|x^-)$ be the conditional distribution of $x$'s conditioned on the past $x^- = \dots x_{-2} x_{-1}$.
      \begin{theorem}
        \label{thm:main}
        Let $S_1,\dots, S_N$ be a collection of $N$ uniformly expanding compatible substitution rules on the $M$ prototiles $t_1,\dots, t_M$ of dimension $d$ and $\mu$ a $\sigma$-invariant positively simple, postal ergodic probability measure on $\Sigma_N$ (with corresponding positively simple word $w = w^-.w^+$) such that $\log \|\Phi_x^*\|\in L^1_\mu$, where $\Phi^*_x$ is the map on the first cohomology of $\Omega_x$ induced by the shift. Suppose there is an $\varepsilon>0$ and $C>1$ such that $\int_{C([w^-.w^+])}\|\Phi^*_{\sigma^{-k_1(x)-1}(x)}\cdots \Phi^*_{\sigma^{-1}(x)} \|^\varepsilon \, d\mu(x|x^-) \leq  C$ for all $x^-$ ending with $w^-$. Then if $\mathrm{dim}(Z^+)>d$, there exists a constant $\alpha_\mu\in(0,1]$ such that for $\mu$-almost every $x\in\Sigma_N$, there is a $\vartheta>1$ such that for almost every $[\Rd]\in\mathcal{M}_x$ there exists a representative $\Rd$ such that for any $B\geq 2$, any Lipschitz function $f:\Omega_x^{\Rd}\rightarrow \mathbb{R}$, $\mathcal{T}\in\Omega_x^{\Rd}$  and $\lambda\in\mathbb{R}^d$ with $B^{-1}< \|\lambda\|\leq B$, 
        $$\left|\int_{C_R(0)}e^{-2\pi \imath \langle \lambda, t\rangle }f\circ \varphi_t(\mathcal{T})\, dt\right|\leq C_{f,x} R^{d-\alpha_\mu},$$
        for all $R>R_0(x)\cdot B^\vartheta$. In particular, the lower local dimension of the spectral measures is bounded from below:
        $$2\alpha_\mu\leq d_f^-(\lambda).$$
      \end{theorem}
      I should point out that in this setting it is always true that $\mathrm{dim}\, Z^+_x \geq d$ for a tiling of dimension $d$. Thus the theorem requires one extra expanding Oseledets subspace. The consequence is that a generic time change of a self-similar, aperiodic tiling with large-enough unstable space does not have any eigenfunctions (see Remark (vi) below).
      \begin{corollary}
        \label{cor:main}
        Let $S$ be a primitive substitution rule on the $M>1$ prototiles $t_1,\dots, t_M$ of dimension $d$. If $\mathrm{dim}(Z^+)>d$ then there exists a constant $\alpha_\mu\in(0,1]$ and $\vartheta>1$ such that for almost every $[\Rd]\in\mathcal{M}_x$ there is a representative $\Rd$ such that for any $B>1$, any Lipschitz function $f:\Omega_x^{\Rd}\rightarrow \mathbb{R}$ of zero average, $\mathcal{T}\in\Omega_x$ and $\lambda\in\mathbb{R}^d$ with $B^{-1}< \|\lambda\|\leq B$, 
          $$\left|\int_{C_R(0)}e^{-2\pi \imath \langle \lambda, t\rangle }f\circ \varphi_t(\mathcal{T})\, dt\right|\leq C_{f,x} R^{d-\alpha_\mu},$$
          for all $R>N_0 B^\vartheta$. In particular, the lower local dimension of the spectral measures is bounded from below:
        $$2\alpha_\mu\leq d_f^-(\lambda).$$
      \end{corollary}
      \begin{remark}
        \label{rem:1}
        Some remarks:
        \begin{enumerate}
        \item The main theorem does not imply all of the results of \cite{BS:translation}, which cover the case $d=1$. The reason is that the hypotheses here assume that the substitutions are uniformly expanding, which is not an assumption in \cite{BS:translation}. Removing this assumption for $d>1$ may give tilings of arbitrarily many scales, leading to tilings of infinite local complexity, which significantly complicates the analysis of the group of return vectors and cohomology spaces. This is not an issue for $d=1$ because all boundaries of tiles are the same. On the other hand, \cite{BS:translation} assumes the invertibility of the substitution matrices, which is not an assumption made here, although it seems to me like the results in \cite{BS:translation} can do without such assumption. 
        \item The value of $\alpha_\mu$ can be explicitly given in terms of the dynamical behavior of the return vector cocycle (defined in \S\ref{subsec:RV}); see Proposition \ref{prop:Veech} in \S\ref{sec:criterion}.
        \item There are two main steps in obtaining this theorem. The first is to proove a ``quantitative Veech criterion'' which gives bounds on the growth of the twisted ergodic integrals (and thus on the lower local dimension of the associated spectral measures) based on certain behavior of the return vector cocycle (this is proved in \S\ref{sec:criterion}). Applying this criterion to any given tiling space is not straight-forward, and so to apply this criterion, however, one has to introduce deformations, which is why a set of deformations has to play a role in the statement of the theorem. Using deformations, it is shown that the set of deformation parameters which do not satisfy the criterion has dimension strictly less than the set of deformation parameters. This is done through a ``linear exclusion of bad parameters'' argument, first introduced by Avila-Forni \cite{AvilaForni:WM}. The version employed here follows the Erd\H{o}s-Kahane method of Bufetov-Solomyak adapted to higher rank actions (see \S \ref{sec:dimension}).
        \item That $\alpha_\mu$ is bounded above by $1$ is a consequence that boundary effects appear in ergodic integrals for higher rank actions; see \S \ref{sec:twisted}.
        \item The ``simple'' part of a positively simple word implies that this word has no overlaps. These types of words always exist in a subshift as long as there are periodic orbits of arbitrary high period (see e.g. \cite[Lemma 2.6]{Yang:normal}). Therefore, invariant measures supported on mixing subshifts of finite type are positively simple measures. This was what was used in \cite{BS:translation} to apply their quantitative weak mixing results to flat surfaces: Rauzy-Veech induction has as base dynamics a mixing subshift of finite type, and so the Masur-Veech measure has the property that makes it a positively simple measure.
        \item In \cite[Theorem 1.7]{ClarkSadun:shape}, Clark and Sadun give a cohomological condition so that there exist representatives $\Rd$ for almost any choice of deformation class $[\Rd]$ such that the system is topologically weak mixing when the starting system is self-similar. By a result of Solomyak \cite[Theorem 3.3]{SolomyakEigenfunctions}, this implies measure theoretic weak mixing for typical deformation parameters in the self-similar case. Theorem \ref{thm:main} above can be seen as a non-stationary, quantitative, and measure-theoretic generalization of their result. Corollary \ref{cor:main} here requires more conditions than \cite[Theorem 1.7]{ClarkSadun:shape} but has stronger conclusion.
        \item Two dimensional examples for which the theorem applies can be constructed from the results of \cite{GKM:computer}.
        \end{enumerate}
      \end{remark}    

      Weak mixing is defined as the convergence in the average of correlations of functions of zero mean. The quantitative weak mixing results yield uniform rates of weak mixing for functions with sufficiently regularity in the leaf direction. More precisely, let $\Omega$ be a tiling space whose elements are repetitive, aperiodic tilings of dimension $d$ of finite local complexity. Assume that the $\mathbb{R}^d$ action on $\Omega$ is uniquely ergodic and denote by $C^1_L(\Omega)$ the set of Lipschitz functions on $\Omega$ whose leafwise derivative in any direction is Lipschitz.
\begin{theorem}
  \label{thm:main2}
  Let $S_1,\dots, S_N$ be a collection of $N$ uniformly expanding compatible substitution rules on the $M$ prototiles $t_1,\dots, t_M$ of dimension $d$ and $\mu$ a $\sigma$-invariant positively simple, postal ergodic probability measure on $\Sigma_N$ (with corresponding positively simple word $w = w^-.w^+$) such that $\log \|\Phi_x^*\|\in L^1_\mu$, where $\Phi^*_x$ is the map on the first cohomology of $\Omega_x$ induced by the shift. Suppose there is an $\varepsilon>0$ and $C>1$ such that $\int_{C([w^-.w^+])}\|\Phi^*_{\sigma^{-k_1(x)-1}(x)}\cdots \Phi^*_{\sigma^{-1}(x)} \|^\varepsilon \, d\mu(x|x^-) \leq  C$ for all $x^-$ ending with $w^-$. Then if $\mathrm{dim}(Z^+)>d$, there exists a constant $\alpha_\mu'\in(0,1]$ such that for $\mu$-almost every $x\in\Sigma_N$, for almost every $[\Rd]\in\mathcal{M}_x$ there exists a representative $\Rd$ such that for any $\varepsilon>0$, and any $f,g\in C^1_L(\Omega)$ of zero average,
    \begin{equation}
      \label{eqn:avgCorr}
  \int_{C_R(0)}|\langle f\circ \varphi_t,g\rangle|\, dt \leq C_{f,g,\varepsilon} R^{d-\frac{\alpha'_\mu}{2}+\varepsilon}
    \end{equation}
    for all $R>1$. Moreover, for all $\lambda\in \mathbb{R}^d$:
    \begin{equation}
      \label{eqn:uniform}
      \left|\int_{C_R(0)}e^{-2\pi \imath \langle \lambda, t\rangle }f\circ \varphi_t(\mathcal{T})\, dt\right|\leq C_{f,\varepsilon}\max\left\{R^{d-1+\frac{4}{4+\alpha'_\mu}+\varepsilon},R^{\frac{d+\lambda^*}{2}+\varepsilon}\right\},
    \end{equation}
        for $R>1$, where $\lambda^* = \max\{d-1,d\frac{\lambda_2}{\lambda_1}\}$. In particular, the lower local dimension of the spectral measures is bounded from below for any $\lambda\in\mathbb{R}^d$:
        $$d^-_f(\lambda)\geq \min\left\{\frac{2\alpha'_\mu}{4+\alpha'_\mu}, d\left(1-\frac{\lambda_1}{\lambda_2}\right)\right\}.$$
\end{theorem}
\begin{remark}
Some remarks:
  \begin{enumerate}
  \item The proof for the uniform bounds in Theorem \ref{thm:main2} obtained from non-uniform bounds in Theorem \ref{thm:main} for smooth functions apply to the non-uniform bounds in \cite{BS:translation}, thus giving uniform estimates for their results in dimension 1, bringing them closer to the uniform bounds of Forni in \cite{forni:twist}. In fact, the proof of (\ref{eqn:avgCorr}) follows Forni's argument in \cite[Lemma 9.3]{forni:twist}. The uniform bounds in (\ref{eqn:uniform}) follow from (\ref{eqn:avgCorr}) using an argument of Venkatesh \cite[Lemma 3.1]{venkatesh:sparse}.
  \item Theorem \ref{thm:main2} is of more relevance to the mathematical physics literature than Theorem \ref{thm:main} given the importance of correlation functions in the theory of diffraction. 
  \end{enumerate}
\end{remark}
The bounds for a twisted integral $\mathcal{S}_R(f,\lambda)$ with a spectral parameter $\lambda$, and thus the bounds for the lower local dimension of spectral measures at $\lambda$, is obtained by bounding traces applied to specially-constructed elements of locally finite algebras (see \S \ref{sec:LFreturn}). As detailed in Remark \ref{rem:spectralCocycle}, this can be rewritten in terms of growth of a family of cocycles $\mathbb{M}_x^{(k)}(\lambda)$, parametrized by spectral parameters $\lambda\in\mathbb{R}^d$, over the shift $\sigma:\Sigma_N\rightarrow \Sigma_N$, and which coincides with the trace cocycle (also, the substitution matrix cocycle) when $\lambda=0$. This cocycle goes by several names in the literature: it is the spectral cocycle of Bufetov-Solomyak \cite{BS:spectralCocycle}, the Fourier matrix cocycle of Baake-G\"ahler-Ma\~nibo \cite{BGM:correlation}, or the twisted cocycle of Forni \cite{forni:twist}. In this paper, the growth properties are bounded from above in order to obtain lower bounds for the local dimension of spectral measures; it is unclear how tractable it is to derive upper bounds.

At the beginning of this century there was some contention as to what the definition of a crystal should be \cite{lifshitz:what}. A prevailing view has emerged, requiring a crystal to have long-range positional order which can be infered from the existence of Braggs peaks in the Fourier spectrum of the solid. If tilings serve as mathematical models of crystals, then this means that tilings with some discrete spectrum are candidates for being models for crystals. The results in this paper show that, left to our own devices and constructing tilings at random, we are almost surely not constructing models for any type of crystal. I think this makes both crystals and tilings constructed at random special.
      
      This paper is organized as follows. In \S\ref{sec:background} the necessary background for tilings and Bratteli diagrams is covered. This is followed in \S\ref{sec:RandSub} by a review of the construction of random substitution tilings using graph iterated function systems, which was a construction introduced in \cite{ST:random}. The basic theory of locally finite dimensional algebras and their traces is covered in \S \ref{sec:algebras}. The ideas of renormalization are introduced in \S\ref{sec:renorm}, where the behavior of the groups of return vectors under the substitution maps (the return vector cocycle) is analyzed, as well as the induced action on traces. In \S\ref{sec:LFreturn} elements of the LF algebras instroduced in \S\ref{sec:algebras} are constructed to organize the set of inner products $\langle \lambda,\tau\rangle$ of spectral parameters $\lambda\in\mathbb{R}^d$ and return vectors $\tau$. This allows to give bounds of the norms of traces applied to these elements, which are crucial to estimating the growth of twisted integrals, which is done in \S\ref{sec:twisted}. A criterion is proved in \S\ref{sec:criterion} giving bounds of the spectral measures of functions in terms of the quantitative behavior of the return vector cocycle. This type of result goes by the name of a \textbf{quantitative Veech criterion}, refering to the criterion by Veech for weak mixing in \cite{Veech:metric1}, and the criterion here is a generalization to higher rank of the criterion in \cite{BS:translation}. Cohomology is introduced in \S\ref{sec:cohomology} and the basic theory of deformations is introduced. A link between the return vector cocycle and the induced action on cohomology is established from which a cohomological quantitative Veech criterion is derived in \S \ref{sec:CQVC}. In \S\ref{sec:dimension}, Bufetov and Solomyak's version of the Erd\H{o}s-Kahane argument is used to show that the set of deformation parameters which are bad is very small, allowing the proof of Theorem \ref{thm:main} to be finished, which is done in \S\ref{sec:proofs}. Appendix \ref{sec:plotTwist} introduces twisted cohomology for tiling spaces basic properties are derived. The appendix includes results which are not directly related to the proof of quantitative weak mixing, but they are \emph{not} unrelated to the ideas floating around.

      \begin{ack}
        This work was supported by the NSF grant DMS-1665100. I am also very grateful to Alejandro Maass and the Centro de Modelamiento Matem\'atico in Santiago, where some of this work was carried out, for outstanding hospitality during my visit. I also want to thank Giovanni Forni and Boris Solomyak for invaluable discussions and feedback during the writing of this paper.
        Finally, I am very grateful to an anonymous referee for very valuable feedback which lead to a considerable improvement of the paper.
      \end{ack}
      \section{Background: Tilings and Bratteli diagrams}
      \label{sec:background}
      \subsection{Tilings and tiling spaces}
      \label{subsec:tilings}
      A \textbf{tiling} $\mathcal{T}$ of $\mathbb{R}^d$ is a covering of $\mathbb{R}^d$ by compact sets $t$, called \textbf{tiles}, which are the closure of their interior, such that any two tiles of $\mathcal{T}$ may only intersect along their boundaries. Each tile $t$ consists of both its \textbf{support}, $\mathrm{supp}(t)\subset \mathbb{R}^d$, and its type, which can be thought of as a color or label associated to $t$. We will assume here that for any tiling $\mathcal{T}$ all possible labels/colors of tiles in $\mathcal{T}$ come from a finite set. In a slight abuse of notation, we will mostly associated the tiles with their support, and the tiling with the union of the supports of the tiles. For two tiles $t,t'\in\mathcal{T}$, let $t\sim t'$ denote translation-equivalence, that is, if there exists a $\tau\in\mathbb{R}^d$ such that $\varphi_\tau(t) = t'$, then one writes $t\sim t'$.

      The translation $\varphi_\tau(\mathcal{T})$ of a tiling $\mathcal{T}$ is the tiling obtained by translating each of the tiles $t$ of $\mathcal{T}$ by the vector $\tau\in\mathbb{R}^d$. If there exist compact sets $t_1,\dots, t_M$ such that every tile $t\in\mathcal{T}$ is translation-equivalent to some $t_i$ (including its label/color), then we call the collection $\{t_1,\dots, t_M\}$ the \textbf{prototiles} of $\mathcal{T}$. A union $\mathcal{P}\subset \mathcal{T}$ of a finite number of tiles of $\mathcal{T}$ is called a \textbf{patch} of $\mathcal{T}$. A patch is not only considered as the subset of $\mathbb{R}^d$ which it occupies but as a finite collection of tiles and tile types which tile $\mathcal{P}$, and a specific relative position between these tiles. For any bounded set $A\subset\mathbb{R}^d$ define the patch
$$\mathcal{O}^-_\mathcal{T}(A) = \mbox{ largest patch of $\mathcal{T}$ completely contained in $A$}.$$
A tiling is \textbf{repetitive} if for any patch $\mathcal{P}$ in $\mathcal{T}$ there is an $R_\mathcal{P}>0$ such that for any $x\in\mathbb{R}^d$ the ball $B_{R_\mathcal{P}}(x)$ contains a patch $\mathcal{P}'$ which is translation-equivalent to $\mathcal{P}$ for any $x\in\mathbb{R}^d$. A tiling has \textbf{finite local complexity} if for any $R>0$ there is a finite list of patches $\mathcal{P}_1^R,\mathcal{P}_2^R,\dots, \mathcal{P}_{N(R)}^R$ such that for any $x\in\mathbb{R}^d$, $\mathcal{O}^-_\mathcal{T}(B_R(x))\sim \mathcal{P}^R_j$ for some $j$.

Given a finite collection of tiles $t_1,\dots, t_M$, a \textbf{substitution rule} is a collection of affine maps $\{f_{ijk}\}$ such that
$$t_i = \bigcup_{j=1}^M\bigcup_{k=1}^{\kappa(i,j)} f_{ijk}(t_j).$$
Here it will be assumed that the linear part of $f_{ijk}$ is a dilation, so the maps are of the form $f_{ijk}(x) = \theta_{ijk}x + \tau_{ijk}$ for some $\theta_{ijk}\in (0,1]$ and $\tau_{ijk}\in \mathbb{R}^d$. The substitution rule is \textbf{uniformly expanding} if there is a $\theta\in(0,1)$ such that $\theta_{i,j,k}=\theta$ for all $i,j,k$.
  
  There is a metric on the set $\{\varphi_\tau(\mathcal{T})\}_{\tau\in\mathbb{R}^d}$ of translates of $\mathcal{T}$. First, define
  $$\bar{d}(\mathcal{T},\varphi_\tau(\mathcal{T})) = \inf\left\{ \varepsilon>0: \mathcal{O}^-_\mathcal{T}(B_{1/\varepsilon}(0)) = \mathcal{O}^-_{\varphi_{\tau+s}(\mathcal{T})}(B_{1/\varepsilon}(0))\mbox{ for some }\|s\|\leq \varepsilon\right\},$$
  and from this define
  \begin{equation}
    \label{eqn:metric}
    d(\mathcal{T},\varphi_\tau(\mathcal{T})) = \min\left\{\bar{d}(\mathcal{T},\varphi_\tau(\mathcal{T})),2^{-1/2}\right\}.
  \end{equation}
  That (\ref{eqn:metric}) is metric is a standard fact\footnote{See e.g. \cite[\S 2]{LMS:PP} where this is proved for Delone sets, which implies that is holds for tilings through usual constructions of Delone sets from tilings.}. Denote by
  $$\Omega_\mathcal{T} := \overline{\{\varphi_t(\mathcal{T}):t\in\mathbb{R}^d\}}$$
  the metric completion of the set of all translates of $\mathcal{T}$ with respect to the metric (\ref{eqn:metric}). This set is called the \textbf{tiling space} of $\mathcal{T}$. Note that if $\mathcal{T}$ is repetitive then we have that $\Omega_{\mathcal{T}'} = \Omega_{\mathcal{T}}$ for any other $\mathcal{T}'\in\Omega_\mathcal{T}$. By chosing a marked point in the interior of each prototile $t_i$ it can be assumed that every tile in $\mathcal{T}$ has a marking which is translation equivalent to the marking on prototiles by the corresponding translation equivalence of tiles with prototiles. The \textbf{canonical transversal} of $\Omega_\mathcal{T}$ is the set
  \begin{equation}
    \label{eqn:transversal}
    \mho_\mathcal{T} := \{\mathcal{T}'\in\Omega_\mathcal{T}:\mbox{ the marked point in the tile in $\mathcal{T}'$ containing the origin is the origin}\}.
  \end{equation}
  Let $\mathcal{P}\subset \mathcal{T}$ be a patch containing the origin and $t\in\mathcal{P}$ a tile in the patch which contains the origin, coinciding with the marked point inside of $t$. For such a patch, define
  $$\mathcal{C}_{\mathcal{P}} := \{\mathcal{T}'\in\Omega_\mathcal{T}: \mbox{ $\mathcal{P}$ is a patch in $\mathcal{T}'$ } \},$$
  and note that the marked point in the tile in $\mathcal{P}$ which contains the origin is the origin, we have $\mathcal{C}_\mathcal{P}\subset \mho_\mathcal{T}$. This is the $\mathcal{P}$-cylinder set, and for $\varepsilon>0$ the $(\mathcal{P},\varepsilon)$-cylinder set is
  $$\mathcal{C}_{\mathcal{P},\varepsilon} := \bigcup_{\|\tau\|<\varepsilon}\left\{ \varphi_\tau(\mathcal{T'}):\mathcal{T}'\in  \mathcal{C}_{\mathcal{P}}\right\}\subset \Omega_\mathcal{T}.$$
  The tiling space is compact if $\mathcal{T}$ has finite local complexity. Moreover, if $\mathcal{T}$ is a repetitive tiling of finite local complexity then the topology of $\Omega_\mathcal{T}$ is generated by cylinder sets of the form $\mathcal{C}_{\mathcal{P},\varepsilon}$ for arbitrarily small $\varepsilon>0$. This gives $\Omega_\mathcal{T}$ a the local product structure of $\mathcal{C}\times B_\varepsilon$, where $\mathcal{C}$ is a Cantor set and $B_\varepsilon\subset \mathbb{R}^d$ is an open ball \cite{sadun:book}.
  
  There is a natural action of $\mathbb{R}^d$ on $\Omega_\mathcal{T}$ by translation and this action is minimal if $\mathcal{T}$ is repetitive. The only case considered here is the case when this action is uniquely ergodic, since it is a feature of the constructions in \S \ref{sec:RandSub} that the typical tiling we get defines a uniquely ergodic action. That is, there exists a unique Borel probability measure $\mu$ which is invariant under the action of $\mathbb{R}^d$. By the local product structure $\mathcal{C}\times B_\varepsilon$ of $\Omega_\mathcal{T}$, the  measure $\mu$ has a natural product structure $\nu\times \mbox{Leb}$, where $\nu$ is a measure on a transversal to the action of $\mathbb{R}^d$ in the sense of \cite{BM:UE}. Since open sets along the totally disconnected transversal $\mho_\mathcal{T}$ are given by cylinder sets of patches, the measure $\nu$ satisfies $\nu(\mathcal{C}_{\mathcal{P}}) = \mbox{freq}(\mathcal{P})$, where $\mbox{freq}(\mathcal{P})$ is the asymptotic frequency of translation-equivalent copies of the patch $\mathcal{P}$ in the tiling $\mathcal{T}$.
  \subsubsection{Functions}
  \label{subsubsec:functions}
  \begin{definition}
    Let $\mathcal{T}$ be a tiling of $\mathbb{R}^d$. A function $h:\Omega_\mathcal{T} \rightarrow \mathbb{R}$ is \textbf{transversally locally constant} (TLC) if for any $\mathcal{T}$ there is an $R>0$ such that if $\mathcal{O}^-_{\mathcal{T}'}(B_R(0)) = \mathcal{O}^-_{\mathcal{T}}(B_R(0))$ for any $\mathcal{T}'\in\Omega_\mathcal{T}$ then $h(\mathcal{T}) = h(\mathcal{T}')$.
  \end{definition}
  \begin{definition}
    Let $\mathcal{T}$ be a tiling of $\mathbb{R}^d$. A function $f:\mathbb{R}^d \rightarrow \mathbb{R}$ is \textbf{$\mathcal{T}$-equivariant} if there is an $R>0$ such that if $\mathcal{O}^-_\mathcal{T}(B_R(x)) = \mathcal{O}^-_\mathcal{T}(B_R(y))$ then it follows that $f(x) = f(y)$.
  \end{definition}
  There is a correspondence between TLC functions on $\Omega_\mathcal{T}$ and $\mathcal{T}'$-equivariant functions on $\mathbb{R}^d$. Given a TLC function on $\Omega_\mathcal{T}$ and $\mathcal{T}'\in\Omega_\mathcal{T}$, let $i_{\mathcal{T}'}(h):\mathbb{R}^d\rightarrow \mathbb{R}$ be the function
  \begin{equation}
    \label{eqn:imap}
    i_{\mathcal{T}'}(h)(\tau) = h\circ \varphi_\tau(\mathcal{T}')
  \end{equation}
  for all $\tau\in\mathbb{R}^d$. This map in fact gives an algebra isomorphism between TLC functions and $\mathcal{T}$-equivariant functions \cite[Theorem 20]{KP:RS}.
  \subsubsection{Return vectors}
  \label{subsubsec:retVectors}
  Let $\mathcal{T}$ be a repetitive tiling of finite local complexity. Define
  $$\Lambda_i^\mathcal{T} = \left\{\tau\in \mathbb{R}^d: \varphi_\tau(t_1) = t_2,\;\;\;\; \mbox{ for some }\;\;\;\; t_1\sim t_2\sim t_i \right\}$$
  for $i = 1,\dots, M$, and
  $$\Lambda_\mathcal{T} = \bigsqcup_i \Lambda_i^\mathcal{T}.$$
  The set $\Lambda_i^\mathcal{T}$ is the set of return vectors to a transversal between two tiles which is of the same translation equivalence type. Note that by repetitivity this set is independent of the tiling in the tiling space and, as such, it is defined by the tiling space. 
  
    Let $\Gamma_\mathcal{T}$ be the group generated by the Delone set $\Lambda_\mathcal{T}$ (sometimes called the \textbf{Lagarias group}). This is defined as the group generated by linear combination of integer multiples of elements of $\Lambda_\mathcal{T}$:
    $$\Gamma_\mathcal{T} := \mathbb{Z}[\Lambda_\mathcal{T}] = \bigsqcup_i \mathbb{Z}[\Lambda_i^\mathcal{T}].$$
    By finite local complexity, this is a finitely generated free Abelian group \cite[Theorem 2.1]{Lagarias:GeomMod1}. The group $\Gamma_\mathcal{T}$ is generated by the union of generators of each $\mathbb{Z}[\Lambda_i]$. Let $r$ be the rank of this group and choose a basis $\{v_i,\dots, v_r\}$ for $\Gamma_\mathcal{T}$. This choice defines the \textbf{address map} $\alpha:\Gamma_\mathcal{T}\rightarrow \mathbb{Z}^r$ by $\alpha(v_i)=e_i$ for all $i= 1,\dots, r$ and extending in a way to yield an isomorphism. By a result of Lagarias \cite{Lagarias:GeomMod1} this map is Lipschitz when restricted to $\Lambda_\mathcal{T}$.

    Note that although one could have that $\Lambda_i = \Lambda_j$ for all $i,j$, the group $\Gamma_\mathcal{T}$ distinguishes between the groups $\mathbb{Z}[\Lambda_i]$ and $\mathbb{Z}[\Lambda_j]$ for all $i\neq j$. For example, this could happen when $\Lambda_i = \mathbb{Z}^d$ for all $i$, in which case one would have that $\Gamma_\mathcal{T}\cong \mathbb{Z}^{dM}$.


      \subsection{Bratteli Diagrams}
      \label{subsec:brat}
A \textbf{Bratteli diagram} is an infinite directed graph $B= (V,E)$ with
$$V = \bigsqcup_{k\geq 0} V_k\,\hspace{1in}\mbox{ and }\hspace{1in}\, E = \bigsqcup_{k>0}E_k$$
and range and source maps $r,s:E\rightarrow V$ with $r(E_k) = V_k$ and $s(E_{k}) = V_{k-1}$ for $k\in\mathbb{N}$. It is assumed that $|E_k|$ and $|V_k|$ are finite for every $k$.  A Bratteli diagram $B$ can also be defined by a sequence of matrices $\{A_k\}$, where $A_k$ is a $|V_{k}|\times |V_{k-1}|$ matrix and $(A_k)_{ij}$ is the number of edges $e$ with $s(e) = v_j\in V_{k-1}$ and $r(e) = v_i\in V_{k}$. 

      A finite path $\bar{e}$ is a finite collection of edges $(e_m,\dots, e_{m+n})$ with the property that $e_i\in E_i$ and $r(e_i) = s(e_{i+1})$  for $i=m, \dots, m+n-1$. The range and source maps $r,s$ can be extended to the set of finite paths by setting $s((e_m,\dots, e_{m+n}) ) = s(e_m)$ and $r((e_m,\dots, e_{m+n}) ) = r(e_{m+n})$. The set $E_{m,n}$ is then the set of finite paths $\bar{e}$ such that $s(\bar{e})\in V_m$ and $r(\bar{e})\in V_n$. For two vertices $v\in V_m$ and $w\in V_n$ with $m<n$, $E_{vw}$ is the set of finite paths $\bar{e}$ such that $s(\bar{e})=v$ and $r(\bar{e})=w$. A Bratteli diagram is \textbf{minimal} is for every $k\geq 0$ there is a $\ell\in\mathbb{N}$ such that $E_{vw}\neq\varnothing$ for every $v\in V_k$ and $w\in V_{k+\ell}$.

      An infinite path $\bar{e}$ is a path of infinite length with $s(\bar{e})\in V_0$. The set $X_B$ is a set of all infinite paths, that is, 
$$X_B:=\left\{ \bar{e} = (e_1,e_2,\dots)\in \prod_{k\in\mathbb{N}}E_k:r(e_i) = s(e_{i+1})\mbox{ for all }i>0 \right\},$$
      and it is topologized by a basis of cylinder sets of the form
      $$C_p = \{\bar{e}' = (e_1',e_2',e_3', \dots)\in X_B: e_i' = e_i \mbox{ for all }i\leq k\}$$
      for a finite path $p = (e_1,\dots, e_k)\in E_{0,k}$.

      Two paths $e,e'\in X_B$ are \textbf{tail-equivalent} (denoted by $e\sim e'$) if there exists a $k\geq 0$ such that $e_i = e_i'$ for all $i\geq k$. This relation is an equivalence relation and the tail-equivalence class of a path $\bar{e}$ is denoted by $[\bar{e}]$. Note that if a Bratteli diagram $B$ is minimal then $\overline{[\bar{e}]} = X_B$ for any $\bar{e}\in X_B $. A Borel measure $\mu$ on $X_B$ is \textbf{invariant} (or invariant for the tail-equivalence relation) if for any two finite paths $p,q\in E_{0,k}$ with $r(p) = r(q)$ it holds that $\mu(C_{p})=\mu(C_q)$. For such measures, for $v\in V_k$, denote $\mu(v) :=  \mu(C_p)$ for any $p\in E_{0,k}$ with $r(p)=v$.
      \subsubsection{Bi-infinite diagrams}
      \label{subsubsec:biinf}
      The notion of a Bratteli diagram is easily generalized to a bi-infinite graph.
      \begin{definition}
A \textbf{bi-infinite Bratteli diagram} is an infinite directed graph $\mathcal{B}= (\mathcal{V},\mathcal{E})$ with
$$\mathcal{V} = \bigsqcup_{k\in\mathbb{Z}} \mathcal{V}_k\,\hspace{1in}\mbox{ and }\hspace{1in}\, \mathcal{E} = \bigsqcup_{k\in\bar{\mathbb{Z}}}\mathcal{E}_k,$$
where $\bar{\mathbb{Z}} = \mathbb{Z}-\{0\}$, along with range and source maps $r,s:\mathcal{E}\rightarrow \mathcal{V}$ with $r(\mathcal{E}_k) = \mathcal{V}_k$ and $s(\mathcal{E}_{k}) = \mathcal{V}_{k-1}$ for $k\in\mathbb{N}$, and $r(\mathcal{E}_k) = \mathcal{V}_{k+1}$ and $s(\mathcal{E}_{k}) = \mathcal{V}_{k}$ for $k<0$.  It is assumed that $|\mathcal{E}_k|$ and $|\mathcal{V}_k|$ are finite for every $k$.
      \end{definition}
      The notational convention adopted here is that $B,V,E, A$ will denote objects corresponding to Bratteli diagrams as in \S \ref{subsec:brat}, whereas $\mathcal{B},\mathcal{V},\mathcal{E},\mathcal{A}$ are the objects corresponding to bi-infinite diagrams as defined above. The sets of paths $\mathcal{E}_{m,n}$ are analogously defined as in $E_{m,n}$ in the obvious way, and so is the concept of minimality for a bi-infinite diagram $\mathcal{B}$.

      A bi-infinite diagram $\mathcal{B}$ can be thought of a two one-sided diagrams $B^+$ and $B^-$ as follows. The \textbf{positive part} $\mathcal{B}^+$ of the Bratteli diagram $\mathcal{B}$ is the restriction of $\mathcal{B}$ to its vertices and edges with non-negative indices. The \textbf{negative part} $\mathcal{B}^-$ of $\mathcal{B}$ is the one-sided Bratteli diagram obtained by considering the vertices and edges of $\mathcal{B}$ with non-positive indices and multiplying the indices by $-1$ to obtain a one-sided diagram.  Equivalently, if $\{\mathcal{A}_k\}_{k\in\bar{\mathbb{Z}}}$ is the sequence of matrices defining $\mathcal{B}$, then $\mathcal{B}^+$ is defined by the matrices $A^+_k = \mathcal{A}_k$ for $k>0$, and $\mathcal{B}^-$ is defined by the matrices $A^-_k = \mathcal{A}^T_{-k}$ for $k>0$. This convention makes $A_{1}^- = \mathcal{A}_{-1}^T$ a $ |\mathcal{V}_{-1}|\times |\mathcal{V}_{0}|$ matrix.
      Note that $\mathcal{B}^+$ and $\mathcal{B}^-$ share the same level of vertices at level 0. As such, the space of all infinite paths $X_{\mathcal{B}}$ on $\mathcal{B}$ has a local product structure given by the identification
      $$X_{\mathcal{B}} = \left\{(\bar{e}^+,\bar{e}^-)\in X_{\mathcal{B}^+}\times X_{\mathcal{B}^-}: s(\bar{e}^-) = s(\bar{e}^+)\right\}. $$
      \section{Random substitution tilings}
      \label{sec:RandSub}
      Let $S_1,\dots, S_N$ be $N$ substitution rules on the prototiles $t_1,\dots, t_M$, where it is assumed without loss of generality that each prototile contains the origin in the interior. Recall that each substitution rule $S_\ell$ gives a collection of affine maps (compositions of dilations and translations) $\{f^\ell_{ijk}\}$ with the property that
      \begin{equation}
        \label{eqn:subsRule}
        t_i = \bigcup_{j=1}^M\bigcup_{k=1}^{\kappa_\ell(i,j)} f^\ell_{ijk}(t_j).
      \end{equation}
      Let $\Sigma_N = \{1,\dots, N\}^{\bar{\mathbb{Z}}}$ be the full $N$-shift, where $\bar{\mathbb{Z}} = \mathbb{Z}\backslash \{0\}$ and it inherits the natural order from $\mathbb{Z}$. For $x\in\Sigma_N$, define the Bratteli diagram $\mathcal{B}_x$ to be the diagram with $M$ vertices at every level, and, for $n>0$, having $\kappa_{x_n}(i,j)$ edges going from $v_j\in \mathcal{V}_{n-1}$ to $v_i\in\mathcal{V}_n$ for all $i,j = 1,\dots, M$. For $n<0$, having $\kappa_{x_n}(i,j)$ edges going from $v_j\in \mathcal{V}_{n}$ to $v_i\in\mathcal{V}_{n+1}$ for all $i,j = 1,\dots, M$. As such, since there is a natural bijection at level $n$ between edges and maps $\{f^{x_n}_{ijk}\}$, we denote by $f_e$ the corresponding affine linear map corresponding to the edge $e\in \mathcal{E}_n$. For a finite path $\bar{e} = (e_m,\dots, e_n)$ define $f_{\bar{e}} = f_{e_n}\circ \cdots\circ f_{e_m}$ and $f_{\bar{e}}^{-1} = f_{e_m}^{-1}\circ \cdots \circ f_{e_n}^{-1}$. Finally, for an infinite path $(\dots, e_{-1},e_1,e_2,\dots)=(\bar{e}^-,\bar{e}^+)=\bar{e} \in X_{\mathcal{B}_x}$ denote by $ \bar{e}|_k= \bar{e}^+|_k$ the truncation of $\bar{e}^+$ after level $k$. Thus $f_{\bar{e}|_k} = f_{\bar{e}^+|_k} = f_{e_k}\circ \cdots \circ f_{e_1}$.

      Following the constructions in \cite{ST:random, T:TTT}, one can build a tiling $\mathcal{T}_{\bar{e}}$ from a path $\bar{e} = (\bar{e}^-,\bar{e}^+)\in X_{\mathcal{B}_x}$ as follows. The $k^{th}$-approximant of $\mathcal{T}_{\bar{e}^+}$ is the set
      \begin{equation}
        \label{eqn:approximant}
        \mathcal{P}_k(\bar{e}^+) = \bigcup_{\substack{\bar{e}'\in \mathcal{E}_{0,k}\\ r(\bar{e}') = r(\bar{e}|_k)}}f_{\bar{e}|_k}^{-1}\circ f_{\bar{e}'}(t_{s(\bar{e}')}).
      \end{equation}
      It can be quickly verified that $\mathcal{P}_k(\bar{e}^+)\subset \mathcal{P}_{k+1}(\bar{e}^+)$ as a subpatch, and so by considering larger and larger indices one obtains
      $$\mathcal{T}_{\bar{e}}^+ = \bigcup_{k>0}\mathcal{P}_k(\bar{e}^+).$$
      This tiling is tiled by the tiles $f_{\bar{e}|_k}^{-1}\circ f_{\bar{e}'}(t_{s(\bar{e}')})$ for $\bar{e}'\in \mathcal{E}_{0,k}$ with $r(\bar{e}') = r(\bar{e}|_k)$ and arbitrary $k$. This tiling may or may not cover all of $\mathbb{R}^d$, but this potential problem will addressed below in Proposition \ref{prop:sum}. Let
      $$\tau_{\bar{e}}^- = \bigcap_{k<0} f_{(\bar{e}_k,\dots, \bar{e}_{-1})}(t_{s(e_{k})})$$
      and assume for the moment that this intersection is a single point (this point can be in the boundary of $t_{r(e_{-1})}$). Then the tiling
      $$\mathcal{T}_{\bar{e}} = \varphi_{\tau_{\bar{e}}^-}(\mathcal{T}^+_{\bar{e}})$$
      is the tiling constructed from $\bar{e}$ as long as 1) $\mathcal{T}^+_{\bar{e}}$ covers all of $\mathbb{R}^d$ and 2) $\tau_{\bar{e}}$ is a single point. As will be discussed below in Proposition \ref{prop:sum}, both of these conditions will generically hold true. Note that if one defines
      $$\mathcal{P}_k(\bar{e}) = \varphi_{\tau^-_{\bar{e}}}(\mathcal{P}_k(\bar{e}^+))$$
      then $\mathcal{P}_k(\bar{e})\subset \mathcal{P}_{k+1}(\bar{e})$ for all $k>0$ and
      $$\mathcal{T}_{\bar{e}} = \bigcup_k \mathcal{P}_k(\bar{e}).$$
      If $\tau_{\bar{e}}^-$ lies on the boundary of $t_{r(e_{-1})}$, then the origin will fall on the boundary of tiles in $\mathcal{T}_{\bar{e}}$.
      
      Suppose $\mathcal{B}_x$ is minimal and let $\varnothing\neq \mathring{X}_{\mathcal{B}_x}\subset X_{\mathcal{B}_x}$ be the set of paths $\bar{e}$ such that $\mathcal{T}_{\bar{e}}$ covers all of $\mathbb{R}^d$. Then the construction $\bar{e}\mapsto \mathcal{T}_{\bar{e}}$ defines a surjective map $\Delta_x:\mathring{X}_{\mathcal{B}_x}\rightarrow \Omega_{\mathcal{T}_{\bar{e}}} = \Omega_x$ by \cite[Lemma 5]{ST:random}, called the \textbf{Robinson map}. That is, the tiling space $\Omega_{\mathcal{T}_{\bar{e}}}$ obtained is independent of the path $\bar{e}$ used. A \textbf{level-$k$ supertile} of a tiling $\mathcal{T}\in\Omega_x$ is a patch $\mathcal{P}$ of $\mathcal{T}$ which is translation equivalent to a level-$k$ approximant $\mathcal{P}_k(\bar{e}^+)$ in (\ref{eqn:approximant}) for some $\bar{e}^+\in X_{\mathcal{B}_x^+}$ not only as a set, but as a tiled patch.
            
      Since each of the substitution maps is of the form, $f_{ijk}^\ell(x) = \theta_{ijk\ell}x+\tau_{ijkl}$ for $\theta_{ijk\ell}\in(0,1]$ (see (\ref{eqn:subsRule})), we can assign to any edge $e$ in $\mathcal{B}_x$ one of the maps $f_{ijk}^\ell$. More specifically, for $e\in \mathcal{E}_n$ then $f_e\in\{f^{x_n}_{ijk}\}$ is determined by the natural bijection of edges in $\mathcal{E}_n$ with maps in $\{f^{x_n}_{ijk}\}$ defined by $S_{x_n}$ in (\ref{eqn:subsRule}). As such, we can write $f_e = \theta_e x+\tau_e$ for any $e\in\mathcal{E}$ without ambiguity. Define
        $$S^\pm_x(\bar{e}) = \lim_{n\rightarrow \pm \infty} \frac{1}{|n|}\sum_{\ell=1}^n\log \theta_{e_\ell}.$$
        If the collection of substitutions are uniformly expanding, then the contraction constants $\theta_{ijk\ell}$ only depend on $\ell\in\{1,\dots, N\}$, which is determined by which of the substitutions rules $S_1,\dots, S_N$ is used (see (\ref{eqn:subsRule})), and so
        $$S^\pm_x(\bar{e}) = \lim_{n\rightarrow \pm \infty} \frac{\log \theta_{(n)_x}}{|n|},$$
        where $\theta_{(\pm n)_x}:= (\theta_{x_{\pm1}}\cdots \theta_{x_{\pm n}})$. If $X_{\mathcal{B}_x}^\pm$ admits only one tail-invariant probability measure, then $S^\pm_x:X^\pm_{\mathcal{B}_x}\rightarrow \mathbb{R}$ is constant. Moreover, we have that
        $$S^\pm_x(\bar{e}) = S^\pm_{\sigma(x)}(\sigma(\bar{e})),$$
        i.e., the functions $S^\pm_x$ are invariant under a shift of indices of the Bratteli diagram $\mathcal{B}_x$. As such, if $\mu$ is an $\sigma$-invariant ergodic probability measure on $\Sigma_N$ then $S^\pm_{x}$ are constant along the orbits.
        \begin{definition}
          An invariant probability measure $\mu$ is \textbf{contracting} if the $\mu$-almost every $x$ we have that $S^\pm_x(\bar{e})<0$ for all $\bar{e}\in X^\pm_{\mathcal{B}_x}$.
        \end{definition}
        Thus if the collection of substitution rules are uniformly contracting any invariant measure is contracting. The following comes from \cite{GM:mixed}.
        \begin{definition}
          A set of subtitution rules $S_1,\dots, S_N$ on the prototiles $t_1,\dots, t_M$ is \textbf{compatible} if each prototile $t_i$ is a CW complex such that for any $x\in\Sigma_N$, $(\bar{e}^-,\bar{e}^+)\in  X_{\mathcal{B}_x}$ and any $k\in\mathbb{N}$, the $k^{th}$ approximant $\mathcal{P}_k(\bar{e}^+)$ in (\ref{eqn:approximant}) is tiled by copies of prototiles such any intersections of two tiles is a union of subcells.
        \end{definition}
        Compatibility in dimension 1 is automatic. In higher dimension is harder to enforce but, as shown in \cite[\S 4]{GM:mixed}, there are still plenty of interesting examples in higher dimensions. The following is a collection of results from \cite{ST:random} (see \cite[Propositions 1,2,3]{ST:random}).
        \begin{proposition}
          \label{prop:sum}
          Let $S_1,\dots, S_N$ be a collection of $N$ compatible substitution rules on the prototiles $t_1,\dots, t_M$ and let $\mu$ be a minimal, contracting, $\sigma$-invariant ergodic probability measure. Then for $\mu$-almost every $x$ there are unique invariant mesures  $\mu^\pm_x$ on $X^\pm_{\mathcal{B}_x}$ with $\mu_x = \mu^+_x\times \mu^-_x$ a measure on $X_{\mathcal{B}_x}$ with $\mu_x(\mathring{X}_{\mathcal{B}_x}) = 1$. The map $\Delta_x$ sends $\mu_x$ to the unique $\mathbb{R}^d$-invariant Borel probability measure on $\Omega_x$, locally given as $(\Delta_x)_*\mu_x = (\Delta_x)_*(\mu_x^+\times \mu_x^-) = \nu_x\times\mbox{Leb}_x$, where $\nu_x$ is the unique (up to scaling) $\mathbb{R}^d$-invariant measure $\mho_x\subset \Omega_x$.
        \end{proposition}
        \subsection{Functions}
        \begin{definition} 
          Let $S_1,\dots, S_N$ be a family of compatible substitution rules on the prototiles $t_1,\dots, t_M$, $x\in \Sigma_N$ and $\mathcal{B}_x$ the Brattelli diagram constructed from $x$ and the family of substitutions. Suppose that $\mathcal{B}_x$ is minimal and $S^\pm_x(\bar{e})<0$ for all $\bar{e}\in X_{\mathcal{B}_x}$. If $\Delta_{\bar{e}}:\mathring{X}_{\mathcal{B}_x}\rightarrow\Omega_x $ is the surjective Robinson map, then $h:\Omega_x\rightarrow \mathbb{R}$ is a \textbf{level-}$k$ transversally locally constant (TLC) function if there exist compactly supported functions $\psi_i:t_i\rightarrow \mathbb{R}$, $i=1,\dots, M$, such that
          $$h(\Delta_x(\bar{e})) =   \psi_i\left( \bigcap_{j<k} f_{(\bar{e}_j,\dots, \bar{e}_{k-1})}(t_{s(\bar{e}_{j})})\right)$$
        if $r(\bar{e}_k) = v_i\in\mathcal{V}_k$.
        \end{definition}
        In words: if $\mathcal{T}'\in\Omega_x$ then there is a path $\bar{e}' = (\dots, e_{-1},e_1,e_2,\dots)$ with $\Delta_x(\bar{e}') = \mathcal{T}'$. Now, the level-$k$ supertile associated to $\mathcal{T}'$ is the patch associated to the level-$k$ approximant $\mathcal{P}_k(\bar{e}')$ which contains the origin, and this approximant, as a set, is a comformally-equivalent to some prototile $t_i$. Then $h(\mathcal{T}')$ is the value of $\psi_i$ at the point in $t_i$ corresponding to where the origin sits in $\mathcal{P}_k(\bar{e}')$. That the intersection is a single point is a consequence of the assumption $S^\pm_x(\bar{e})<0$.

        Having the parameter $k$ to describe TLC functions allows one to organize the set of all TLC functions as introduced in \S\ref{subsubsec:functions}. The same idea motivates the following definition, which is a modification to the current setting of the definition in \S\ref{subsubsec:functions}.
        \begin{definition}
          Let $\mathcal{T}\in\Omega_x$ be an aperiodic, repetitive tiling of finite local complexity. A function $g:\mathbb{R}^d\rightarrow \mathbb{R}$ is a \textbf{level-$k$ $\mathcal{T}$-equivariant function} if the value of $g$ at $x$ only depends on the level-$k$ supertile of $\mathcal{T}$ which contains $x$. More specifically, there is a level-$k$ TLC function $h:\Omega_\mathcal{T}\rightarrow \mathbb{R}$ such that $g(x) = h\circ \varphi_x(\mathcal{T})$.
        \end{definition}
        \begin{lemma}
          \label{lem:LipApprox0}
          Let $S_1,\dots, S_N$ be a family of uniformly expanding compatible substitution rules on the prototiles $t_1,\dots, t_M$ and let $x\in\Sigma$ so that $\mathcal{B}_x$ is minimal and $\Omega_x$ has a uniquely ergodic $\mathbb{R}^d$ action. Let $f:\Omega_x\rightarrow \mathbb{R}$ be a Lipschitz function with Lipschitz constant $L_f$. Then for every $k\in\mathbb{N}$ there exists a level-$k$ TLC function $f_k$ such that
          $$|f(\mathcal{T})-f_k(\mathcal{T})|\leq L_f \max_{v\in\mathcal{V}_k}\mu_x^+(v)$$
          on a full measure set of $\mathcal{T}$.
      \end{lemma}
      \begin{proof}
        The tiling space $\Omega_x$ is partitioned (up to a set of measure zero) into sets $\Omega_x^{k,1},\dots, \Omega_x^{k,M}$ with the property that any $\mathcal{T}\in\Omega_x^{k,i}$ has as its level-$k$ supertile containing the origin a supertile of type $i$. Pick any  $\mathcal{T}_i\in \Omega_{x}^{k,i}$ for $i=1,\dots, M$. Define $f_k$ as follows. If $\mathcal{T}\in\Omega_x^{k,i}$, then it is at a distance less than or equal to $\mu_x^+(v_i)$ from a translate $\mathcal{T}'\in \Omega_{x}^{k,i}$ of $\mathcal{T}_i$. Define in this case $f_k(\mathcal{T}) = f(\mathcal{T}')$. The approximation estimate follows.
      \end{proof}
      
      \subsection{Choices, canonical supertiles and control points}
      \label{subsec:choices}
      \begin{definition}
        \label{def:choice}
        Let $S_1,\dots, S_N$ be a collection of compatible substitution rules on $M$ prototiles, and $x\in \Sigma_N$. By the construction in this section we obtain a (bi-infinite) Bratteli diagram $\mathcal{B}_x=(\mathcal{V},\mathcal{E})$. A \textbf{choice function} is any map $c:\mathcal{V}\rightarrow \mathcal{E}$ that satifies
        \begin{enumerate}
        \item $c(v)\in r^{-1}(v)$ for all $v\in\mathcal{V}$,
        \item if $x_k = x_{k'}$ then $c(v) = c(v')$ for $v\in \mathcal{V}_{k}$ and $v'\in \mathcal{V}_{k'}$ under the obvious bijection between the sets of edges $\mathcal{E}_k$ and $\mathcal{E}_{k'}$.
        \end{enumerate}
        Given a choice function $c$, $k>0$ and $v\in\mathcal{V}_k$, the \textbf{choice path} $\bar{e}_{c,v}\in\mathcal{E}_{0,k}$ is the unique path $(e_1,\dots, e_k)\in\mathcal{E}_{0,k}$ satisfying $c(v) = e_k$, and $e_{i} = c(s(e_{i+1}))$ for all $i = 1,\dots, k-1$.
      \end{definition}
      What choice functions give is a canonical choice of tile in any supertile. More specifically, by (\ref{eqn:approximant}), any level-$k$ supertile is associated to a vertex $v\in \mathcal{V}_k$, and the super tile is obtained through the union of a collection of tiles in bijection with the set of paths $\bar{e}\in\mathcal{E}_{0,k}$ with $r(\bar{e}) = v$. Thus, since a choice function $c$ picks an element $\bar{e}_{c,v}\in\mathcal{E}_{0,k}$, then this corresponds to a choice of tile in the level-$k$ supertile.

\subsubsection{Control points}
\label{subsubsec:controlPts}
Let $S_1,\dots, S_N$ be $N$ compatible substitution rules on the prototiles $t_1,\dots, t_M$. As discussed in \S \ref{sec:RandSub}, a choice of point in the interior of each of the prototiles defines a canonical transversal (see (\ref{eqn:transversal})). This choice is arbitrary, but if $x\in\Sigma_N$ with $\mathcal{B}_x$ minimal, we can make a choice which will prove to be advantageous. According to \cite{SolomyakEigenfunctions}, this idea goes back to Thurston.

      Given a choice function $c$ and $v\in\mathcal{V}_0$, let $\bar{e}_{c,v}^{-\infty} = (\dots, e_{-2}^{c,-\infty},e_{-1}^{c,-\infty}) = \in\mathcal{E}_{-\infty,0}$ be the unique path $(\dots, e_{-k},\dots, e_{-1})\in\mathcal{E}_{-\infty,0}$ satisfying $e_{i-1} = c(s(e_{i}))$ for all $i<0$ and $r(e_{-1}) = v$. In other words, this path starts at $v$ and traces back a path by going along the edges given by the choice function $c$.
      \begin{definition}
        Let $S_1,\dots, S_N$ be $N$ compatible substitution rules on the prototiles $t_1,\dots, t_M$, $x\in\Sigma_N$ such that $\mathcal{B}_x$ is minimal and $c$ a choice function. The \textbf{control points} on the prototile $t_i$ corresponding to the vertex $v_i\in\mathcal{V}_0$ is given by
\begin{equation}
\label{eqn:controlPt}
\bigcap_{k<0} f_{(\bar{e}_k^{c,-\infty},\dots,e_{-2}^{c,-\infty},e_{-1}^{c,-\infty})}(t_{s(e_{k})})\in t_i
\end{equation}
The collection of the $M$ control points gives a marking in teach tile $t\in\mathcal{T}\in\Omega_x$ for any $\mathcal{T}\in\Omega_x$.
\end{definition}
      Let $\mathcal{C}_\mathcal{T}\subset\mathbb{R}^d$ be the set of points given by the markings on every tile of $\mathcal{T}$. Then it is immediate to see that
      \begin{equation}
      \label{eqn:CtrlPt1}
      \Lambda_\mathcal{T}\subset \mathcal{C}_\mathcal{T}-\mathcal{C}_\mathcal{T}.
      \end{equation}
      Let $X_{c}^x\subset X_{\mathcal{B}_x}$ be the set of \textbf{maximal choice paths}. This is any path $\bar{e} = (\dots, e_{-1},e_1,e_2,\dots)\in X_{\mathcal{B}_x}$ such that $e_i = c(r(e_i))$ for all $i$. By compactness of $\mathcal{B}_x^+$, $X_{c}^x\neq \varnothing$. Note that by construction one has that $\Delta_x(X_{c}^x)\subset\mho_x$. For $v_i\in\mathcal{V}_k$, $k\geq 0$, the choice function $c$ gives a unique choice path $\bar{e}_{c,v_i}\in\mathcal{E}_{-\infty,k}$ with $r(\bar{e}_{c,v_i}) = v_i$ and $e_i = c(r(e_i))$ for all $i$, and thus a unique cylinder set $C_{c,v_i} := C_{\bar{e}_{c,v_i} }$.

      \begin{definition}
        \label{def:canonical}
        For $k>0$, $i\in \{1,\dots, M\}$ and $v_i\in\mathcal{V}_k$, the $M$ \textbf{canonical level-$k$ supertiles}
        $$\mathcal{P}_k(v_i) := \mathcal{P}_k(\bar{e}),$$
        each of which is independent of the choice $\bar{e}\in C_{c,v_i}$. Level-$k$ supertiles which are canonical will also be supertiles \textbf{in canonical position}.
      \end{definition}
      By construction, the tile which contains the origin in a canonical level-$k$ supertile has it coincide with its control/marked point. Thus, the control point of a canonical level-$k$ supertile is the origin.
      \section{Algebras and traces}
      \label{sec:algebras}
      \subsection{Algebras}
      \label{subsec:algebras}
      In this subsection we review the basics of AF algebras. Since we are working with bi-infinite Bratteli diagrams $\mathcal{B}$, each such diagrams will define two algebras, one for the positive part and one for the negative part. The reader is refered to \cite[\S III.2]{davidson:book} for more details on and examples of AF algebras.
      
      A \textbf{multi-matrix algebra} $\mathcal{A}$ is a $*$-algebra of the form
      $$\mathcal{A} = M_{n_1}\oplus \cdots \oplus M_{n_k},$$
      where $n_1,\dots, n_k\in\mathbb{N}$ and $M_\ell$ is the algebra of $\ell\times \ell$ matrices over $\mathbb{C}$. Any finite-dimensional C$^*$-algebra is isomorphic to a multimatrix algebra \cite[\S III]{davidson:book}. Let $\mathcal{M} = M_{\ell_{1}}\oplus \cdots \oplus M_{\ell_{n}}$ and $\mathcal{M}' = M_{\ell_{1}'}\oplus \cdots \oplus M_{\ell_{n'}'}$ be multi-matrix algebras and suppose $\phi:\mathcal{M}\rightarrow \mathcal{M}'$ is a unital homomorphism of $\mathcal{M}$ into $\mathcal{M}'$. Then $\phi$ is determined up to unitary equivalence in $\mathcal{M}'$ by a $\ell_{n'}'\times \ell_{n}$ non-negative integer matrix $A_\phi$ \cite[\S III.2]{davidson:book}. It follows that the inclusion of a multi-matrix algebra $\mathcal{M}$ into another multimatrix algebra $\mathcal{M}'$ is determined up to unitary equivalence by a matrix $A$ which roughly states how many copies of a particular subalgebra of $\mathcal{M}$ goes into a particular subalgebra of $\mathcal{M}'$.

      Using these facts, we will construct from a (bi-infinite) Bratteli diagram two $*$-algebras $LF(\mathcal{B}^\pm)$, whose completions are $C$*-algebras. Let $\mathcal{B}$ be a Bratteli diagram and let $A_k^+$, $k\in\mathbb{N}$, be the connectivity matrix at level $k$. In other words, $A_k(i,j)^+$ is the number of edges going from $v_j\in \mathcal{V}_{k-1}$ to $v_i\in\mathcal{V}_k$. An analogous matrix $A^-_k$ can be defined for $k<0$. Starting with $\mathcal{M}_0 = \mathbb{C}^{|\mathcal{V}_0|}$ and using the facts from the preceding paragraph, the matrices $A_k^\pm$ define two families of inclusions $i_{|k|}^\pm:\mathcal{M}_{|k|-1}^\pm\rightarrow \mathcal{M}_{|k|}^\pm$ (up to unitary equivalence), one for $+$ and one for $-$, where each $\mathcal{M}_k^\pm$ is a multimatrix algebra. More explicitly, starting with the vector $h^0=(1,\dots, 1)^T\in \mathbb{C}^{|\mathcal{V}_0|}$ and defining $h^{k,+} = A_k^+h^{k-1,+} = A_k^+\cdots A_{1}^+ h^0$ for $k\geq 0$ and $h^{k,-} =  h^{k-1,-}A_k^- =  h^0 A_1^-\cdots A_{k}^- $ for $k\leq 0$, we necessarily have that the families of multimatrix algebras are
      $$\mathcal{M}_k^+ = M_{h_1^{k,+}}\oplus \cdots \oplus M_{h_{n_k}^{k,+}}\hspace{.5in}\mbox{ and }\hspace{.5in}\mathcal{M}_k^- = M_{h_1^{k,-}}\oplus \cdots \oplus M_{h_{n_k}^{k,-}},$$
      where $n^\pm_k$ is the number of matrix algebras in the multimatrix algebra $\mathcal{M}_k^\pm$, $h_j^{k,\pm}$ is the $j^{th}$ entry of $h^{k,\pm}$, and the inclusions $i_{|k|}^\pm:\mathcal{M}_{|k|-1}^\pm\rightarrow \mathcal{M}_{|k|}^\pm$ are defined up to unitary equivalence by the matrices $A_k^\pm$. The vectors $h^{k,\pm}$ are called the \textbf{height vectors} of the positive and negative parts of $\mathcal{B}$, respectively.

With these systems of inclusions one can define the inductive limits
\begin{equation}
 LF(\mathcal{B}^+) := \bigcup_k\mathcal{M}_k^+ = \lim_{\rightarrow}(\mathcal{M}_k^+,i_k^+) \hspace{.7in}  LF(\mathcal{B}^-) := \bigcup_k\mathcal{M}_k^- = \lim_{\rightarrow}(\mathcal{M}_k^-,i_k^-)
\end{equation}
which are $*$-algebras called the \textbf{locally finite (LF) algebras} defined by $\mathcal{B}$. Their $C^*$-completion
$$AF(\mathcal{B}^+) := \overline{LF(\mathcal{B}^+)},\hspace{.8in AF(\mathcal{B}^-) := \overline{LF(\mathcal{B}^-)}}$$
are the \textbf{approximately finite-dimensional (AF) algebras} defined by $\mathcal{B}$. Note that the LF algebras $LF(\mathcal{B}^\pm)$ sit densely inside the AF algebras $AF(\mathcal{B}^\pm)$.
\subsection{Traces}
\label{subsec:traces}
A \textbf{trace} on a $*$-algebra $\mathcal{A}$ is a linear functional $\tau:\mathcal{A}\rightarrow \mathbb{C}$ which satisfies $\tau(ab) = \tau(ba)$ for all $a,b\in\mathcal{A}$. The set of all traces of $\mathcal{A}$ forms a vector space over $\mathbb{C}$ and it is denoted by $\tr(\mathcal{A})$. There is no assumption that traces are positive (that is, $\tau(aa^*)>0$). A \textbf{cotrace} $\tau'$ is an element of the dual vector space $\tr^*(\mathcal{A}):=\tr(\mathcal{A})^*$.

For $M_\ell$, the algebra of $\ell\times\ell$ matrices, $\tr(M_\ell)$ is one-dimensional and generated by the trace $ a\mapsto \sum_{j=1}^\ell a_{jj}$. For a multimatrix algebra $\mathcal{M} = M_{\ell_1}\oplus\cdots\oplus M_{\ell_n}$, the dimension of $\tr(\mathcal{M})$ is $n$ and is generated by the canonical traces in each summand, so $\mathrm{Tr}(\mathcal{M}) =  \tr(M_{\ell_1})\oplus\cdots\oplus \tr(M_{\ell_n})$.

Let $i_k^+:\mathcal{M}_{k-1}^+\rightarrow \mathcal{M}_k^+$ be the family of inclusions defined by the positive part of a Bratteli diagram $\mathcal{B}^+$. Then there is a dual family of inclusions $i_k^*:\tr(\mathcal{M}_k^+)\rightarrow \tr(\mathcal{M}_{k-1}^+)$ (and an analogous family $i_k^*:\tr(\mathcal{M}_k^-)\rightarrow \tr(\mathcal{M}_{k-1}^-)$). The trace spaces of the LF algebras defined by a Bratteli diagram $\mathcal{B}$ are then the inverse limits
\begin{equation}
  \begin{split}
    \tr(\mathcal{B}^+) &:= \tr(LF(\mathcal{B}^+)) = \lim_{\leftarrow} (i^*_k,\tr(\mathcal{M}^+_k)) \\
    \tr(\mathcal{B}^-) &:= \tr(LF(\mathcal{B}^-)) = \lim_{\leftarrow} (i^*_k,\tr(\mathcal{M}^-_k))
    \end{split}
  \end{equation}
which are vector spaces. The respective spaces of cotraces are then
$$\tr^*(\mathcal{B}^+)  = \lim_{\rightarrow} ((i^*_k)^*,\tr^*(\mathcal{M}^+_k))\hspace{.4in}\mbox{ and } \hspace{.4in} \tr^*(\mathcal{B}^-)  = \lim_{\rightarrow} ((i^*_k)^*,\tr^*(\mathcal{M}^-_k)).$$

For a Bratteli diagram $\mathcal{B}^+$ the canonical generating traces for $\mathrm{Tr}(\mathcal{M}_k^+)$ will be denoted by $\mathfrak{t}_{k,1},\dots, \mathfrak{t}_{k,|\mathcal{V}_k|}$ so that
\begin{equation}
    \label{eqn:canTraces}
  \left\langle \mathfrak{t}_{k,1},\dots,  \mathfrak{t}_{k,|\mathcal{V}_k|}\right\rangle = \mathrm{Tr}(\mathcal{M}_k^+)\hspace{.5in}\mbox{ and }\hspace{.5in} \mathfrak{t}_{k,j}(\mathrm{Id}) = h^{k,+}_{j}
\end{equation}
for $j\in\{1,\dots, |\mathcal{V}_k|\}$.
\section{Renormalization}
\label{sec:renorm}
Let $\sigma:\Sigma_N\rightarrow \Sigma_N$ be the shift map and $V\subset \Sigma_N$ an open set.
\begin{definition}
A \textbf{maximal sequence of return times} to $V$ for $x\in\Sigma_N$ under $\sigma$ is a sequence $k_n\rightarrow \infty$ such that $\sigma^{k_n}(x)\in V$ for all $n$ and such that if $\{k'_m\}$ satisfies $\sigma^{k'_m}(x)\in V$ for all $m$ then $\{k'_m\}\subset \{k_n\}$.
\end{definition}

Let $S_1,\dots, S_N$ be a collection of uniformly expanding compatible substitution rules on the prototiles $t_1,\dots, t_M$. For $x\in\Sigma_N$, suppose that the resulting Bratteli diagram $\mathcal{B}_x$ is minimal and denote by $\Omega_x$ the resulting tiling space. The (right) shift map $\sigma:x\mapsto \sigma(x)$ induces a homeomorphism $\Phi_x:\Omega_x\rightarrow \Omega_{\sigma(x)}$ giving the time-change equation
      \begin{equation}
        \label{eqn:conj}
\Phi_x\circ \varphi_{\theta_{x_{-1}}t}(\mathcal{T}) = \varphi_{t}\circ \Phi_x(\mathcal{T})
        \end{equation}
      for any $\mathcal{T}\in\Omega_x$. That $\Phi_x$ is a homeomorphism follows from the fact that $\Omega_x$ is constructed from the bi-infinite Bratteli diagram $\mathcal{B}_x$ and since the shift is invertible no information is lost when shifting indices. Recall the notion of control points from \S \ref{subsubsec:controlPts}. From (\ref{eqn:controlPt}) it follows that for any $\mathcal{T} = \Delta(\bar{e})\in\Omega_x$, with $\bar{e}\in \mathring{X}_{\mathcal{B}_x}$,
      \begin{equation}
        \label{eqn:CtrlPt2}
\theta_{x_{-1}}^{-1} \mathcal{C}_{\mathcal{T}}\subset \mathcal{C}_{\Phi_{\sigma(x)}^{-1}(\mathcal{T})}.
      \end{equation}
      Indeed, consider consider $\Delta_x(\bar{e}) = \mathcal{T}\in\Omega_x$ such that the origin coincides with the control point of a tile of type $i$. Then the negative part of $\bar{e}$ can be assumed to satisfy $e_{i-1} = c(s(e_i))$ for all $i<0$ and $r(e_{-1}) = v_i\in\mathcal{V}_0$ for some choice function $c$. Then the tiling $\Phi_{\sigma(x)}^{-1}(\mathcal{T}) = \Delta_{\sigma(x)}(\sigma(\bar{e}))$ is the image of a path $\sigma(\bar{e})$ with negative part satisfying $\sigma(e)_{i-1} = c(s(\sigma(e)_i))$ for all $i<0$, meaning, by (\ref{eqn:controlPt}), that the origin coincides with the control point of the tile covering the origin in $\Phi_{\sigma^{-1}(x)}^{-1}(\mathcal{T})$. Thus, (\ref{eqn:CtrlPt2}) holds for the control points located in the origin for tilings in $\mho_x$. Using (\ref{eqn:conj}) this can be extended to other control points, giving (\ref{eqn:CtrlPt2}). Moreover, since the choice function determines the canonical transversal via the control points, it follows that
      \begin{equation}
        \label{eqn:CtrlPt3}
        \Phi_{\sigma(x)}^{-1}(\mho_x)\subset \mho_{\sigma(x)}\hspace{.3in}\mbox{ and if }\hspace{.3in} \bar{e}\in X_c^x,\hspace{.3in}\mbox{ then }\hspace{.3in}(\Phi^{(k)}_{\sigma^{k}(x)})^{-1}(\mathcal{T}_{\bar{e}})\in\mho_{\sigma^{k}(x)}
      \end{equation}
      for any $k\in\mathbb{N}$, where $\Phi_x^{(k)} := \Phi_{\sigma^{k-1}(x)}\circ \cdots \circ \Phi_x:\Omega_x\rightarrow \Omega_{\sigma^k(x)}$. Finally, it follows that
      \begin{equation}
        \label{eqn:conj2}
\Phi_x^{(k)}\circ \varphi_{\theta_{(-k)_x}t}(\mathcal{T}) = \varphi_{t}\circ \Phi_x^{(k)}(\mathcal{T}),
      \end{equation}
      where $\theta_{(-k)_x}:= \theta_{x_{-k}}\cdots\theta_{x_{-1}}$. The following fact will be used later; the proof is left to reader.
      \begin{lemma}
        \label{lem:TLCfunctions}
        Let $h:\Omega_x\rightarrow \mathbb{R}$ be a level-$k$ TLC function. Then $\Phi_{\sigma^{-y}(x)}^{(y)*}h:\Omega_{\sigma^{-y}(x)}\rightarrow \mathbb{R}$ is a level-$k'$ TLC function, where $k' = \max\{0,k-y\}$.
      \end{lemma}
      \subsection{Return vectors}
      \label{subsec:RV}
Throughout this section it will be assumed that $\mathcal{B}_x$ is minimal. Recalling the definition of return vectors in \S\ref{subsubsec:retVectors}, note that since the sets $\Lambda_i^\mathcal{T}$ defined there do not depend on $\mathcal{T}$ by repetitivity, so denote by $\Lambda_{x,i}^{(0)}$ the set of return vectors to the subset $\mho_x^i\subset \mho_x$ of the canonical transversal made up of tilings whose tile containing the origin is translation equivalent to $t_i$:
$$\Lambda_{x,i}^{(0)} := \{ \tau\in\mathbb{R}^d: \varphi_\tau(\mathcal{T})\in \mho_x^i\ \mbox{ for some }\mathcal{T}\in\mho^i_x \} \hspace{.5in}\mbox{ and }\hspace{.5in}  \Lambda_x^{(0)} := \bigsqcup_i \Lambda_{x,i}^{(0)}.$$
Now, for every $k\in\mathbb{N}$, let $\Lambda_{x,i}^{(k)}$ be the set of vectors taking level-$k$ supertiles of type $i$ to level-$k$ supertiles of type $i$:
$$\Lambda_{x,i}^{(k)} := \{ \tau\in\mathbb{R}^d: \mbox{ there are level-$k$ supertiles }\mathcal{P},\mathcal{P}'\subset \mathcal{T}\in\Omega_x \mbox{ of type $i$ such that } \varphi_\tau(\mathcal{P})= \mathcal{P}' \} $$
and
$$\Lambda_x^{(k)} := \bigsqcup_i \Lambda_{x,i}^{(k)}.$$
All of these sets are well-defined for any $x\in\Sigma_N$ as long as $\mathcal{B}_x$ is minimal. In particular, it is well-defined along the orbit $\{\Omega_{\sigma(x)}\}$ of $x$. Note now that for every $i\in\{1,\dots, M\}$ and $k\in\mathbb{N}$, since level-$k$ supertiles of tilings in $\Omega_x$ are in bijection with level-0 supertiles of tilings in $\Omega_{\sigma^{-k}(x)}$, by (\ref{eqn:conj2}),
\begin{equation}
  \label{eqn:vecScale}
  \Lambda_{x,i}^{(k)} = \theta^{-1}_{(k)_x} \Lambda_{\sigma^{-k}(x),i}^{(0)}\hspace{.4in} \mbox{ and so }\hspace{.4in}\Lambda_{x}^{(k)} = \theta^{-1}_{(k)_x} \Lambda_{\sigma^{-k}(x)}^{(0)}.
  \end{equation}
Because any $\tau\in\Lambda_{x,i}^{(k)}$ is a translation vector between level-$k$ supertiles, it is also a translation between the tiles which tile the supertile. As such, by repetitivity/minimality, for all large enough $k$ one has that $\Lambda_{x,i}^{(k)}\subset \Lambda_x^{(0)}$ for any $i$ and so 
\begin{equation}
  \label{eqn:inclusionRet}
\Lambda_{x}^{(k)} \subset \Lambda_x^{(0)}.
\end{equation}
Denote by $\Gamma_x^{(k)} := \mathbb{Z}(\Lambda_{x}^{(k)})$ the subgroups of $\mathbb{R}^d$ that these point sets generate, with $\Gamma_x := \Gamma_x^{(0)}$. As in \S \ref{subsubsec:retVectors}, these groups are not merely generated by the vectors which represent them, but by such vectors along with the index corresponding to the tile type which defined the return. By final local complexity, these are finitely generated subgroups of $\mathbb{R}^d$ \cite[Theorem 2.1]{Lagarias:GeomMod1}. Let $r_x^{(k)}$ be their respective ranks, with $r_x := r_x^{(0)}$.
\begin{proposition}
Let $\mu$ be a minimal, $\sigma$-invariant ergodic probability measure. Then $r_x$ is constant $\mu$-almost everywhere.
\end{proposition}
\begin{proof}
By (\ref{eqn:vecScale}) and (\ref{eqn:inclusionRet}) we have that for any $k\in\mathbb{N}$
\begin{equation}
  \label{eqn:inclusionRet2}
  \Gamma_x^{(k)}= \theta_{(k)_x}^{-1} \Gamma_{\sigma^{-k}(x)} \subset   \Gamma_x.
  \end{equation}

For any $k\in\mathbb{N}$, consider $v_i\in\mathcal{V}_k$ and paths $p,q\in\mathcal{E}_{0,k}$ with $r(p) = r(q)$. Since $\mathcal{P}_k(v_i)$ is tiled by $|\mathcal{E}_{0,{v_i}}|$ tiles, there is a vector $\tau(p,q)$ which describes the translation between the marked point of the tile corresponding to $p$ in $\mathcal{P}_k(v_i)$ to the marked point on the tile corresponding to the path $q$. As such, $\tau(q,p) = -\tau(p,q)$.
\begin{lemma}
\label{lem:generators}
  Let $\mathcal{B}_x$ be minimal. For any $R>0$ there exists a $k\in\mathbb{N}$ such that
  $$\Lambda_x\cap B_R(0) \subset \left\{\tau(p,q)\in\mathbb{R}^d: p,q\in\mathcal{E}_{0,k} \mbox{ with }r(p) = r(q)\right\}.$$
\end{lemma}
\begin{proof}By definition, each $\tau\in \Lambda_x\cap B_R(0)$ comes from the translation within a patch of a diameter bounded by some constant which depends on $R$ and $x$. By finite local complexity, there are finitely many such patches. By repetitivity there exists a $R'>0$ such that a ball of radius $R'$ contains a copy of each of these patches. Let $r_{min}$ be the minimum of the injectivity radii of the prototiles $t_1,\dots, t_M$. That is, $r_{min}$ is the greatest $r>0$ such that there is an open ball of radius $r$ completely contained in each prototile $t_i$. If $k$ is such that $\theta_x^{-(k)}r_{min}>R'$ then each level-$k$ supertile contains a ball of radius greater than $R'$ and thus all of the patches which give the vectors in $\Lambda_x\cap B_R(0)$.
\end{proof}
\begin{lemma}
  \label{lem:recurrence}
Let $x\in\Sigma_N$ be recurrent under $\sigma^{-1}$ and $\mathcal{B}_x$ minimal. Then $r_x = r_{\sigma^{-k}(x)}$ for all $k\in\mathbb{N}$.
\end{lemma}
\begin{proof}
Let $k_n\rightarrow\infty$ be such that $\sigma^{-k_n}(x)\rightarrow x$ as $n\rightarrow \infty$. By (\ref{eqn:inclusionRet2}) it follows that
  $$r_x\geq r_{\sigma^{-1}(x)} \geq \cdots \geq r_{\sigma^{-k_n}(x)}\geq\cdots\geq  r_{\sigma^{-k_{n+1}}(x)}\geq\cdots.$$
Choosing generators $v_1,\dots , v_{r_x}\in\Lambda_x$, by Lemma \ref{lem:generators} and (\ref{eqn:inclusionRet2}), for all large enough $n$, we have that $r_{\sigma^{-k_n}(x)} = r_x$, so $r_x = r_{\sigma^{-k}(x)}$ for all $k\in\mathbb{N}$.
  \end{proof}
Finally, by Poincar\'e's recurrence theorem, $\mu$-almost every $x$ is $\sigma^{-1}$-recurrent and $r_x$ is constant along $\sigma^{-1}$ orbits, so it is constant almost-everywhere by ergodicity.
\end{proof}
\begin{definition}
The \textbf{rank} of a minimal, $\sigma$-invariant ergodic probability measure $\mu$ is the $\mu$-almost everywhere constant value $r_x$, and it is denoted by $r_\mu$.
\end{definition}
\begin{lemma}
  There exists a collection of nonempty open subsets $\{U_i\}$ of $\Sigma_N$ such that for any two $x,x'\in U_i$ with $\mathcal{B}_x$ and $\mathcal{B}_{x'}$ minimal and $r_x  =  r_{x'}$ then $\Gamma_x  = \Gamma_{x'}$.
\end{lemma}
\begin{proof}
Let $\mathcal{B}_x$ be minimal and consider a generating set $v_1,\dots, v_{r_x}\in\Lambda_x$ of $\Gamma_x$. Let $R$ be the maximum of the norm of the vectors in this generating set. By Lemma \ref{lem:generators} there is a $k$ such that the vectors in the generating set are all of the form $\tau(p,q)$ with $p,q\in\mathcal{E}_{0,k}$. Let $x'\in\Sigma$ be such that $x_i=x_i'$ for all $0<i\leq k$ and such that $\mathcal{B}_{x'}$ is minimal. Then the vectors $v_1,\dots, v_{r_x}$ are also elements of $\Lambda_{x'}$ and they generate a subgroup $\Gamma_{x'}\subset \Gamma_{x}$. Since $r_x = r_{x'}$ it follows that $\Gamma_x = \Gamma_{x'}$.
  \end{proof}
\begin{corollary}
  \label{cor:2}
Let $\mathcal{R}_k\subset \Sigma_N$ be the set of recurrent $x$ such that $r_x = k$ and $\mathcal{B}_x$ is minimal. For each $k\in\mathbb{N}$ such that $\mathcal{R}_k\neq \varnothing$ there exists a collection of disjoint subsets $\{U_i^k\}_i$ of $\mathcal{R}_k$ such that if $x\in U_i^k$ and $v_1,\dots, v_k \in \Lambda_x$ is a generator for $\Gamma_x$, then $v_1,\dots, v_k$ generates $\Gamma_{x'}$ for any $x'\in U_i^k$. In particular, $\Gamma_x = \Gamma_{x'}$ for any two $x,x'\in U_i^k$.
\end{corollary}
Let $y\in\Sigma_N$ be recurrent and $\mathcal{B}_y$ minimal and choose a set of generators $v_1^i,\dots, v_{r_y}^i$ in each $U_i^{r_y}$. For $x\in U_i^{r_y}\subset  \Sigma_N$, let $V_x$ be the $d\times r_y$ matrix with column vectors $v_1^i,\dots, v_{r_y}^i$ so that for each $\tau\in\Lambda_x$ there is a unique $\alpha_x(\tau)\in\mathbb{Z}^{r_y}$ such that $\tau = V_x\alpha_x(\tau)$. Here $\alpha_x:\Gamma_x\rightarrow \mathbb{Z}^{r_y}$ is the localized address map, which is well defined by the previous corollary. Note that both the matrix $V_x$ and address map $\alpha_x$ are locally constant in $\mathcal{R}_{r_y}$.

Given that $\theta_{x_{-1}}^{-1}\tau \in\Lambda_{\sigma(x)}$, there exists a $r_x\times r_x $ matrix $G_x$ such that $\theta_{x_{-1}}^{-1}V_{x} =  V_{\sigma(x)}G_x$ and, in general, 
\begin{equation}
  \label{eqn:vectorCocycle}
  \begin{split}
  &\theta_{(-n)_x}^{-1}V_{x} =  V_{\sigma^n(x)}G_{\sigma^{(n-1)}(x)}\cdots G_{\sigma(x)} G_x, \hspace{.5in}\mbox{ and thus } \\
&\theta_{(-n)_x}^{-1}\tau = \theta_{(-n)_x}^{-1} V_{x}\alpha_x(\tau) =  V_{\sigma^n(x)} G_x^{(n)} \alpha_x(\tau) = V_{\sigma^n(x)} \alpha_{\sigma^n(x)}\left(\theta_{(-n)_x}^{-1}\tau\right) 
  \end{split}
\end{equation}
for any $\tau\in\Gamma_x$, where $G^{(k)}_x :=  G_{\sigma^{k-1}(x)}\cdots G_{\sigma(x)} G_x$. Note that the map $x\mapsto G_x$ is locally constant on $\mathcal{R}_{r_x}$.

For $k$ and $i$ such that $U_i^k\neq\varnothing$, let $\tau_1,\dots, \tau_{k}$ denote a set of canonical generators for $\Gamma_{x}$ for all $x\in U_i^k$, and $v_j:= \alpha_x(\tau_j)\in \mathbb{Z}^{r_x}$ be their image under the address map, where it is assumed that $v_j$ is the $j^{th}$ canonical generator in $\mathbb{Z}^{r_x}$. Consider the real vector spaces generated by $v_1,\dots, v_{r_x}$
$$\mathfrak{R}_x:= \left\{\sum_{j=1}^{r_x} a_j v_j: a_j \in \mathbb{R}\right\}\hspace{.7in}\mbox{ and } \hspace{.7in}\mathfrak{R}_x^*:= \left\{\sum_{j=1}^{r_x} a_j v_j^*: a_j \in \mathbb{R}\right\}$$
with $\langle v_j,v^*_\ell\rangle = \delta_{j\ell}$. Here $\mathfrak{R}_x$ (and its dual) is regarded as an abstract vector space and not as a subspace of $\mathbb{R}^d$. In other words, $\mathfrak{R}_x$ is identified with is the real vector space $\mathbb{R}^{\{v_1,\dots, v_{r_x}\}}$ of dimension $r_x$ (and $\mathfrak{R}_x^*$ is identified with the dual of $\mathbb{R}^{\{v_1,\dots, v_{r_x}\}}$). By Corollary \ref{cor:2}, given $k$, since one can pick a canonical choice of generators in each $U_i^{k}$, one can identify all the vector spaces $\mathfrak{R}_{x'}$ for any $x'\in U_i^k$. In particular, one can chose a norm $\|\cdot\|$ for all the spaces $\mathfrak{R}_{x'}$.

The locally constant maps $G_x$ define dual maps $\bar{G}_x$ in a natural way: since $\theta_{x_1}^{-1}:\Gamma_{\sigma^{-1}(x)} \rightarrow \Gamma_x$ and, by (\ref{eqn:vectorCocycle}), $G_{\sigma^{-1}(x)}:\mathfrak{R}_{\sigma^{-1}(x)}\rightarrow \mathfrak{R}_x$, the dual maps satisfy $\bar{G}_x:= G^*_{\sigma^{-1}(x)}:\mathfrak{R}^*_x\rightarrow \mathfrak{R}^*_{\sigma^{-1}(x)} $.

Let $S_1,\dots, S_N$ be a collection of $N$ compatible substitution rules on prototiles $t_1,\dots, t_M$. Let $\Sigma_N'\subset \Sigma_N$ be the set of $x\in\Sigma$ such that $\Omega_x$ is well defined (if the substitution rules $S_i$ are all primitive, then this is defined for all $x\in\Sigma_N$). 
\begin{definition}
The \textbf{return vector bundle} is the set
$$\mathfrak{R}_N := \left\{(x,v):x\in \Sigma_N'\mbox{ and }v\in\mathfrak{R}^*_x\right\}$$
along with projection map $p_R:\mathfrak{R}_N\rightarrow \Sigma_N'$. The \textbf{return vector cocycle} is the map $\mathfrak{G}:\mathfrak{R}_N\rightarrow \mathfrak{R}_N$ defined by $\mathfrak{G}:(x,v)\mapsto (\sigma^{-1}(x), \bar{G}_x v)$ for all $(x,v)\in\mathfrak{R}_N$. The cocycle will be denoted $\bar{G}^{(n)}_x:= \bar{G}_{\sigma^{n-1}(x)}\circ\cdots\circ \bar{G}_x$, and set $\bar{G}^{(m,n)}_x:= \bar{G}_{\sigma^{m-1}(x)}\cdots\bar{G}_{\sigma^{n}(x)} $.
\end{definition}

It is not clear whether the maps $G_x$ are invertible. In any case, one can invoke the semi-invertible Oseledets theorem \cite[Theorem 4.1]{FLQ:coherent}. 
\begin{theorem}[Semi-invertible Oseledets Theorem]
  \label{thm:oseledetsReturn}
  Let $\mu$ be a minimal $\sigma$-invariant ergodic probability measure. Suppose that $\log^+\|G_x\|\in L^1(\Sigma_N,\mu)$. Then there exists numbers $\omega_1^+ > \omega_2^+ > \cdots >\omega^+_{\sh^+_\mu} > 0 >\omega_1^- > \omega_2^- > \cdots > \omega^-_{\sh^-_\mu}$  and a measurable, $\mathfrak{G}$-invariant family of subspaces $Y_j^\pm(x), Y^\infty(x)\subset \mathfrak{R}_x$ such that for $\mu$-a.e. $x$:
  \begin{enumerate}
  \item The decomposition $\mathfrak{R}_x =Y^\infty_x \oplus Y^+_x\oplus Y^-_x$ holds with
    $$ Y^+_x = \bigoplus_{i=1}^{\sh_\mu^+} Y^+_i(x)\hspace{.8in}\mbox{ and }\hspace{.8in}Y^-_x=\bigoplus_{i=1}^{\sh_\mu^-} Y_i^-(x),$$
  \item \hspace{.3in} $G_x Y_i^\pm(x) = Y_i^\pm(\sigma^{-1}(x))$ \hspace{.3in} and \hspace{.3in} $G_x Y^\infty(x)\subset Y^\infty(\sigma^{-1}(x))$,
  \item For all $\xi\in Y_i^\pm(x)\backslash \{0\}$ and $\xi_\infty\in Y^\infty(x)\backslash \{0\}$, one has
    $$\lim_{n\rightarrow \infty} \frac{1}{n}\log\|G^{(n)}_x\xi\| = \omega_i^\pm\hspace{.7in}\mbox{ and } \hspace{.7in} \lim_{n\rightarrow \infty} \frac{1}{n}\log\|G_x^{(n)}\xi\| = -\infty.$$
\end{enumerate}
\end{theorem}
The numbers $\omega_i^\pm$ are the \textbf{Lyapunov exponents} of the return vector cocycle.
\subsection{The trace cocycle}
\label{subsec:traceCocycle}
This section recalls the construction from \cite{T:TTT}. Let $\mathcal{B}_x$ be a bi-infinite Bratteli diagram as constructed in \S \ref{sec:RandSub}. By the definition of $LF(\mathcal{B}^+_x)$, there is a countable family of multimatrix algebras $\mathcal{M}_k^x := \mathcal{M}_k^{\mathcal{B}^+_x}$, and thus a countable family of vector spaces of traces $\tr(\mathcal{M}_k^x) = \left\langle \mathfrak{t}^k_{x,i}\right\rangle_{i=1}^{|\mathcal{V}_k|}$, where $\mathfrak{t}^k_{x,j}$ is the canonical trace in the $j^{th}$ matrix algebra of $\mathcal{M}_k^x$ as defined in (\ref{eqn:canTraces}).

The map $i^*_1:\tr(\mathcal{M}_1^x)\rightarrow \tr(\mathcal{M}_0^x)$ can be explicitly computed as follows. Denote the standard basis of $\mathbb{C}^{M} = \mathcal{M}_0^x$ by $\{\delta_1,\dots, \delta_{M}\}$ and by $\{\mathfrak{t}^1_{x,1},\dots, \mathfrak{t}^1_{x,M}\}$ the canonical basis for $\tr(\mathcal{M}_1^x)$. Then every element in $a\in \mathcal{M}_0^x$ and any element $\mathfrak{t}\in \tr(\mathcal{M}_1^x)$ is of the form
$$a = \sum_{j=1}^{M}\alpha_j \delta_j,\hspace{1in}\mbox{ and } \hspace{1in} \mathfrak{t}  = \sum_{i = 1}^{M} \beta_i \mathfrak{t}^1_{x,i},$$
respectively, where $\alpha_j,\beta_i\in\mathbb{C}$. Then
\begin{equation}
  \label{eqn:tracemap}
  \begin{split}
    \mathfrak{t}(i_1(a)) &= \mathfrak{t}\left(i_1\left(\sum_{j=1}^{M}\alpha_j\delta_j\right)\right) = \sum_{i=1}^{M}\beta_i \mathfrak{t}^x_{1,i}\circ i_1 \left(\sum_{j=1}^{M}\alpha_j\delta_j\right) \\
    &= \sum_{i=1}^{M}\beta_i  \sum_{j=1}^{M}\alpha_j\left(\mathfrak{t}^x_{1,i}\circ i_1\left(\delta_j\right)\right)  = \sum_{i=1}^{M}\beta_i  \sum_{j=1}^{M}\alpha_j (\mathcal{A}^x_1)_{i,j} = \langle (\mathcal{A}^x_1)^* \mathfrak{t},a\rangle = (i_1^*\mathfrak{t})(a). 
  \end{split}
\end{equation}
So the induced map $i^*_1$ is given by the matrix $(\mathcal{A}^x_1)^*$ and the dual map $(i_1^*)^*:\tr^*(\mathcal{M}_0^x)\rightarrow \tr^*(\mathcal{M}_1^x)$ is given by the matrix $\mathcal{A}^x_1$.
\begin{lemma}
  \label{lem:induced}
  The left shift $\sigma^{-1}:\Sigma_N\rightarrow \Sigma_N$ induces a linear map $\sigma_*^{-1}:\tr^*(\mathcal{M}_0^x)\rightarrow \tr^*\left(\mathcal{M}^{\sigma^{-1}(x)}_0\right)$ given by the matrix $\mathcal{A}^x_1$ after the canonical identifications $\tr^*(\mathcal{M}_0^x) = \mathbb{C}^{M}=\tr^*(\mathcal{M}_0^{\sigma^{-1}(x)})$.
\end{lemma}
\begin{proof}
  Define $\mathfrak{t}^x_1:\mathcal{M}_1^x\rightarrow \mathbb{C}^{M}= \mathcal{M}_0^{\sigma^{-1}(x)}$ to be the map defined by
  $$a\mapsto \left(\mathfrak{t}^x_{1,1}(a), \mathfrak{t}^x_{1,2}(a), \dots,\mathfrak{t}^x_{1,M}(a)\right)$$
  for all $a\in\mathcal{M}_1^x$. Now, for $a\in\mathcal{M}_0^x$, consider the composition of $\mathfrak{t}^x_1$ with the inclusion map $i_1$, $\sigma_*^{-1}:=\mathfrak{t}^x_1\circ i_1: \mathcal{M}_0^x\rightarrow \mathcal{M}^{\sigma^{-1}(x)}_0$. Note that the dual map $\mathfrak{t}^{x*}_1:\tr(\mathcal{M}_0^{\sigma^{-1}(x)})\rightarrow \tr( \mathcal{M}_1^x )$ is represented by the identity matrix when using the canonical basis in both $\tr(\mathcal{M}_0^{\sigma^{-1}(x)})$ and $\tr( \mathcal{M}_1^x )$. By this fact and (\ref{eqn:tracemap}) we have that the dual map $\sigma^{-1*}: \tr(\mathcal{M}^{\sigma^{-1}(x)}_0) \rightarrow\tr( \mathcal{M}_0^x )$ is represented by the matrix $(\mathcal{A}^x_1)^*$ in canonical coordinates, which defines a map $\sigma_*^{-1}:\tr^*(\mathcal{M}_0^x)\rightarrow \tr^*(\mathcal{M}^{\sigma^{-1}(x)}_0)$ through the matrix $\mathcal{A}^x_1$.
\end{proof}
\begin{definition}
  Let $S_1,\dots, S_N$ be $N$ compatible substitution rules on the prototiles $t_1,\dots, t_M$. The \textbf{trace bundle} $\tr^+(\Sigma_N)$ is the set of pairs $(x,s)$ where $x\in\Sigma_N$ and $s\in \tr^*(\mathcal{M}_0^{\mathcal{B}^+_x})$. The \textbf{trace cocycle} is map $G_+:\tr^+(\Sigma_N)\rightarrow \tr^+(\Sigma_N)$ over the left shift $\sigma^{-1}:\Sigma_N\rightarrow \Sigma_N$ with the linear map $\sigma_*^{-1}$ on the fiber, as obtained from Lemma \ref{lem:induced}.  The composition $\sigma_*^{-1}\circ \cdots \circ \sigma_*^{-1}$ $k$ times will be denoted by $\sigma^{-(k)}_*:\tr^*(\mathcal{M}_0^\mathcal{B})\rightarrow \tr^*(\mathcal{M}_0^{\sigma^{-k}(\mathcal{B})})$ .
\end{definition}
Since the fiberwise linear maps $\sigma_*^{-1}$ agree, under a canonical choice of coordinates, with the linear maps given by the matrix $\mathcal{A}^x_1$, the cocycle can be written in these coordinates, for $m< n$, as
\begin{equation}
  \label{eqn:cocycleCoord}
  \Theta_x^{(m,n)} := \mathcal{A}_n^x \cdots \mathcal{A}_{m+1}^x\hspace{.5in} \mbox{ and }\hspace{.5in}\Theta_x^{(k)}:= \Theta_x^{(0,k)}.
\end{equation}
Oseledets theorem will not be necessary for the trace cocycle, only the existence of the leading Lyapunov exponent $\eta_1$. This is guaranteed by the theorem of Furstenberg and Kesten \cite{FK:matrices} as long as $\log^+\|\sigma_*^{-1}\|\in L^1(\Sigma_N,\mu)$. Any point for which the theorem of Furstenberg and Kesten result holds will be called \textbf{Furstenberg-Kesten-regular}.

\section{LF algebras and return vectors}
\label{sec:LFreturn}
Let $S_1,\dots, S_N$ be $N$ uniformly expanding substitutions on the prototiles $t_1,\dots, t_M\subset \mathbb{R}^d$, $x\in\Sigma_N$ and suppose that $\mathcal{B}_x$ is minimal. Recall the sequence of height vectors $h^k_x$ for the positive part $\mathcal{B}_x^+$ from \S \ref{subsec:algebras}. Let $\{\mathcal{A}^x_k\}$ be the sequence of matrices defined by $\mathcal{B}_x^+$. Starting with $h^0_x=(1,\dots, 1)^T$, the height vectors are defined by $h^k_x = \mathcal{A}^x_k\cdots \mathcal{A}^x_1 h^0_x$. Note that $h_{x,i}^k$ also is the number of paths from $\mathcal{V}_0$ to $v_i\in\mathcal{V}_k$: indeed, by the definition of the matrix $\mathcal{A}_1^x$ defined by the Bratteli diagram, $h^1_{x,i} = \mathcal{A}_1^xh_x^0$ is the number of paths from vertices in $\mathcal{V}_0$ to $v_i\in\mathcal{V}_1$, $h^2_{x,i} = \mathcal{A}_2^xh_x^1 = \mathcal{A}_2^x\mathcal{A}_1^xh_x^0$ is the number of paths from vertices in $\mathcal{V}_0$ to $v_i\in\mathcal{V}_2$, and so on.

Recall also that $\mathcal{B}^+_x$ defines a sequence of inclusions (up to unitary equivalence), which yields the LF-algebra $LF(\mathcal{B}^+_x)$ via the direct limit of the inclusions $i_k : \mathcal{M}^x_{k-1}\hookrightarrow \mathcal{M}^x_k$, $k\in\mathbb{N}$, where
\begin{equation}
  \label{eqn:MMalgebra}
  \mathcal{M}_k^x = M_{h_{x,1}^k}\oplus\cdots\oplus M_{h_{x,M}^k}.
\end{equation}
From \S \ref{sec:RandSub}, the number $h_{x,i}^k$ also denotes the number of tiles inside the patch defined by the canonical level-$k$ supertile $\mathcal{P}_k(v_i)$. Indeed, by (\ref{eqn:approximant}), a level-$k$ supertile is tiled by tiles in bijection with the number of paths starting at $\mathcal{V}_0$ and ending at the corresponding vertex in $\mathcal{V}_k$, and the number of such tiles is an entry in $h^{k}_x$. Thus $\mathcal{B}_x^+$ defines, through its sequence $\{i_k\}$ of inclusions, for any $k\in\mathbb{N}$, a bijection between the set of paths from $\mathcal{V}_0$ to $v_i\in\mathcal{V}_k$ and the set $\{1,\dots, h^k_{x,i}\}$. That is, any $\ell\in\{1,\dots, h_{x,i}^k\}$ determines a unique path $p_\ell\in \mathcal{E}_{v_i}$. Under this correspondence, define the sets of indices
$$I^k_{x,i,j} = \left\{\ell\in \{1,\dots, h^k_{x,j}\}:\substack{\mbox{ the path $p_\ell\in\mathcal{E}_{0,v_j}$ corresponding to this index satisfies} \\\mbox{ $s(p) = v_i\in\mathcal{V}_0,\; r(p) = v_j\in\mathcal{V}_k$}} \right\},$$
and note that $\sum_i|I^k_{x,i,j}| = h^k_{x,j}$. Define for each index $i\in\{1,\dots, M\},k\in\mathbb{N}$ and spectral parameter $\lambda\in\mathbb{R}^d$ the element
$$a^k_{x,i,\lambda} = \left(a^k_{x,i,1,\lambda},\dots,a^k_{x,i,M,\lambda}\right) \in \mathcal{M}^x_k$$
by defining $a^k_{x,i,j,\lambda}\in M_{h^k_{x,j}}$ to be the diagonal matrix which is only non-zero for the diagonal elements with index in $I^k_{x,i,j}$ and, for an indices $\ell\in I^k_{x,i,j}$, the entry is
\begin{equation}
  \label{eqn:diagonal}
  \left(a^{k}_{x,i,j,\lambda}\right)_{\ell\ell} = e^{-2\pi \imath \langle\lambda, \tau_\ell\rangle},
\end{equation}
where $\tau_\ell\in\mathbb{R}^d$ is the vector corresponding to the control point inside the tile of type $t_i$ associated to the index $\ell\in I^k_{x,i,j}$ in the canonical supertile $\mathcal{P}_k(v_j)$ for any $v_j\in\mathcal{V}_k$. Thus the element $a^k_{x,i,\lambda}\in LF(\mathcal{B}_x^+)$ tracks the return vectors to the subtransveral $\mho_x^i$ inside of level-$k$ supertile $\mathcal{P}_k(v_j)$. That is, it records how tiles of type $i$ sit inside canonical level-$k$ $\mathcal{P}_k(v_j)$ supertiles through its interaction with the spectral parameter $\lambda\in\mathbb{R}^d$ defined in (\ref{eqn:diagonal}).

There is a relationship between canonical level-$k$ super tiles and the level-$(k-1)$ supertiles that tile it. But first, consider the relationship between level-$1$ supertiles with the tiles that tile it. The first set of edges $\mathcal{E}_1$ of the diagram $\mathcal{B}_x$ is determined by the element $x_1$. Taking any vertex $v\in\mathcal{V}_1$, its associated level-1 supertile $\mathcal{P}_1(v)$ is tiled by $|r^{-1}(v)|$ tiles, each of which is marked by its own control point and one of which contains the control point of the level-$1$ supertile $\mathcal{P}_1(v)$. As such, there are vectors $\bar{0} = \tau_1,\dots, \tau_{|r^{-1}(v)|}\in\mathbb{R}^d$, one for each edge in $r^{-1}(v)$, which describe the translations between the control points on these $|r^{-1}(v)|$ tiles. Since there are $N$ choices for $x_1$, one for each substotution rule, there is a set of $N$ sets of vectors $\{\mathcal{W}^\ell\}_{\ell=1}^N$, one for each substitution rule $S_\ell$, with
\begin{equation}
  \label{eqn:canVectors}
  \mathcal{W}^\ell = \bigcup_{j=1}^M\mathcal{W}^\ell_{j} = \bigcup_{j=1}^M\bigcup_{i=1}^M \mathcal{W}^\ell_{i,j} \hspace{.6in}\mbox{ and }\hspace{.6in} |\mathcal{W}^\ell_{i,j}| = |r^{-1}(v_i)\cap s^{-1}(v_j)|.
\end{equation}
More precisely $\mathcal{W}^\ell_{i,j}$ is the set of vectors $\tau$ describing the position of the control points of tiles of type $j$ in the level-1 supertile $\mathcal{P}_1(v_i)$, positioned so that the control point of $\mathcal{P}_1(v_i)$ coincides with the origin (that is, in canonical position, $cf.$ \S \ref{subsec:choices}), under the subtitution rule $S_\ell$. It should be emphasized again that the zero vector is an element of $\bigcup_i\mathcal{W}^\ell_{i,j}$ for any $j,\ell$. This corresponds to the trivial translation from the tile which contains the control point in $\mathcal{P}_1(v_j)$ to itself, for any $v_j\in\mathcal{V}_k$.

\begin{figure}[t]
  \centering
  \includegraphics[width = 5.5in]{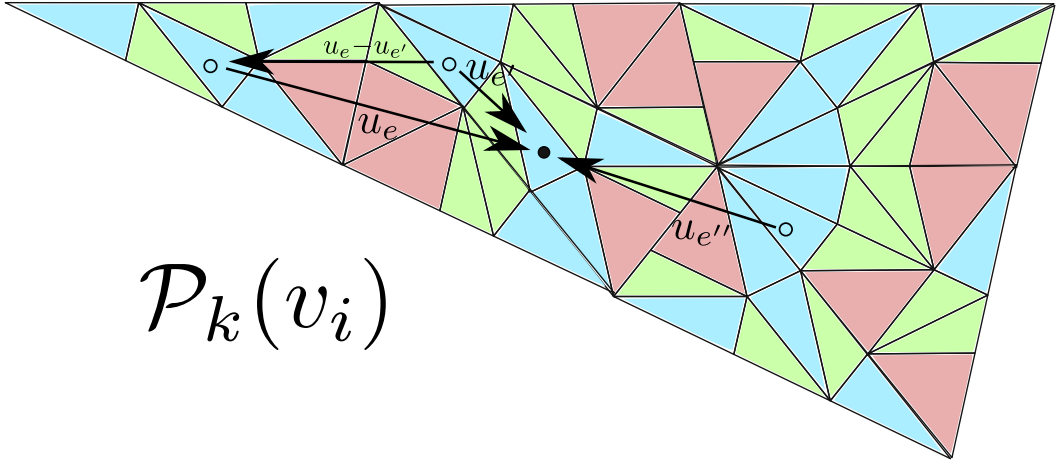}
  \caption{\small{This is a canonical level-$k$ supertile tiled by level-$(k-1)$ supertiles. The solid dot is the control point of the supertiles that contain it. The other circles represent the control points of some level-$(k-1)$ supertiles of the same type. The vectors $u_e, u_{e'}, u_{e''}$ represent the translations for the level-$(k-1)$ supertiles which are not in canonical position to be in canonical position (see Definition \ref{def:canonical}). This system of references allows to express translations between tiles of the same type, such as $u_e-u_{e'}$ as depicted here.}}
  \label{fig:canonical}
\end{figure}
More generally, for $k>0$ and $i,j\in \{1,\dots, M\}$ there is a collection of vectors $\mathcal{U}^x_{i,j,k}$ with $|\mathcal{U}^x_{i,j,k}| = |r^{-1}(v_i)\cap s^{-1}(v_j)| = (\mathcal{A}_k^x)_{i,j}$ whose elements are described as follows. Consider the canonical level-$k$ supertile $\mathcal{P}_k(v_i)$. It is tiled by $|r^{-1}(v_i)|$ level-$(k-1)$ supertiles $\mathcal{P}_{k-1}(\bar{p}_e)$, $e\in r^{-1}(v_i)$, all but one of which are not in canonical position (recall Definition \ref{def:canonical}). Therefore for each edge $e\in r^{-1}(v_i)$ there is a vector $u_e\in\mathbb{R}^d$ such that $\varphi_{u_e}(\mathcal{P}_{k-1}(\bar{p}_e)) = \mathcal{P}_{k-1}(v_j)$ for $e\in r^{-1}(v_i)\cap s^{-1}(v_j)$, that is, so that translating the level-$(k-1)$ supertile $\mathcal{P}_{k-1}(\bar{p}_e)$ is translated into canonical position. Thus $\mathcal{U}^x_{i,j,k}$ is the collection of all such vectors $u_e$ for $e\in r^{-1}(v_i)\cap s^{-1}(v_j)$, and difference vectors of the form $u_e-u_{e'}$ describe translations between level-$(k-1)$ supertiles inside level-$k$ supertiles (see Figure \ref{fig:canonical} for an auxiliary illustration). Thus, by (\ref{eqn:canVectors}), for $k>1$,
\begin{equation}
  \label{eqn:canVectors2}
  \mathcal{U}^x_{i,j,k} = \bar\theta_{x_{k-1}}\cdots\bar\theta_{x_1}\mathcal{W}^{x_k}_{i,j}.
  \end{equation}

Now consider the element $(i_k(a^{k-1}_{x,i,\lambda}))_j\in M_{h^k_{x,j}}$, where $M_{h^k_{x,j}}$ is the $j^{th}$ matrix algebra in $\mathcal{M}^x_k$ (see (\ref{eqn:MMalgebra})) of size $h^k_{x,j}\times h^k_{x,j}$. It is a diagonal matrix with entries expressing how tiles of type $i$ are located in all the level-$(k-1)$ supertiles contained inside $\mathcal{P}_k$, but it does not describe how the tiles of type $i$ are located inside the level-$k$ supertile $\mathcal{P}_k(v_j)$. Thus there is a diagonal matrix $\mathbb{M}_{k,j,\lambda}\in M_{h^k_{x,j}}$ with entries of the form $e^{-2\pi \imath \langle\lambda , u_e\rangle}$, where $u_e\in \mathcal{U}^x_{i,j,k}$, such that
 $$a^k_{x,i,j,\lambda}= \left(i_k(a^{k-1}_{x,i,\lambda})\right)_j  \mathbb{M}_{k,j,\lambda}$$
 and, more generally, there are diagonal elements $\mathbb{M}_{k,\lambda} = (\mathbb{M}_{k,1,\lambda},\dots, \mathbb{M}_{k,M,\lambda})\in\mathcal{M}^x_k$ such that
 $$ a^k_{x,i,\lambda}= i_k\left(a^{k-1}_{x,i,\lambda}\right)  \mathbb{M}_{k,\lambda},$$
 and so recursively applying the same argument there are $k$ diagonal matrices $M_{1,\lambda},\dots, M_{k,\lambda}$ with $M_{\ell,\lambda}$ having entries of the form $e^{2\pi \imath \langle\lambda , u_e\rangle}$, where $u_e\in \mathcal{U}^x_{i,j,\ell}$, such that
 \begin{equation}
   \label{eqn:aDecomp}
   a^k_{x,i,\lambda} = i_k\left(i_{k-1}\left(i_{k-2}\left( \cdots i_1(a^0_{x,i,\lambda})\mathbb{M}_{1,\lambda}\cdots    \right) \mathbb{M}_{k-2,\lambda}\right) \mathbb{M}_{k-1,\lambda}\right) \mathbb{M}_{k,\lambda}.
 \end{equation}
 
 Following (\ref{eqn:canTraces}), denote by $\mathfrak{t}_{k,1}^x,\dots, \mathfrak{t}_{k,M}^x$ the canonical traces for $\mathcal{M}^x_k$. The goal now is to find good bounds for $|\mathfrak{t}_{k,j}^x(a^k_{x,i,\lambda})|$ for any $i,j$. First, note that since all the entries of the diagonal matrices $\mathbb{M}_{\ell,\lambda}$ have unit norm, by (\ref{eqn:canTraces}) and (\ref{eqn:aDecomp}),
 \begin{equation}
   \label{eqn:traceUpBnd}
\sum_{i=1}^M|\mathfrak{t}_{k,j}^x(a^k_{x,i,\lambda})| \leq h^k_j.
\end{equation}

 Now consider
 $$\sum_{i=1}^M a^1_{x,i,\lambda} = \sum_{i=1}^M i_1(a^0_{x,i,\lambda})\mathbb{M}^{-1}_{1,\lambda} =  i_1(\mathrm{Id})\mathbb{M}^{-1}_{1,\lambda}.$$
 Applying the $j^{th}$ canonical trace at level 1:
 \begin{equation}
   \label{eqn:traceLevel1}
   \sum_{i=1}^M \mathfrak{t}_{1,j}^x(a^1_{x,i,\lambda}) =  \mathfrak{t}_{1,j}^x(i_1(\mathrm{Id})\mathbb{M}^{-1}_{1,\lambda}) = \sum_{e\in r^{-1}(v_j)}e^{2\pi \imath \langle \lambda, u_e\rangle} \mathfrak{t}_{0,s(e)}^x(\mathrm{Id})
 \end{equation}
 is the sum of $h^1_j$ exponential terms including, if $|r^{-1}(v_j)|>1$, a term for the form $1+e^{2\pi \imath \langle \lambda, w_e\rangle}$ for some $w_e\in \mathcal{W}^{x_1}_{j}-\mathcal{W}^{x_1}_{j}$. It follows that
 \begin{equation}
   \label{eqn:traceLevel1'}
   |\mathfrak{t}_{1,j}^x(i_1(\mathrm{Id})\mathbb{M}^{-1}_{1,\lambda})|\leq h^1_j - 2 + |1+e^{2\pi \imath \langle \lambda, w_e\rangle}|
 \end{equation}
and, from the inequality $|1+e^{2\pi \imath \omega}| \leq 2 - (1/2)\|\omega\|^2_{\mathbb{R}/\mathbb{Z}}$ for any $\omega\in\mathbb{R}$, that
\begin{equation}
  \label{eqn:traceBnd1}
 |\mathfrak{t}_{1,j}^x(i_1(\mathrm{Id})\mathbb{M}^{-1}_{1,\lambda})|\leq h^1_j - \frac{1}{2} \| \langle \lambda, w_e\rangle\|_{\mathbb{R}/\mathbb{Z}}^2
\end{equation}
for any $w_e\in\mathcal{W}^{x_1}_j$. The goal now is to generalize this type of estimate to traces applied to any of the elements $a^k_{x,i,\lambda}$.

First, note that $\mathfrak{t}_{k,j}^x(a^k_{x,i,\lambda})$ can be written as
\begin{equation}
  \label{eqn:traceExp}
  \begin{split}
    \mathfrak{t}_{k,j}^x(a^k_{x,i,\lambda}) &=  \mathfrak{t}_{k,j}^x(i_k(a^{k-1}_{x,i,\lambda} )\mathbb{M}_{k,\lambda})=\sum_{e\in r^{-1}(v_j)}e^{2\pi \imath \langle \lambda, u_e\rangle} \mathfrak{t}_{k-1,s(e)}^x(a^{k-1}_{x,i,\lambda})\\
    &= \sum_{e\in r^{-1}(v_j)}e^{2\pi \imath \langle \lambda, u_e\rangle} \sum_{\substack{e'\in\mathcal{E}_{k-1}:\\r(e') = s(e)}}e^{2\pi \imath \langle \lambda, u_{e'}\rangle} \mathfrak{t}_{k-2,s(e')}^x(a^{k-2}_{x,i,\lambda}) =\cdots \\
&=\sum_{e_k\in r^{-1}(v_j)} e^{2\pi \imath \langle \lambda, u_{e_k}\rangle}  \sum_{e_{k-1}\in r^{-1}(s(e_k))}e^{2\pi \imath \langle \lambda, u_{e_{k-1}}\rangle}\cdots  \sum_{e_{1}\in r^{-1}(s(e_2))}e^{2\pi \imath \langle \lambda, u_{e_1}\rangle} \mathfrak{t}_{0,s(e_1)}^x(a^{0}_{x,i,\lambda}),
\end{split}
\end{equation}
where $u_{e_\ell}\in \mathcal{U}^x_{i,j,\ell}$ for all $\ell\in\{1,\dots,k\}$ in the last expression.
\begin{remark}
  \label{rem:spectralCocycle}
The expression in (\ref{eqn:traceExp}) can be viewed from a different perspective. Let $\hat{\mathbb{M}}_x^k$ be the matrix-valued function which assigns, for each $\lambda\in\mathbb{R}^d$, the matrix $\hat{\mathbb{M}}_x^k(\lambda)$ with $i,j$ entry
$$\hat{\mathbb{M}}_{x}^k(\lambda)_{i,j} := \sum_{e\in \mathcal{W}_{i,j}^{x_k}}e^{-2\pi \imath \langle \lambda, \tau_e\rangle}.$$
As such, define family of matrix products parametrized by $\lambda\in\mathbb{R}^d$
\begin{equation}
  \label{eqn:SpecCocycle0}
  \hat{\mathbb{M}}_x^{(k)}(\lambda):=\hat{\mathbb{M}}_{x}^{k}\left(\theta_{(k)_x}^{-1}\lambda\right) \cdot \hat{\mathbb{M}}_{x}^{k-1}\left(\theta_{(k-1)_x}^{-1}\lambda\right) \cdots  \hat{\mathbb{M}}_{x}^{1}\left(\theta_{(1)_x}^{-1}\lambda\right)
\end{equation}
for each $k\in\mathbb{N}$ and $\lambda\in\mathbb{R}^d$. As such, it follows that
\begin{equation}
  \label{eqn:SpecCocycle}
  \mathfrak{t}_{k,i}^x(a^k_{x,j,\lambda}) = \hat{\mathbb{M}}_x^{(k)}(\lambda)_{i,j}
\end{equation}
and that $\hat{\mathbb{M}}^{(k)}_x(0) = \Theta_x^{(k)}$ for any $k\in\mathbb{N}$.
The matrices $\hat{\mathbb{M}}^k_{x}(\lambda)$ are called the \textbf{Fourier matrices} in \cite{BGM:correlation}, where the growth of the cocycle $\hat{\mathbb{M}}_x^{(k)}$ in (\ref{eqn:SpecCocycle0}) was used to rule out existence of absolutely continuous parts of diffraction measures. In the work of Bufetov-Solomyak \cite{BS:spectralCocycle}, this is called the \textbf{spectral cocycle}, which has been used to extablish quantitative weak-mixing results. In my take here, (\ref{eqn:SpecCocycle}) aims to extend these types of results.
\end{remark}
\begin{definition}
  \label{def:posSimple}
  Let $S_1,\dots, S_N$ be a family of uniformly expanding compatible substitution rules on the prototiles $t_1,\dots, t_M$. A word $w  = w_1w_2\dots w_n$ is \textbf{simple} if $w_i\cdots w_n\neq w_1\cdots w_{n-i+1}$ for all $1<i\leq n$. A word $w$ is \textbf{positively simple} if it is a simple word and in addition:
  \begin{enumerate}
  \item it breaks up into two subwords $w= w^-w^+ = w_1^-\cdots w_{n^-}^-w_1^+\cdots w_{n^+}^+$, and
  \item each entry in both matrices $Q^\pm = A_{w_{n^\pm}^\pm}\cdots A_{w_1^\pm}$ is strictly greater than 1.
  \end{enumerate}
  A $\sigma$-invariant probability measure $\mu$ for which $\mu([w])>0$ for a positively simple word $w$ is called a \textbf{positively simple measure}.
\end{definition}
Note that a positively simple measure is a minimal measure.
\begin{proposition}
  \label{prop:TracesBound}
  Let $S_1,\dots, S_N$ be a family of uniformly expanding compatible substitution rules on the prototiles $t_1,\dots, t_M$ and let $w = w^-w^+$ be a positively simple word. Suppose  $x\in\Sigma_N$ has infinitely many occurences of the word $w$, both in the future and the past. Denote by $k_n\rightarrow \infty$ the maximal sequence of return times of $x$ to $C([w^-.w^+])$ under $\sigma^{-1}$. Then there exists a $\ell\in\{1,\dots, M\}$, a finite set of vectors $\bar{\Lambda}_x\subset \Lambda_\ell^x$, $c>0$, and $k_x>0$ such that for any $m\in\mathbb{N}$
  \begin{equation}
    \label{eqn:twistBnd1}
    \begin{split}
      &\left|\mathfrak{t}^x_{m,j}(a^{m}_{x,i,\lambda})\right|  = |\hat{\mathbb{M}}_x^{(m)}(\lambda)_{i,j}|\\
      &\hspace{.4in}\leq \left\{
      \begin{array}{ll}
        \displaystyle\left\| \Theta_x^{(m)}\right\|\prod_{n<N}\left( 1 - c\max_{\tau\in \bar\Lambda_x} \|\langle\lambda,\bar\theta_{(k_n+k_x)_x}\tau\rangle \|^2_{\mathbb{R}/\mathbb{Z}}\right) & \mbox{ if }\,k_{N}\leq m < k_N+k_x  \\
        \displaystyle\left\| \Theta_x^{(m)}\right\|\prod_{n \leq N}\left( 1 - c\max_{\tau\in \bar\Lambda_x} \|\langle\lambda,\bar\theta_{(k_n+k_x)_x}\tau\rangle \|^2_{\mathbb{R}/\mathbb{Z}}\right) & \mbox{ if }\, k_N+k_x\leq m \leq k_{N+1} 
      \end{array} \right.
    \end{split}
  \end{equation}
  for any $i,j$.
\end{proposition}
\begin{proof}
  Since the positively simple word $w = w^-w^+$ occurs infinitely often, there exists (possibly empty) words $\{y_n\}_n$, none of which contains $w$ as a subword, such that
  $$x^+ = y_0w^-w^+y_1 w^-w^+y_2w^-w^+y_3w^-w^+y_4\cdots .$$
  Let $k_n\rightarrow \infty$ be defined as $k_1 = |y_0| + |w^-|$ and recursively $k_{n+1} = k_n+|w|+|y_n|$. This is a maximal sequence of return times to $C([w^-.w^+])$. Let $k_x = |w^+|$.

Consider the tiling space $\Omega_{\sigma^{-k_1}(x)}$ and $i\in \{1,\dots, M\}$. Consider the canonical level-$k_x$ supertile $\mathcal{P}_{k_x}(v_i)$ for tilings in this tiling space. Since $Q^+_{i,j}>1$, there are at least two tiles of type $t_{j}$ inside this supertile. As such, there is at least one vector $\tau\in\Lambda_{j}^{\sigma^{-k_1}(x)}$ which describes the translation between two tiles of type $t_j$. Let $\bar\Lambda_x$ be the union of all vectors $\Lambda_{j}^{\sigma^{-k_1}(x)}$ describing a translation equivalence between tiles of type $t_{j}$ in the canonical supertile $\mathcal{P}_{k_n}(v_j)$ for $\Omega_{\sigma^{-k_1}(x)}$, for all $j$; this is a finite set of vectors.

First suppose that $k_{n-1}+k_x\leq m\leq k_{n}$ for some $n>1$. Then as in (\ref{eqn:traceExp}),
\begin{equation}
  \label{eqn:traceExp2}
  \begin{split}
    \mathfrak{t}_{k_n,j}^x(a^{k_n}_{x,i,\lambda}) &= \sum_{\substack{p\in\mathcal{E}_{k_{n-1}+k_x,m}:  \\ r(p) = v_j}}e^{2\pi \imath \langle\lambda, u_p\rangle} \mathfrak{t}_{k_{n-1}+k_x,s(p)}^x(a^{k_{n-1}+k_x}_{x,i,\lambda}) \\
    &=\sum_{\substack{p\in\mathcal{E}_{k_{n-1}+k_x,m}: \\  r(p) = v_j}}e^{2\pi \imath \langle\lambda, u_p\rangle} \sum_{\substack{p'\in\mathcal{E}_{k_{n-1},k_{n-1}+k_x}: \\ r(p')= s(p)}} e^{2\pi \imath \langle\lambda, u_{p'}\rangle }    \mathfrak{t}_{k_{n-1},s(p')}^x(a^{k_{n-1}}_{x,i,\lambda}),
\end{split}
\end{equation}
where every vector $u_p$ above describes the position of a level-$k_{n-1}+k_x$ supertile inside a level-$k_n$ supertile relative to their control points, and the vectors $u_{p'}$ similarly describe the position of level $k_{n-1}$ supertiles inside level-$k_{n-1}+k_x$ supertiles. As such, in every term above of the form
\begin{equation}
  \label{eqn:problematic}
  \sum_{\substack{p'\in\mathcal{E}_{k_{n-1},k_{n-1}+k_x}: \\  r(p')= s(p)}} e^{2\pi \imath \langle\lambda, u_{p'}\rangle }\mathfrak{t}_{k_{n-1},s(p')}^x(a^{k_{n-1}}_{x,i,\lambda}),
\end{equation}
which depends only on $Q^+$, the vector $u_{p'}$ is of the form $\bar\theta_{(k_{n-1}+k_x)_x}u$ for some $u\in \bar\Lambda_x$. Using the fact that $Q^+_{i,j}>1$ for all $i,j$, this term can be bounded the in the same way that the bounds (\ref{eqn:traceLevel1})-(\ref{eqn:traceBnd1}) were obtained. Fixing $v_\ell\in \mathcal{V}_{k_{n-1}+k_x}$:
\begin{equation}
  \label{eqn:traceExp3}
  \begin{split}
&\left| \sum_{\substack{p'\in\mathcal{E}_{k_{n-1},k_{n-1}+k_x} \\ r(p')=v_\ell}} e^{2\pi \imath \langle\lambda, u_{p'}\rangle }
\mathfrak{t}_{k_{n-1},s(p')}^x(a^{k_{n-1}}_{x,i,\lambda}) \right| = \left|\sum_{j=1}^M \sum_{\substack{p'\in\mathcal{E}_{k_{n-1},k_{n-1}+k_x}: \\ s(p')=v_j \in\mathcal{V}_{k_{n-1}}\\ r(p')=v_\ell}}  e^{2\pi \imath \langle\lambda, u_{p'}\rangle } \mathfrak{t}_{k_{n-1},j}^x(a^{k_{n-1}}_{x,i,\lambda})\right| \\
&\leq  \sum_{j=1}^M \left|\sum_{\substack{p'\in\mathcal{E}_{k_{n-1},k_{n-1}+k_x}: \\ s(p')=v_j \in\mathcal{V}_{k_{n-1}}\\ r(p')=v_\ell}}  e^{2\pi \imath \langle\lambda, u_{p'}\rangle } \mathfrak{t}_{k_{n-1},j}^x(a^{k_{n-1}}_{x,i,\lambda})\right|\leq  \sum_{j=1}^M \left( Q_{\ell,j}^+ -2 + |1+e^{2\pi \imath \langle\lambda,u_{p'}\rangle}| \right)  |\mathfrak{t}_{k_{n-1},j}^x(a^{k_{n-1}}_{x,i,\lambda})| \\
&\hspace{.25in}\leq  \sum_{j=1}^M \left(  Q^+_{\ell,j} - \frac{1}{2}\| \langle\lambda, u_{p'}\rangle\|^2_{\mathbb{R}/\mathbb{Z}}\right)  \left| \mathfrak{t}_{k_{n-1},j}^x(a^{k_{n-1}}_{x,i,\lambda})\right|\\
&\hspace{.5in}\leq  \sum_{j=1}^M \left( 1 - c_w\| \langle\lambda, u_{p'}\rangle\|^2_{\mathbb{R}/\mathbb{Z}}\right)     Q^+_{\ell,j}   \left| \mathfrak{t}_{k_{n-1},j}^x(a^{k_{n-1}}_{x,i,\lambda})\right| \\
&\hspace{.75in}\leq  \left( 1 - c_w\max_{\tau\in \bar\Lambda_x}\| \langle\lambda, \theta_{(k_{n-1}+k_x)_x}^{-1}\tau\rangle\|^2_{\mathbb{R}/\mathbb{Z}}\right) \sum_{j=1}^M \left| \mathfrak{t}_{k_{n-1},j}^x(a^{k_{n-1}}_{x,i,\lambda})\right|   Q^+_{\ell,j} ,
  \end{split}
\end{equation}
where $c_w = (2M\max_{i,j}Q^+_{i,j})^{-1}$, where the second inequality was obtained by bounding the terms of the form (\ref{eqn:problematic}) in the same way as (\ref{eqn:traceLevel1})-(\ref{eqn:traceBnd1}) were obtained, the third inequality follows from a bound of the type done in (\ref{eqn:traceBnd1}), and where the last one follows from the fact that this estimate works for any return vector of the form $\bar\theta_{(k_{n-1}+k_x)_x}u$ for some $u\in \bar\Lambda_x$. The bound (\ref{eqn:traceExp3}) can now be used in bounding (\ref{eqn:traceExp2}):
\begin{equation*}
  \label{eqn:traceExp4}
  \begin{split}
   & \left| \mathfrak{t}_{m,j}^x(a^{m}_{x,i,\lambda})\right| \\
   & \hspace{.25in}\leq   \sum_{\substack{p\in\mathcal{E}_{k_{n-1}+k_x,m}: \\  r(p) = v_j}}e^{2\pi \imath \langle\lambda, u_p\rangle}   \left( 1 - c_w\max_{\tau\in \bar\Lambda_x}\| \langle\lambda, \bar\theta_{(k_{n-1}+k_x)}\tau\rangle\|^2_{\mathbb{R}/\mathbb{Z}}\right)\sum_{\ell=1}^M Q^+_{s(p),\ell}  \left| \mathfrak{t}_{k_{n-1},m}^x(a^{k_{n-1}}_{x,i,\lambda})\right| \\
    &\hspace{.25in}\leq \left( 1 - c_w\max_{\tau\in \bar\Lambda_x}\| \langle\lambda, \bar\theta_{(k_{n-1}+k_x)}\tau\rangle\|^2_{\mathbb{R}/\mathbb{Z}}\right)\sum_{\ell=1}^M \left|\left\{\substack{p\in \mathcal{E}_{k_{n-1},m}:\\ s(p) = v_\ell,\, r(p) = v_j}\right\}\right| \left| \mathfrak{t}_{k_{n-1},\ell}^x(a^{k_{n-1}}_{x,i,\lambda})\right| \\
    &\hspace{.25in}= \left( 1 - c_w\max_{\tau\in \bar\Lambda_x}\| \langle\lambda, \bar\theta_{(k_{n-1}+k_x)}\tau\rangle\|^2_{\mathbb{R}/\mathbb{Z}}\right) \sum_{\ell=1}^M \left(\Theta_x^{(k_{n-1},m)}\right)_{j,\ell} \left| \mathfrak{t}_{k_{n-1},\ell}^x(a^{k_{n-1}}_{x,i,\lambda})\right| .
  \end{split}
  \end{equation*}
The same type of bound applies to $\left| \mathfrak{t}_{k_{n-1},\ell}^x(a^{k_{n-1}}_{x,i,\lambda})\right|$. Thus, recursively, after finitely many bounds, one arrives at (\ref{eqn:twistBnd1}) in the case that $k_{n-1}+k_x\leq m\leq k_{n}$. In the case that $k_n\leq m <k_n+|w^+|$ the same argument holds, the point being that the number of terms in the product is one less than the number of times the word $w^+$ appears in $x$ up to its $m^{th}$ spot.
\end{proof}
\section{Twisted integrals}
\label{sec:twisted}
      The purpose of this section is to understand the behavior of twisted ergodic integrals along the dense $\mathbb{R}^d$ orbits on a compact tiling space $\Omega$. By Proposition \ref{prop:sum} the action is typically uniquely ergodic, and we denote the unique $\mathbb{R}^d$-invariant measure on $\Omega_x$ by $\mu_x$. More specifically, given $\mathcal{T}\in\Omega_x$ and $f\in L^2(\mu_x)$, recall the twisted ergodic integrals of $f$ by $\lambda$ as the family
      \begin{equation}
        \label{eqn:twisted}
\mathcal{S}_R^{\mathcal{T}}(f,\lambda) = \int_{C_R(0)} e^{-2\pi \imath \langle \lambda, \tau\rangle} f\circ \varphi_\tau(\mathcal{T})\, d\tau,
      \end{equation}
      where $C_R(0) = [-R,R]^d$ is the $L^\infty$-ball of radius $R$ in $\mathbb{R}^d$ centered at the origin, is defined for $\mu_x$-almost every $\mathcal{T}$.  Note that if $f\in L^2$ is an eigenfunction with eigenvalue $\lambda$, then $\mathcal{S}_R^\mathcal{T}(f,\lambda) =  C_d f(\mathcal{T})R^d$ for all $R$ and $\mathcal{T}\in\Omega_x$. Moreover, the spectral measure for an eigenfunction consists of a Dirac measure on $\lambda$.   So in order to rule out eigenfunctions it suffices to show either that the lower local dimension of spectral measures for a dense set of functions in $L^2$ is bounded below by a positive quantity, or that the twisted integrals (\ref{eqn:twisted}) for the same dense family of functions has growth of order $R^{d-\alpha}$ for some positive $\alpha$ (see Lemma \ref{lem:dimension}).

      From the time-change equation (\ref{eqn:conj}) and the change of variables formula it is straight-forward to verify that
      \begin{equation}
        \label{eqn:TwistedChange}
        \mathcal{S}_{R}^\mathcal{T}(f,\lambda) =\theta_{(n)_x}^{-d}\mathcal{S}_{R_{(n)_x}}^{\mathcal{T}_{(n)_x}}\left(f_{(n)_x},\lambda_{(n)_x}\right),
      \end{equation}
      for any $n>0$, where
      \begin{equation}
        \label{TwistedChange2}
        \begin{split}
          f_{(n)_x} = \Phi_{\sigma^{-n}(x)}^{(n)*}f,\hspace{.85in} \lambda_{(n)_x} &= \theta^{-1}_{(n)_x}\lambda, \hspace{.85in} R_{(n)_x} = \theta_{(n)_x}R,\\ \mbox{ and } \hspace{.25in} \mathcal{T}_{(n)_x} &= \Phi^{-1}_{\sigma^{-n}(x)}\circ\cdots\circ \Phi^{-1}_{\sigma^{-1}(x)}(\mathcal{T}).
        \end{split}
      \end{equation}
      This will become useful since it allows to convert twisted integrals of level-$n$ TLC functions on $\Omega_x$ into twisted integrals of level-$0$ TLC functions on $\Omega_{\sigma^{-n}(x)}$.
      
      For $\mathcal{T}\in\Omega_x$, let $g:\mathbb{R}^d\rightarrow \mathbb{R}$ be a (level-0) $\mathcal{T}$-equivariant function which is supported in the union of the interiors of tiles of type $t_\ell$ for some $\ell$. This function can be more concretely described as follows: there is a compactly supported function $\psi^\mathcal{T}_\ell:\mathbb{R}^d\rightarrow \mathbb{R}$ and a Delone set $\Lambda(\mathcal{T},\ell)$ such that
      $$g = \sum_{x\in \Lambda(\mathcal{T},\ell)}\delta_x * \psi_\ell^\mathcal{T},$$
      where $\Lambda(\mathcal{T},\ell)$ describes the position of the tiles of type $\ell$ in $\mathcal{T}$ relative to the origin, i.e., it is the set of control points for tiles of type $\ell$.
      Recall that there is a (level-0) TLC function $h$ such that $g(t) = h\circ \varphi_t(\mathcal{T})$. Let $t_j^{(k)}\subset \mathcal{T}$ be a level-$k$ supertile of type $j$ which is a patch of $\mathcal{T}$. Then there is a finite set $\Lambda_\ell(t_j^{(k)})\subset \Lambda(\mathcal{T},\ell)$ in bijection with the set of finite paths $p\in\mathcal{E}_{0,k}$ with $s(p) = v_\ell$ and $r(p) = v_j\in\mathcal{V}_k$ such that
      $$\int_{t_j^{(k)}} e^{-2\pi \imath \langle \lambda, \tau \rangle} g(\tau)\, d\tau = \sum_{x\in \Lambda_\ell(t_j^{(k)})} \int_{\mathbb{R}^d}e^{-2\pi \imath \langle\lambda ,\tau \rangle}\left( \delta_{x}*\psi^\mathcal{T}_\ell\right)\, d\tau = \sum_{x\in \Lambda_\ell(t_j^{(k)})} e^{-2\pi \imath \langle\lambda ,x \rangle}\widehat{\psi^\mathcal{T}_\ell}(\lambda) $$
      by the properties of the Fourier transform, and so
      \begin{equation}
        \label{eqn:localTwist}
        \begin{split}
          \int_{t_j^{(k)}} e^{-2\pi \imath \langle \lambda, \tau \rangle} h\circ \varphi_\tau (\mathcal{T})\, d\tau &= \int_{t_j^{(k)}} e^{-2\pi \imath \langle \lambda, \tau \rangle} g(\tau)\, d\tau \\
       &= e^{2\pi \imath \langle \lambda , \tau(t^{(k)}_j)\rangle} \mathfrak{t}^x_{k,j}(a^k_{x,\ell,\lambda})\widehat{\psi^\mathcal{T}_\ell}(\lambda),
        \end{split}
      \end{equation}
      where $\tau(t^{(k)}_j)\in\mathbb{R}^d$ is the vector which describes the position of the supertile $t^{(k)}_j$ with respect to the canonical level-$k$ supertile $\mathcal{P}_k(v_j)$, i.e. the translation equivalence. This immediately generalizes to any level-$0$, $\mathcal{T}$-equivariant (equivalently, TLC) functions as follows. 

Let $h = \sum h_i:\Omega_x\rightarrow \mathbb{R}$ be a level-$0$ TLC function, where $h_1,\dots, h_M:\Omega_x\rightarrow\mathbb{R}$ are the associated functions for the $M$ different level-$0$ tiles and, for $\mathcal{T}\in\Omega$, $g_\mathcal{T} = \sum g_i$ the associated $\mathcal{T}$-equivariant functions on $\mathbb{R}^d$. Then there exist compactly supported functions $\psi_1^\mathcal{T},\dots, \psi_M^\mathcal{T}$, where the support of each $\psi^\mathcal{T}_i$ is translation equivalent to a compact subset of the interior of $t_i$, and Delone sets $\Lambda(\mathcal{T},1),\dots,  \Lambda(\mathcal{T},M) $ such that
$$g_\mathcal{T} = \sum_{\ell=1}^M\sum_{\tau\in \Lambda(\mathcal{T},\ell)} \delta_\tau * \psi_\ell^\mathcal{T}.$$
Then as in (\ref{eqn:localTwist}), for a patch $t_j^{(k)}\subset \mathcal{T}$ which is also a level-$k$ supertile:
      \begin{equation}
        \label{eqn:localTwist2}
        \begin{split}
          \int_{t_j^{(k)}} e^{-2\pi \imath \langle \lambda, \tau \rangle} h\circ \varphi_\tau (\mathcal{T})\, d\tau &= \int_{t_j^{(k)}} \sum_{\ell=1}^Me^{-2\pi \imath \langle \lambda, \tau \rangle} g_\ell(\tau)\, d\tau \\
       &= e^{2\pi \imath \langle \lambda , \tau(t^{(k)}_j)\rangle} \sum_{\ell=1}^M\mathfrak{t}^x_{k,j}(a^k_{x,\ell,\lambda})\widehat{\psi^\mathcal{T}_\ell}(\lambda),
        \end{split}
      \end{equation}
      where, again, $g_\mathcal{T}(t) = h\circ \varphi_t(\mathcal{T})$ was used.

      Recall the renormalization maps $\Phi_{\sigma^{-1}(x)}:\Omega_{\sigma^{-1}(x)}\rightarrow\Omega_{x} $ and observe that if $h:\Omega_x\rightarrow \mathbb{R}$ is a level-$1$ TLC function, then $h':=\Phi_{\sigma^{-1}(x)}^*h:\Omega_{\sigma^{-1}(x)}\rightarrow \mathbb{R}$ is a level-$0$ TLC function. In particular, if $h = \sum_{i=1}^M h_i$, where $h_i$ is supported on tilings $\mathcal{T}\in\Omega_x$ whose level-$1$ supertile around the origin is of type $i$, then $h' =\sum_{i=1}^M \Phi_{\sigma^{-1}(x)}^*h_i = \sum_{i=1}^M h'_i$, where $h_i'$ are supported on tilings $\mathcal{T}'\in\Omega_{\sigma^{-1}(x)}$ whose tiles containing the origin is of type $i$. In addition, if $g_\mathcal{T}(t) = h\circ\varphi_t(\mathcal{T})$, $g_{\mathcal{T}'}'(t) = h'\circ\varphi_t(\mathcal{T}')$ and $\Phi_{\sigma^{-1}(x)}(\mathcal{T}') = \mathcal{T}$, using (\ref{eqn:conj}), then
      $$g'_{\mathcal{T}'}(t) = h'\circ \varphi_t(\mathcal{T}') = h\circ\Phi_{\sigma^{-1}(x)}\circ\varphi_t(\mathcal{T}') = h\circ\varphi_{\bar\theta_{x_1}t}\circ\Phi_{\sigma^{-1}(x)}(\mathcal{T}') = h\circ\varphi_{\bar\theta_{x_1}t}(\mathcal{T}) = g_\mathcal{T}(\bar\theta_{x_1}t),$$
      so they can both be written as
$$g_{\mathcal{T}'}'(t)= \sum_{\ell=1}^M\sum_{\tau\in \Lambda_\ell^{\mathcal{T}'}}\delta_\tau * \psi^{\mathcal{T}'}_\ell\hspace{.4in}\mbox{ and }\hspace{.4in} g_{\mathcal{T}}(t)= \sum_{\ell=1}^M\sum_{\tau\in \Lambda_\ell^{\mathcal{T}'}}\delta_{\bar{\theta}_{x_1}\tau} * (\bar{\theta}_{x_1}^*\psi^{\mathcal{T}}_\ell).$$

The following is \cite[Lemma 8.1]{ST:random}.
\begin{lemma}
\label{lem:packing}
  Let $S_1,\dots, S_N$ be a collection of uniformly expanding compatible substitution rules on $m$ prototiles. For $\mu$ a minimal, $\sigma$-invariant ergodic probability measure on $\Sigma_N$, for almost every $x$ and tiling $\mathcal{T}\in\Omega_x$ and $R>0$ there exists an integer $n = n(R)$ and a decomposition
  \begin{equation}
    \label{eqn:decomp3}
    \mathcal{O}_\mathcal{T}^-(C_R(0)) = \bigcup_{i = 0}^n\bigcup_{j=1}^M \bigcup_{k=1}^{\kappa_j^{(i)}}t^{(i)}_{j,k},
  \end{equation}
  where $t^{(i)}_{j,k}$ is a level-$i$ supertile of the tiling $\mathcal{T}$ of type $j$, such that
  \begin{enumerate}
  \item $\kappa_j^{(n)}\neq 0$ for some $j$ and $\mathrm{Vol}(C_R(0)) = (2R)^d\leq \RU \theta_{(n)_x}^{-d}$,
    \item $\sum_{j=1}^M \kappa_j^{(i)} \leq \RN \mathrm{Vol}(\partial C_R(0))  \theta_{(i)_x}^{d-1}$ for $i = 0,\dots, n-1$
  \end{enumerate}
  for some $\RU, \RN$ which depend only on the substitution rules.
\end{lemma}

\begin{proposition}
\label{prop:twistedBound}
  Let $S_1,\dots, S_N$ be a collection of uniformly expanding compatible substitution rules on $M$ prototiles. Let $\mu$ be a positively simple, $\sigma$-invariant ergodic probability measure for which the trace cocycle is integrable. For every positively simple word $w = w^-w^+$ such that $\mu(C([w^-.w^+]))>0$ there exists a finite set of vectors $\Lambda_w\subset \mathbb{R}^d$ and $c_w>0$ such that for any $\epsilon\in(0,1)$, for $\mu$-almost every $x\in\Sigma_N$ there is a subsequence $k_n\rightarrow \infty$, $z>1$ and $\bar{K}_{\epsilon,x}>0$ such that for any $y\in\mathbb{N}$ and any level-$y$ bounded TLC function $f:\Omega_x\rightarrow \mathbb{R}$, $\mathcal{T}\in\Omega_x$, and $\lambda\in\mathbb{R}^d\setminus \{0\}$
  $$\left|\mathcal{S}^\mathcal{T}_R(f,\lambda)\right|\leq C_{x,f,\epsilon} R^{d+\epsilon} \prod_{m_0<n<N_\epsilon(R)} \left( 1-c_w \max_{\tau\in \Lambda_w}\left\| \left\langle\lambda, \theta^{-1}_{(k_n+|w^+|)_x}\tau\right\rangle\right\|^2_{\mathbb{R}/\mathbb{Z}}  \right)+ \mathcal{O}(R^{d-1})$$
 for all $R>C_{\epsilon, x}z^y$, and $N_\epsilon(R)$ is the largest return time of $x$ to $C([w^-.w^+])$ less than $\frac{d\log R - \bar{K}_{\epsilon,x}}{\eta_1+\epsilon}$, $\eta_1$ is the largest Lyapunov exponent of the trace cocycle.
\end{proposition}
Let me summarize how this Proposition is proved. First, by Lemma \ref{lem:TLCfunctions} we can work with level-$0$ supertiles by shifting to the appropriate tiling space along the orbit of $x$. Secondly, the decomposition (\ref{eqn:decomp3}) gives a decomposition of the set over which the twisted integral is done, up to an error proportional to the size of the boundary, into supertiles of different levels. Then (\ref{eqn:localTwist2}) is used to rewrite the integrals over different supertiles as the product of the traces of the matrix elements defined in (\ref{eqn:diagonal}) in \S \ref{sec:LFreturn}. Finally, the bounds obtained in Proposition \ref{prop:TracesBound} for such traces are used to obtain the desired bound.
\begin{proof}
  Let $w$ be a positively simple word such that the cylinder set $C([w])$ has $\mu(C([w]))>0$. For $\mu$-almost every $x$ which is Furstenberg-Kesten-regular for the trace cocycle, without loss of generality, it can be assumed that $x^+ = w_0w^+y_1w^-w^+y_2w^-w^+y_3\cdots$ as in the proof of Proposition \ref{prop:TracesBound}. Let $\{k_m\}_m$ be defined recursively with $k_1 = |w_0|$ and $k_{m+1} = k_m + |w| + |y_n|$, and $k_x = |w^+|$. This defines the maximal sequence of return times of $x$ to the cylinder set $C([w])$ under $\sigma^{-1}$. Let $x\in\Sigma_N$ be one such $\mu$-generic, Furstenberg-Kesten-regular point.

Recall in the proof of Proposition \ref{prop:TracesBound} a finite set of special vectors denoted as $\bar\Lambda_x$ whose importance comes from the fact that they are return vectors found inside a level-$|w^+|$ supertile for any tiling $\mathcal{T}\in\Omega_z$ for any $z\in C([w^-.w^+])$. This set will be denoted by $\Lambda_w$, and so scaled versions of $\Lambda_w$ are return vectors for $\mu$-almost every $x$. In addition, the constant $c$ which appears in Proposition \ref{prop:TracesBound} depends only on $w$, so it will be denoted $c_w$ here for the same reasons. In order to simplify some notation, the product in Proposition \ref{prop:TracesBound} will be shortened to:
\begin{equation}
  \label{eqn:prod}
  \mathfrak{P}(m,n) = \prod_{m\leq i<n}\left( 1 - c_w\max_{\tau\in \Lambda_w} \|\langle\lambda,\theta_{(k_{i}+|w^+|)_x}^{-1}\tau\rangle \|^2_{\mathbb{R}/\mathbb{Z}}\right).
\end{equation}

Since $x$ is Furstenberg-Kesten-regular, for all $\varepsilon>0$ there exists a $C_\varepsilon = C_{\varepsilon,x}$ such that $\|\Theta_x^{(n)}\|\leq C_\varepsilon e^{(\eta_1+\varepsilon)n}$ for all $n>0$, where $\eta_1$ is the largest Lyapunov exponent of the trace cocycle. Given also that $n^{-1}\log \theta^{-d}_{(n)_x}\rightarrow \eta_1$, for any $\varepsilon>0$ there is a $K_\varepsilon = K_{\varepsilon,x}>0$ such that,
$$K_{\varepsilon}^{-1}\exp(\eta_1-\varepsilon)n< \theta^{-d}_{(n)_x}<K_{\varepsilon} \exp(\eta_1+\varepsilon)n$$
for all $n>0$. This implies that
$$\theta^{d-1}_{(n)_x}\leq K^2_{\varepsilon}\exp\left[\left(\frac{1-d}{d}\eta_1+\frac{d+1}{d}\varepsilon\right)n\right]\hspace{.2in}\mbox{ and }\hspace{.2in}\theta^{1-d}_{(n)_x}\leq C_\varepsilon K_\varepsilon \exp\left[\left(\frac{d-1}{d}\eta_1+\frac{d+1}{d}\varepsilon\right)n\right],$$
for all $n>0$, and in particular for any $0<i\leq n$, that
\begin{equation}
  \label{eqn:bdryEffect}
  \theta_{(n)_x}^{1-d}\theta_{(i)_x}^{d-1}\leq C_\varepsilon K^3_\varepsilon \exp\left[\frac{d-1}{d}\eta_1(n-i)+\frac{d+1}{d}\varepsilon (n+i)\right],
\end{equation}
which will be used below. Moreover, for every $\varepsilon>0$ there exists a $\Jh_\varepsilon>1$ such that
\begin{equation}
  \label{eqn:bdryEffect2}
\sum_{i=1}^{n-1} e^{\frac{\eta_1}{d}i} \leq \Jh_\varepsilon \left(1-\frac{c_\mu}{2}\right)\exp\left[\left(\frac{\eta_1+\varepsilon}{d}\right)n\right]
\end{equation}
for all $n>0$. Finally, for every $\varepsilon>0$ there is a $\Ch_\varepsilon>1$ such that for all $n>0$
\begin{equation}
  \label{eqn:bdryEffect3}
  \exp\left[\left( \eta_1+\frac{3(d+1)}{d}\varepsilon \right)n\right]\leq \Ch_\varepsilon \theta_{(n)_x}^{-\left(d+\frac{3d+4}{\eta_1}\varepsilon\right)},
\end{equation}
 which will also be used below.

 Let $h = \sum h_i :\Omega_x\rightarrow \mathbb{R}$ be a level-$y$ TLC function.
 By Lemma \ref{lem:TLCfunctions} $h':= \Phi^*_{\sigma^{-y}(x)}h = \sum \Phi^*_{\sigma^{-y}(x)} h_i = \sum  h_i'$ is a level-$0$ TLC function on $\Omega_{\sigma^{-y}(x)}$. By (\ref{eqn:TwistedChange}), instead of computing the twisted integral $\mathcal{S}_R^\mathcal{T}(h,\lambda)$ one can compute the twisted integral $\theta_{(y)_x}^{-d}\mathcal{S}_{R'}^{\mathcal{T}'}(h',\lambda')$, where $x' = \sigma^{-y}(x)$, $R' = \theta_{(y)_x}R$, $\mathcal{T}' = \Phi^{-1}_{\sigma^{-y}(x)}\circ \cdots \circ \Phi_{\sigma^{-1}(x)}^{-1}(\mathcal{T})$ and $\lambda' = \theta^{-1}_{(y)_x}\lambda$.

Note that in order for the decomposition in (\ref{eqn:decomp3}) to have at least one level-$y$ supertile it is necessary that $R>C_{\varepsilon,x}'z^{y}$ for some $C_{\varepsilon,x}'>0$, where $z:=e^{(\eta_1+\varepsilon)}$. This is a consequence of repetitivity and the fact that $\theta^{-d}_{(n)_x}$ and $e^{\eta_1 n}$ have the same exponential growth, and $\theta^{-d}_{(y)_x}$ is comparable to the volume of level-$y$ supertiles. For such large $R$ the decomposition from Lemma \ref{lem:packing} gives a decomposition for $\mathcal{O}^-_{\mathcal{T}'}(C_{R'}(0))$ as
$$\mathcal{O}_{\mathcal{T}'}^-(C_{R'}(0)) = \bigcup_{i' = 0}^{n'}\bigcup_{j=1}^M \bigcup_{k=1}^{\kappa_j^{(i)}}\bar{t}^{(i)}_{j,k},$$
where $n'(R) = n(R)-y$ is the shifted index and $\bar{t}_{j,n}^{(k_{i'})} =  \theta_{(y)_x} t_{j,n}^{(k_{i'})}$. Using this decomposition for $\mathcal{O}^-_{\mathcal{T}'}(C_{R'}(0))$:
\begin{equation*}
  \begin{split}
    &\theta^{d}_{(y)_x}\mathcal{S}_R^{\mathcal{T}}(h,\lambda) = \mathcal{S}_{R'}^{\mathcal{T}'}(h',\lambda') \\
    &\hspace{.6in}=\int_{\mathcal{O}^-_{\mathcal{T}'}(C_{R'}(0))}e^{-2\pi \imath \langle\lambda', \tau\rangle}\sum_{\ell=1}^M h'_\ell\circ \varphi_\tau(\mathcal{T}')\, d\tau+ \int_{B_{R'}\backslash \mathcal{O}^-_{\mathcal{T}'}(C_{R'}(0))}e^{-2\pi \imath \langle\lambda', \tau\rangle} h'\circ \varphi_\tau(\mathcal{T}')\, d\tau \\
    &\hspace{.4in}= \sum_{\ell=1}^M\int_{\mathcal{O}^-_{\mathcal{T}'}(C_{R'})}e^{-2\pi \imath \langle\lambda', \tau\rangle} h'_\ell\circ \varphi_\tau(\mathcal{T}')\, d\tau + \mathcal{O}(R^{d-1}) \\
    &\hspace{.2in}= \sum_{\ell=1}^M\sum_{i'=0}^{n'(R)}\sum_{j=1}^M\sum_{n=1}^{\kappa^{(k_{i'})}_j}\int_{\bar{t}^{(i')}_{j,n}}e^{-2\pi \imath \langle\lambda', \tau\rangle} h'_\ell\circ \varphi_\tau(\mathcal{T}')\, d\tau + \mathcal{O}(R^{d-1}) \\
    &= \sum_{i'=0}^{n'(R)}\sum_{j=1}^M\sum_{n=1}^{\kappa^{(i')}_j}  e^{2\pi \imath \langle \lambda' ,\tau(t^{(i')}_{j,n})\rangle} \sum_{\ell=1}^M   \mathfrak{t}^{x'}_{i',j}\left(a^{i'}_{x',\ell,\lambda'}  \right) \widehat{\psi_\ell^{\mathcal{T}'}}(\lambda')  + \mathcal{O}(R^{d-1}), 
  \end{split}
\end{equation*}
where the last line follows from (\ref{eqn:localTwist2}). Note that $|\widehat{\psi_\ell^{\mathcal{T}'}}(\lambda')|$ can be bounded by the $\|h\|_\infty$ times the maximal volume $C$ of a level-$y$ supertile. Using this, Lemma \ref{lem:packing} and Proposition \ref{prop:TracesBound},
\begin{equation}
\label{eqn:twistedBnd}
  \begin{split}
    & \theta^d_{(y)_x} \left|\mathcal{S}_R^\mathcal{T}(h,\lambda)\right|\leq \sum_{i'=0}^{n'(R)}\sum_{j=1}^M\kappa^{(i')}_j \sum_{\ell=1}^M\left|\mathfrak{t}^{x'}_{i',j}\left(a^{i'}_{x',\ell,\lambda'}  \right) \widehat{\psi_\ell^{\mathcal{T}'}}(\lambda')\right| + \mathcal{O}(R^{d-1}) \\
    &\hspace{1in}\leq C\|h\|_\infty \sum_{i'=0}^{n'(R)}\sum_{j=1}^M\kappa^{(i')}_jM \max_\ell\left\{\left|\mathfrak{t}^{x'}_{i',j}\left(a^{i'}_{x',\ell,\lambda'}  \right) \right|\right\} + \mathcal{O}(R^{d-1}) \\
    &\hspace{.75in}\leq C\|h\|_\infty M \sum_{i'=0}^{n'(R)} \RN (2R)^{d-1}\theta^{d-1}_{(i')_x} \max_{\ell,j}\left\{\left|\mathfrak{t}^{x'}_{i',j}\left(a^{i'}_{x',\ell,\lambda'}  \right) \right|\right\} + \mathcal{O}(R^{d-1}) \\
    &\hspace{.5in}\leq C\|h\|_\infty M 2^{d-1} \RN\sum_{i'=0}^{n'(R)}  R^{d-1}\theta^{d-1}_{(i')_x} \max_{\ell,j}\left\{\left|\mathfrak{t}^{x'}_{i',j}\left(a^{i'}_{x',\ell,\lambda'}  \right) \right|\right\} + \mathcal{O}(R^{d-1}) \\
    &\hspace{.25in}\leq C\|h\|_\infty M \RN \RU^{\frac{d-1}{d}} \sum_{i'=0}^{n'(R)}  \theta_{(n)_x}^{1-d}\theta^{d-1}_{(i')_x}\left\|\Theta^{(i')}_x\right\| \mathfrak{P}(0,N(i')) + \mathcal{O}(R^{d-1}),
  \end{split}
\end{equation}
where $N(i')$ is the largest return time to $C([w^-.w^+])$ less than or equal to $i'$, that is, it satisfies $k_{N(i')}\leq i' < k_{N(i')+1}$. Now the estimates (\ref{eqn:bdryEffect})-(\ref{eqn:bdryEffect3}) can be used:
\begin{equation}
  \label{eqn:twistedBnd'}
  \begin{split}
    &\theta^d_{(y)_x}\left|\mathcal{S}_R^\mathcal{T}(h,\lambda)\right|\\
    &\leq C\|h\|_\infty M \RN \RU^{\frac{d-1}{d}} C_\varepsilon K_\varepsilon^4\exp\left[\left(\frac{d-1}{d}\eta_1 +\frac{3d+2}{d}\varepsilon\right)n'\right]\sum_{i'=0}^{n'(R)} \mathfrak{P}(0,N(i')) e^{\frac{\eta_1}{d}i}   + \mathcal{O}(R^{d-1}) \\
    &\hspace{.3in}\leq C\|h\|_\infty M \RN \RU^{\frac{d-1}{d}} C_\varepsilon K_\varepsilon^4\Jh_\varepsilon e^{\frac{\eta_1+\varepsilon}{d}}         \exp\left[\left(\eta_1 +\frac{3(d+1)}{d}\varepsilon\right)n'\right]       \mathfrak{P}(0,N(n'))   + \mathcal{O}(R^{d-1})\\
    &\hspace{.6in}\leq C\|h\|_\infty M \RN \RU^{\frac{d-1}{d}} C_\varepsilon K_\varepsilon^4\Jh_\varepsilon e^{\frac{\eta_1+\varepsilon}{d}}\Ch_\varepsilon \theta_{(n)_x}^{-\left(d+\frac{3d+4}{\eta_1}\varepsilon\right)}\mathfrak{P}(0,N(n'))   + \mathcal{O}(R^{d-1})
  \end{split}
\end{equation}
  
Since $n(R) = n'+y$ is the largest level of supertile guaranteed to be completely contained in $C_R(0)$, $\min_i\{\mathrm{Vol}(t_i)\}\theta_{n(R)}^{-d} \leq \mathrm{Vol}(C_R(0)) = (2R)^d$. That is, there is a uniform constant $\RU'$ such that $\theta_{(n(R))_x}^{-d}\leq \RU' R^d$ for all $R>1$.
Putting this together with (\ref{eqn:twistedBnd'}),
\begin{equation}
\label{eqn:twistedBnd3}
  \begin{split}
 \left|\mathcal{S}_R^\mathcal{T}(h,\lambda)\right|\leq K_{x,h,\varepsilon} R^{d+\frac{3d+4}{\eta_1}\varepsilon}\mathfrak{P}(m_0,N(n(R))) + \mathcal{O}(R^{d-1})
  \end{split}
\end{equation}
for all $R>C_{\varepsilon, x}z^y$. Now, since $R^d\leq 2^{-d}\RU \theta^{-d}_{(n)_x}\leq 2^{-d}\RU K_\varepsilon e^{(\eta_1+\varepsilon)n}$, setting $\bar{K}_\varepsilon := \log(2^{-d}\RU K_\varepsilon)$ it follows that
$$\frac{d\log R - \bar{K}_\varepsilon}{\eta_1+\varepsilon}\leq n(R)$$
which, combined with (\ref{eqn:twistedBnd3}), finishes the proof.
\end{proof}

  Let
  $$\mathcal{A}_B := \{\lambda\in\mathbb{R}^d: B^{-1}\leq \|\lambda\|\leq B\}.$$
  The following is essentially \cite[Lemma 4.5]{BS:translation}.
  \begin{lemma}
    \label{lem:LipApprox}
  Let $S_1,\dots, S_N$ be $N$ compatible substitution rules on $M$ prototiles $t_1,\dots,t_M$, $\mu$ a minimal $\sigma$-invariant ergodic probability measure and suppose that the trace cocycle is integrable. For an Oseledets-regular $x\in\Sigma_N$ suppose that there is a $N_0,\vartheta>1$, and $\gamma\in(0,1)$ such that for any $B\geq 2$, for any $y > N_0 + \vartheta\log B$ and any bounded level-$y$ TLC function $f_y$ satisfies
\begin{equation}
  \label{eqn:LipApp1}
  \left| \mathcal{S}_R^\mathcal{T}(f_y,\lambda)   \right|\leq C_{x,f_y}R^{d-\gamma/2}
  \end{equation}
 for all $\lambda\in\mathcal{A}_B$ and $R>e^{\gamma^{-1}\eta_1 y}$, where $\eta_1$ is the largest Lyapunov exponent of the trace cocycle. Then for any Lipschitz function $f:\Omega_x\rightarrow \mathbb{R}$:
  \begin{equation}
    \label{eqn:LipHold}
    \left| \mathcal{S}_R^\mathcal{T}(f,\lambda)   \right|\leq C_{x,f}R^{d-\gamma/2}
    \end{equation}
 for all $\lambda\in\mathcal{A}_B$ and $R>C'e^{\gamma^{-1}\eta_1 y}$ for some $C'>1$.
\end{lemma}
\begin{proof}
  Let $f:\Omega_x\rightarrow \mathbb{R}$ be a Lipschitz function. By Lemma \ref{lem:LipApprox0}, for any $y\in\mathbb{N}$ there is a level-$y$ TLC function $f_y$ with
  $$| f(\mathcal{T})-f_y(\mathcal{T})|\leq L_{f}\max_{v\in\mathcal{V}_k}\mu^+_x(v)\hspace{.5in}\mbox{ and }\hspace{.5in}|f_y(\mathcal{T})|\leq |f(\mathcal{T})|$$
  on a full measure set of $\mathcal{T}$. Note that
  \begin{equation}
    \label{eqn:TransvContr}
    \lim_{n\rightarrow\infty}\frac{\log\mu^+_x(v_i\in\mathcal{V}_n)}{n} = -\eta_1
  \end{equation}
for any $i$. Indeed, since $x$ is Furstenberg-Kesten-regular for the trace cocycle and each renormalization map $\Phi_{\sigma^k(x)}$ along an orbit preserves the unique $\mathbb{R}^d$ invariant measure which is locally $\nu_x\times \mbox{Leb}$, the transverse measure $\nu_x = (\bar\Delta_x)_*\mu^+_x$ is contracted at a rate which is the inverse of that of the expansion of the Lebesgue measure along the leaves, which is $\theta^{-d}_{(n)_x}$. Therefore, since $n^{-1}\log \theta^{-d}_{(n)_x}\rightarrow \eta_1$, (\ref{eqn:TransvContr}) follows. Thus there is a $K_f>1$ such that for any $y>0$,
\begin{equation}
  \label{eqn:LipApp2}
  \|f-f_y\|_\infty \leq K_f e^{-\frac{\eta_1}{2}y}.
  \end{equation}
For $R>e^{\gamma^{-1}\eta_1 y}$, let
$$y_R = \left\lfloor\frac{\gamma\log R}{\eta_1}\right\rfloor$$
so that for $y> N_0+\vartheta \log B$ both (\ref{eqn:LipApp1}) and (\ref{eqn:LipApp2}) hold, and so:
 $$\left| \mathcal{S}_R^\mathcal{T}(f,\lambda)- \mathcal{S}_R^\mathcal{T}(f_y,\lambda)\right|\leq (2R)^dK_fe^{-\frac{\eta_1}{2}y_R} \leq K_f'R^{d-\gamma/2},$$
which combined with (\ref{eqn:LipApp1}) implies (\ref{eqn:LipHold}).
\end{proof}
\section{A quantitative Veech criterion}
\label{sec:criterion}
\begin{definition}
  Let $S_1,\dots, S_N$ be $N$ compatible substitution rules on the prototiles $t_1,\dots, t_M$. Suppose that there is a $\sigma$-invariant probability measure for which there exists a positively simple word $w = w^-w^+$ with the following properties:
  \begin{enumerate}
  \item $\mu(C([w^-.w^+]))>0$,
  \item for any $x_*\in C([w^-.w^+])\subset\Sigma_N$ there is an $\ell\in\{1,\dots,d\}$ such that for any canonical level-$|w^+|$ supertile there are $k$ vectors $\tau_1,\dots, \tau_k\in \Lambda_{x_*,\ell}^{(0)}$, describing translation equivalences of tiles of type $\ell$ in any $\mathcal{P}_{|w^+|}(v)$, with the property that $\alpha_{x_*}(\tau_1),\dots, \alpha_{x_*}(\tau_{k})$ generate $\mathbb{Z}^{r_x}$.
  \end{enumerate}
  Then $w$ is called a \textbf{postal word} and $\mu$ is called a \textbf{postal} measure, since there are return vectors $\tau_1,\dots,\tau_k$ inside any level-$|w^+|$ supertile $\mathcal{P}_{|w^+|}(v)$ which can find any address.
\end{definition}
The following is a consequence of \cite[Lemma 5.1]{BS:translation}.
\begin{lemma}
  \label{lem:Postal}
Let $S_1,\dots, S_N$ be $N$ compatible substitution rules on the prototiles $t_1,\dots, t_M$, and let $w = w^-w^+$ be a positively simple word. Assume that for any $x_*\in U_{i_*}^{r_x}\cap C([w^-.w^+])\subset\Sigma_N$ there is an $\ell\in\{1,\dots,d\}$ such that for any canonical level-$|w^+|$ supertile there are $k$ vectors $\tau_1,\dots, \tau_k\in \Lambda_{x_*,\ell}^{(0)}$, describing translation equivalences of tiles of type $\ell$ in any $\mathcal{P}_{|w^+|}(v)$, with the property that $\alpha_{x_*}(\tau_1),\dots, \alpha_{x_*}(\tau_{k})$ generate $\mathbb{Z}^{r_x}$. Then there exists a $C_{w}>0$ such that
  $$C_w \|z\|_{\mathbb{R}^{r_x}/\mathbb{Z}^{r_x}}\leq \max_{j\leq k} \left\|\left\langle\alpha_{x_*}(\tau_j),z\right\rangle\right\|_{\mathbb{R}/\mathbb{Z}}$$
  for any $z\in\mathbb{R}^{r_x}$.
\end{lemma}

Recall (\ref{eqn:vectorCocycle}), which gives, for $\lambda\in\mathbb{R}^d$ and $\tau\in\Lambda_{\sigma^{-n}(x)}$
\begin{equation}
  \begin{split}
    \left\langle \lambda, \theta_{(n)_x}^{-1}\tau\right\rangle &= \left\langle \lambda, \theta_{(-n)_{\sigma^{-n}(x)}}^{-1}\tau\right\rangle = \left\langle  \lambda,   V_{x} G^{(n)}_{\sigma^{-n}(x)}\alpha_{\sigma^{-n}(x)}(\tau)\right\rangle \\
    &= \left\langle  \alpha_{\sigma^{-n}(x)}(\tau),\bar{G}^{(n)}_x V_{x}^* \lambda \right\rangle 
  \end{split}
\end{equation}
for the corresponding some $d\times r_x$ matrix $V_{x}$ and the address vector $\alpha_{\sigma^{-n}(x)}(\tau)\in\mathbb{Z}^{r_x}$. If $k_n\rightarrow \infty$ is a sequence of return times for $x$ to a cylinder set $C([w])$ under $\sigma^{-1}$, then $V_{\sigma^{-k_n}(x)} = V$ for some fixed $V$. In view of the previous Lemma, the goal now is to bound the quantity in Proposition \ref{prop:twistedBound} by something of the form
\begin{equation}
  \label{eqn:oftheForm}
  \begin{split}
    &\prod_{y'<n<N_\varepsilon(R)} \left( 1-c_w \max_{\tau\in \Lambda_w}\left\| \left\langle\lambda, \theta^{-1}_{(k_n+|w^+|)_x}\tau\right\rangle\right\|^2_{\mathbb{R}/\mathbb{Z}}  \right) \\
    &\hspace{2.1in}\leq \prod_{y'<n<N_\varepsilon(R)} \left( 1-c_w' \left\| \bar{G}^{(k_n+|w^+|)}_x V_{x}^*  \lambda    \right\|^2_{\mathbb{R}^{r_x}/\mathbb{Z}^{r_x}}  \right).
  \end{split}
\end{equation}

Suppose $x$ is a $\mu$-generic and Oseledets- and Furstenberg-Kesten-regular point for the respective cocycles, where $\mu$ is a positively simple, postal, ergodic measure. Let $r_\mu$ be the rank of $\Gamma_x$ for $\mu$-almost every $x$. Recall that by \S \ref{subsec:RV} there is a collection of subsets $\{U_i^{r_\mu}\}$ of $\Sigma_N$ with the property that for each $i$ there is a collection of vectors $v_1,\dots, v_{r_\mu}$ such that $\Gamma_x$ is generated by those vectors for any $x\in U_i^{r_\mu}$. By the hypotheses on $\mu$, without loss of generality it can be assumed that there is an $i$ such that $U_i^{r_\mu} = C([w^-.w^+])$, where $w = w^-w^+$ is a positively simple word and postal. As such, $\mu(U_i^{r_\mu})>0$, and set $U_\star = U_i^{r_\mu}$.

Let $k_n\rightarrow \infty$, $n\in\mathbb{N}$, be a maximal sequence of return times to $U_\star$ for $x$ under $\sigma^{-1}$. Since there is a local choice of generators $v_1,\dots, v_{r_\mu}$ for all $\Gamma_x$ with $x\in U_\star$, there is a natural identification of all fibers $\mathfrak{R}_x$ and their lattices $\alpha(\Gamma_x)\subset \mathfrak{R}_x$ for all $x\in U_\star$, and in particular for $\sigma^{k_n}(x)$. Under this identification, the vector space is denoted by $\mathfrak{R}_\star$ and the corresponding lattice is $\Gamma_\star\subset \mathfrak{R}_\star$. This also holds for the dual spaces and lattices $\Gamma_\star^*\subset \mathfrak{R}_\star^*$. As such, for $\rho\in(0,1/2)$, let $B_\rho(\Gamma_\star^*)$ be the $\rho$-neighborhood of the lattice $\Gamma^*_\star\subset\mathfrak{R}^*_\star$.

\begin{proposition}[Quantitative Veech criterion]
  \label{prop:Veech}
  Let $S_1,\dots, S_N$ be a collection of compatible substitution rules on tiles $t_1,\dots, t_M$ and $\mu$ a positively simple, postal, $\sigma$-invariant ergodic probability measure for which the trace and return vector cocycles are integrable. For any positively simple postal word $w = w^-w^+$ with $\mu(C([w]))>0$, for any Oseledets- and Furstenberg-Kesten-regular $x\in\Sigma_N$, if there exists $\rho,\delta\in (0,\frac{1}{2})$ and $\vartheta>1$ such that for any $B\geq 2$
  \begin{equation}
    \label{eqn:densityVeech}
    \frac{\displaystyle\sum_{j=1}^{N}\mathbbm{1}_{C([w^-.w^+])}\left(\sigma^{-j}(x)\right)\mathbbm{1}_{B_\rho(\Gamma^*_\star)}\left( \bar{G}^{(j)}_x V_{x}^*  \lambda\right)}{\displaystyle\sum_{j=1}^{N}\mathbbm{1}_{C([w^-.w^+])}(\sigma^{-j}(x))}<1-\delta
  \end{equation}
  for all $N>N_0(x)+\vartheta \log B$ and $\lambda \in \mathcal{A}_B$, then there exists a $\gamma>0$ such that for any Lipschitz function $f:\Omega_x\rightarrow \mathbb{R}$,
  $$\left|\int_{C_R(0)}e^{-2\pi \imath \langle \lambda, t\rangle} f\circ \varphi_t(\mathcal{T})\, dt\right|\leq C_{x,f} R^{d-\gamma}+\mathcal{O}(R^{d-1})$$
  for all $R>N_0^*(x)B^\vartheta$. Moreover, the lower local dimension of the spectral measure $\mu_f$ satisfies
  $$2\min\left \{1,- \frac{d\delta\log(1-c'_w\rho)\mu(C([w]))}{\eta_1}\right\}\leq d^-_f(\lambda),$$
  where $c'_w$ comes from (\ref{eqn:oftheForm}).
\end{proposition}
This type of theorem goes by quantitative Veech criterion because it is a quantitative version of Veech's criterion in \cite{Veech:metric1}. There are other qualitative criteria for eigenvalues which are in the spirit to Veech's, e.g \cite{CDHM:eigenfunctions, BDM:eigenvalues}. The criterion here is a generalization to higher rank of the criterion in \cite{BS:translation}.
\begin{proof}
  Let $x\in\Sigma_N$ be an $\mu$-typical, Furstenberg-Kesten-regular point such that the conclusions of Proposition \ref{prop:twistedBound} hold. Since $\mu$ is ergodic, for all $\varepsilon>0$ there exists a $N_*(x)>0$ such that for all $n>N_*(x)$,
  $$\frac{n}{k_n}< \mu(C_w) + \varepsilon$$
  where $C_w:= C([w^-.w^+])$. Denote $\mu_w = \mu(C_w)$. By Lemma \ref{lem:LipApprox} it suffices to check (\ref{eqn:LipApp1}) for all level-$y$ TLC functions for $y>N_0+\vartheta \log B$.
  
  Let $f:\Omega_x\rightarrow \mathbb{R}$ be a level-$y$ TLC function, $\mathcal{T}\in\Omega_x$ and $\lambda\in\mathcal{A}_B$. Given $\varepsilon>0$, pick $\gamma>0$ small enough such
  $$\gamma<-\frac{d\delta\log(1-c'_w\rho^2)(\mu_w+\varepsilon)}{\eta_1+\varepsilon} +\frac{d}{\eta_1}\varepsilon$$
  and $e^{\gamma^{-1}\eta_1y}> z^{\max\{y,k_{N^*},N_*\}+N_0+\vartheta\log B}$, where $z>1$ is given by Proposition \ref{prop:twistedBound}. Thus for $R>e^{\gamma^{-1}\eta_1y}>z^{\max\{y,k_{N^*},N_*\}+N_0+\vartheta\log B}$, 
  \begin{equation}
    \label{eqn:VeechBnd}
    \begin{split}
      &|\mathcal{S}_R^\mathcal{T}(f,\lambda)|\leq C_{x,f,\varepsilon,y} R^{d+\varepsilon} \prod_{m_0<n<N_\varepsilon(R)} \left( 1-c_w \max_{\tau\in \Lambda_w}\left\| \left\langle\lambda, \theta^{-1}_{(k_n)_x}\tau\right\rangle\right\|^2_{\mathbb{R}/\mathbb{Z}}  \right) +\mathcal{O}(R^{d-1})\\
      &\hspace{1in}\leq C_{x,f,\varepsilon} R^{d+\varepsilon}\left(1-c_w'\rho^2\right)^{\delta N_{\varepsilon}(R)-m_0-2} +\mathcal{O}(R^{d-1}) \\
      &\hspace{.8in}\leq C'_{x,f,\varepsilon} R^{d+\varepsilon}\left(1-c_w'\rho^2\right)^{\delta  (\mu_w+\varepsilon) \left(\frac{d\log R -  \bar{K}_{\varepsilon,x}}{\eta_1+\varepsilon}\right)} +\mathcal{O}(R^{d-1}) \\
      &\hspace{.6in}= C'_{x,f,\varepsilon} R^{d+\varepsilon} R^{\frac{\delta d\log(1-c'_w\rho^2)(\mu_w+\varepsilon)}{\eta_1+\varepsilon}\left(1 - \frac{\bar{K}_{\varepsilon,x}}{d\log R}\right)} +\mathcal{O}(R^{d-1}) \\
      &\hspace{.4in}\leq C'_{x,f,\varepsilon} R^{d+\frac{d\delta\log(1-c'_w\rho^2)(\mu_w+\varepsilon)}{\eta_1+\varepsilon} +\frac{d}{\eta_1}\varepsilon} +\mathcal{O}(R^{d-1})\\
      &\hspace{.4in}\leq C'_{x,f} R^{d-\gamma/2} +\mathcal{O}(R^{d-1})
    \end{split}
  \end{equation}
  where (\ref{eqn:oftheForm}) was used to obtain the second inequality. Now if
  $$-\frac{d\delta\log(1-c'_w\rho^2)(\mu_w+\varepsilon)}{\eta_1+\varepsilon} > 1,$$
then the $\mathcal{O}(R^{d-1})$ term could dominate in (\ref{eqn:VeechBnd}) and so for any $\varepsilon>0$ the growth of the twisted integral is bounded by $D_\varepsilon R^{d-1+\varepsilon}$ for some $D_\varepsilon$. By Lemma \ref{lem:dimension}, since the $L^\infty$ norm dominates the $L^2$ norm, it follows from (\ref{eqn:VeechBnd}) that
  $$2\min\left\{ 1-\varepsilon, -\frac{d\delta\log(1-c'_w\rho^2)(\mu_w+\varepsilon)}{\eta_1+\varepsilon} - \frac{d}{\eta_1}\varepsilon\right\} \leq d_f^-(\lambda)$$
  for any $\varepsilon>0$.
\end{proof}

\section{Cohomology and deformations of tiling spaces}
\label{sec:cohomology}
Let $\mathcal{T}$ be an aperiodic, repetitive tiling of $\mathbb{R}^d$ of finite local complexity. This section reviews several aspects of cohomology for tiling spaces. See \cite{sadun:book} for a comprehensive introduction.
\subsection{Anderson-Putnam complexes}
\label{subsec:AP}
For any tile $t$ in the tiling $\mathcal{T}\in\Omega$, the set $\mathcal{T}(t)$ denotes all tiles in $\mathcal{T}$ which intersect $t$. This type of patch is called a \textbf{collared tile}.
\begin{definition}
  \label{def:AP}
  Let $\Omega$ be a tiling space. Consider the space $\Omega\times \mathbb{R}^d$ under the product topology, where $\Omega$ carries the discrete topology and $\mathbb{R}^d$ the usual topology. Let $\sim_1$ be the equivalence relation on $\Omega\times\mathbb{R}^d$ which declares a pair $(\mathcal{T}_1,u_1)\sim_1 (\mathcal{T}_2,u_2)$ if $\mathcal{T}_1(t_1)-u_1 = \mathcal{T}_2(t_2)-u_2$ for some tiles $t_1,t_2$ with $u_1\in t_1\in\mathcal{T}_1$ and $u_2\in t_2\in\mathcal{T}_2$. The space $(\Omega\times\mathbb{R}^d)/\sim_1$ is called the \textbf{Anderson-Putnam (AP) complex} of $\Omega$ and is denoted by $AP(\Omega)$.
\end{definition}
Let us now review the AP-complexes involved in our construction. First, note that for $x  = (x_1,x_2,\dots)\in \Sigma_N^+$, assuming that $B_x$ is minimal, $AP(\Omega_x)$ only depends on finitely many symbols $x_1,\dots, x_\ell$, since the collaring of tiles in tilings of $\Omega_x$ only depends in the $k^{th}$-approximants $\mathcal{P}_k$ for sufficiently large $k$. Thus for a set of compatible substitution rules $S_1,\dots, S_N$, there exists a non-empty collection of clopen subsets $\{U_i\}_i$ of $\Sigma_N$ and CW-complexes $\{\Gamma_i\}_i$ such that if $x\in U_i$ and $B_x$ is minimal, then $AP(\Omega_x) = \Gamma_i$.

It will be useful to also consider higher level AP complexes, defined as follows. For $x\in \Sigma_N$, $\mathcal{T}\in\Omega_x$, and denoting by $t^{(1)}$ a level 1 supertile (which are approximants of the form $\mathcal{P}_1(\bar{e})$ as in (\ref{eqn:approximant})), let $\mathcal{T}(t^{(1)})$ be the union of the supertile $t^{(1)}$ and the level 1 supertiles of $\mathcal{T}$ which intersect $t^{(1)}$. Proceeding similarly, we let $\mathcal{T}(t^{(n)})$ denote the \textbf{collared level-$n$ supertile} corresponding to the level-$n$ supertile $t^{(n)}$. Let $\sim_n$ be the equivalence relation defined by the collared level-$n$ supertiles $\mathcal{T}(t^{(n)})$ in $\Omega$ as in Definition \ref{def:AP}. The quotient $(\Omega\times\mathbb{R}^d)/\sim_n$ is denoted by $AP^n(\Omega)$, and the projection is denoted by $\pi_{\mathcal{T},k}:\mathbb{R}^d\rightarrow AP^k(\Omega)$, although the dependence on $\mathcal{T}$ will sometimes be supressed. By construction, $AP(\Omega_{\sigma^{-n}(x)})$ and $AP^n(\Omega_x)$ are homeomorphic and they are homeomorphic through scaling: $AP^n(\Omega_x)$ is a rescaling of $AP(\Omega_{\sigma^n(x)})$ by $\theta_{x_n}^{-1}\cdots \theta_{x_1}^{-1}$. Denote by
\begin{equation}
  \label{eqn:rescale}
  r_{k,x}:AP^k(\Omega_x)\rightarrow AP(\Omega_{\sigma^{-k}(x)})
\end{equation}
the rescaling homeomorphisms. The following summarizes well-known consequences of the ideas in \cite[\S 4- \S 6]{AP}.
\begin{proposition}
  \label{prop:cohProp}
Let $S_1,\dots, S_N$ be a collection of uniformly expanding compatible substitution rules on the set of prototiles $t_1,\dots, t_M$ and suppose $\mathcal{B}_x$ is minimal for some $x\in\Sigma_N$.
  \begin{enumerate}
  \item The substitution rule $S_{x_i}$ induces a continuous map $\gamma_i:AP(\Omega_{\sigma^{-i}(x)})\rightarrow AP(\Omega_{\sigma^{-i+1}(x)})$ defined by $\gamma_i(\mathcal{T},u) = (S_{x_i}(\mathcal{T}),\theta_{x_i}^{-1}u)$ for all $i>0$.
  \item The tiling space $\Omega_x$ is an inverse limit of maps induced by the sequence of substitutions:
    \begin{equation}
      \label{eqn:invLim}
      \Omega_{x} = \varprojlim_{\gamma_i} AP(\Omega_{\sigma^{-i}(x)}).
    \end{equation}
  \item The \v Cech cohomology groups of $\Omega_{x}$ are obtained through the direct limits
    \begin{equation}
      \label{eqn:directLimit}
      \check{H}^i(\Omega_x;\mathbb{Z}) = \varinjlim_{\gamma_k^*}\check{H}^i(AP(\Omega_{\sigma^{-k}(x)});\mathbb{Z}).
    \end{equation}
  \end{enumerate}
\end{proposition}
\subsection{Pattern-equivariant cohomology}
Recall $\mathcal{T}$-equivariant functions from \S \ref{subsubsec:functions}. The set of $C^\infty$ $\mathcal{T}$-equivariant functions is denoted by $\Delta^0_\mathcal{T}$. By (\ref{eqn:invLim}), the set $\Delta_\mathcal{T}^0$ is the union of the pullback of smooth functions on the AP complexes through $\pi_{\mathcal{T},k}$ \cite[Theorem 5.4]{sadun:book}. More generally, a smooth $k$-form $\eta$ on $\mathbb{R}^d$ is $\mathcal{T}$-equivariant if all the functions involved are $\mathcal{T}$-equivariant. Let $\Delta_\mathcal{T}^k$ denote the set of $C^\infty$, $\mathcal{T}$-equivariant $k$-forms on $\mathbb{R}^d$. The usual differential operator of differential forms $d$ takes $\mathcal{T}$-equivariant forms to $\mathcal{T}$-equivariant forms. So $\{\Delta^k_\mathcal{T}\}_k$ is a subcomplex of the de-Rham complex of smooth differential forms in $\mathbb{R}^d$.
\begin{definition}
  The \textbf{pattern-equivariant cohomology} is defined as the quotient
  $$H^k(\Omega_\mathcal{T};\mathbb{R})=\frac{\mathrm{ker}\, d:\Delta^k_\mathcal{T}\rightarrow \Delta^{k+1}_\mathcal{T} }{\mathrm{Im}\, d:\Delta^{k-1}_\mathcal{T}\rightarrow \Delta^{k}_\mathcal{T}}.$$
\end{definition}
The following result is a combination of Theorems 20 and 23 in \cite{KP:RS}.
\begin{thm}[\cite{KP:RS}]
Let $\mathcal{T}$ be an aperiodic, repetitive tiling of $\mathbb{R}^d$ of finite local complexity. The \v Cech cohomology $\check H^*(\Omega_\mathcal{T};\mathbb{R})$ is isomorphic to the pattern-equivariant cohomology $H^*(\Omega_\mathcal{T};\mathbb{R})$.
\end{thm}
\subsection{Return vectors and cohomology}
Recall that $\Gamma_x^*$ denotes the dual of the abelian group generated by the finitely many return basis vectors of tilings in $\Omega_x$ (see \S \ref{subsec:RV}).
\label{subsec:RetVecCoh}
\begin{proposition}
  There is a homomorphism
  \begin{equation}
  \label{eqn:vectorMap}
  \beta_x:\check H^1(\Omega_x;\mathbb{Z})\rightarrow \Gamma_x^*
\end{equation}
  making the return vector cocycle equivariant with the induced action on cohomology by the homeomorphism $\Phi^*_{\sigma^{-1}(x)}$:
  \begin{equation}
    \label{eqn:betaEquiv}
  \begin{tikzcd}
    H^1(\Omega_x;\mathbb{R})\arrow{r}{\Phi_{\sigma^{-1}(x)}^*} \arrow{d}{\beta_x} & H^1(\Omega_{\sigma^{-1}(x)};\mathbb{R}) \arrow{d}{\beta_{\sigma^{-1}(x)}}\\
    \mathfrak{R}_x^* \arrow[swap]{r}{\bar{G}_x}  & \mathfrak{R}^*_{\sigma^{-1}(x)} .
  \end{tikzcd}
  \end{equation}
\end{proposition}
As such, the Oseledets decompositions of $\mathfrak{R}_x = Y^-_x\oplus Y^+_x$ and $H^1(\Omega_x;\mathbb{R}) = Z^-_x\oplus Z^+_x$ satisfy $\beta_xZ^\pm_x\subset Y^\pm_x$, where $Y^+_x$ and $Z^+_x$ denote elements with positive Lyapunov exponents.
\begin{proof}
  First, there is a pairing $\Gamma_x\times H^1(AP(\Omega_x);\mathbb{Z})\rightarrow \mathbb{Z}$ described as follows. For every element $v$ generated by return vectors and $c \in H^1(AP(\Omega_x);\mathbb{Z})$ the pairing comes from the evaluation $c (\pi(v))$, where $\pi(v)$ is the cycle defined by $v$ by projection to $AP(\Omega_x)$. If $v$ is large enough, then the evaluation can be done through a class in $H^1(AP^k(\Omega_x);\mathbb{Z})$ for some $k$ large enough. Thus there exists a map $\beta_x:H^1(\Omega_x;\mathbb{Z})\rightarrow \Gamma_x^*$ satisfying $c(\pi(v)) = \langle v,\beta_x(c)\rangle$ for all $v\in \Gamma_x$.  Finally,
  \begin{equation*}
    \begin{split}
      \left\langle\bar{G}_x \beta_x(c),v\right\rangle &= \left\langle \beta_x(c),G_x v\right\rangle =  \left\langle \beta_x(c),\theta^{-1}_{x_{-1}} v\right\rangle = c\left( \Phi_{\sigma^{-1}(x)*}\pi(v)\right) \\
      &=  \left(\Phi_{\sigma^{-1}(x)}^*c\right) \left( \pi(v)\right) =  \left\langle \beta_{\sigma^{-1}(x)}\left( \Phi^*_{\sigma^{-1}(x)} c\right), v \right\rangle,
    \end{split}
  \end{equation*}
  for any $v\in \Gamma_x$, which proves the equivariance of $\beta_x$.
\end{proof}

\subsection{Deformations}
\label{subsec:deform}
Given an polyhedral, repetitive tiling $\mathcal{T}$ of finite local complexity, one can obtain a family of tilings $\{\mathcal{T}_{\Rd}\}_{_{\Rd}}$ which are close to $\mathcal{T}$ in some geometric sense. More precisely, the tiling $\mathcal{T}_{\Rd}$ is obtained from $\mathcal{T}$ by a deformation given by $\Rd$. What needs to be established is what the ranges are for the deformation parameter $\Rd$.

The basic idea to understand deformations of tilings $\mathcal{T}$ is to consider the geometric information contained in the geometric object $AP(\Omega_\mathcal{T})$. Not only does it contain information about the topology of the tiling space, it also contains basic geometric information, namely, the lengths and shapes of the edges of the tiles. More precisely, any edge of a tile in $\mathcal{T}$ projects to a 1-chain in $AP(\Omega_\mathcal{T})$, and the vector defined by such edge in $\mathcal{T}$ can be thought of as the value of a $\mathbb{R}^d$-valued 1-cochain evaluated on this edge. One can similarly assign a $\mathbb{R}^d$-valued 1-cochain to every edge in $AP(\Omega_\mathcal{T})$. To be consistent, one must impose that the values of such cochains sum up to zero around the edge of a 2-cell, i.e., that these $\mathbb{R}^d$-valued cochains are in fact $\mathbb{R}^d$-valued cocycles. As such, the geometry of the tiles in a tiling $\mathcal{T}$ is determined by the representative $g_0$ of an element $[g_0]\in H^1(AP(\Omega_\mathcal{T});\mathbb{R}^d)$. Classes $[\Rd]\in H^1(AP(\Omega_\mathcal{T});\mathbb{R}^d)$ with representatives $\Rd$ which are close enough to $g_0$ gives a family of deformed tilings $\mathcal{T}_{\Rd}$ which are geometrically close to $\mathcal{T}$, all of which have the same combinatorial properties and many of which have tiling spaces $\Omega_{\mathcal{T}_{\Rd}}$ which share similar dynamic properties as those of $\Omega_{\mathcal{T}}$.

One can deform further still by considering $\mathbb{R}^d$-valued 1-cochains in $AP^k(\Omega_\mathcal{T})$. The geometry of level-$k$ supertiles of $\mathcal{T}$ can be deformed by the same method, and so there is a representative $g_k$ an element $[g_k]\in H^1(AP^k(\Omega_\mathcal{T});\mathbb{R}^d)$ which encodes the geometry of level-$k$ supertiles for all tilings in $\Omega_\mathcal{T}$. Thus, a representative $\Rd$ of $[\Rd]\in H^1(AP^k(\Omega_\mathcal{T});\mathbb{R}^d)$ which is close enough to $g_k$ defines a deformation of tilings in $\Omega_{\mathcal{T}}$. The conclusion of all this is that, given (\ref{eqn:rescale}) and (\ref{eqn:directLimit}), the geometry of tilings in $\Omega_\mathcal{T}$ is defined by a representative of an element $[g_\mathcal{T}]\in H^1(\Omega_\mathcal{T};\mathbb{R}^d)=H^1(\Omega_\mathcal{T};\mathbb{Z})\otimes \mathbb{R}^d$ and so any $[\Rd]\in H^1(\Omega_\mathcal{T};\mathbb{R}^d)$ with representative $\Rd$ close enough to $g_\mathcal{T}$ defines a new family of tilings $\Omega^{\Rd}$. Questions about which sort of properties persist under which types of deformations have been studied before in \cite{ClarkSadun:shape, kellendonk:deformation, JulienSadun:deformations}. The goal here will be to deduce under which types of deformations do the systems become weakly mixing.

Going back to the random construction, let $S_1,\dots, S_N$ be a family of uniformly expanding compatible substitution rules on the prototiles $t_1,\dots, t_M$ and suppose $\mathcal{B}_x$ is minimal. Let $[\Rd_x]\in H^1(\Omega_x;\mathbb{R}^d)$ be such that the representative $\Rd_x$ defines the geometry of tilings in $\Omega_x$. Following \cite[\S 8]{JulienSadun:deformations}, for any $[\alpha]$ close enough to $[\Rd_x]$ and $\mathcal{T}\in\Omega_x$, a new deformed tiling $\mathcal{T}'$ is constructed as follows. Let $\alpha$ be a $\mathcal{T}$-equivariant representative of the class $[\alpha]$, that is, a $\mathcal{T}$-equivariant choice of linear transformations $\alpha(t):\mathbb{R}^d\rightarrow \mathbb{R}^d$.
Define $H_\alpha:\mathbb{R}^d\rightarrow \mathbb{R}^d$ by $H_\alpha(s) := \int_{0}^s\alpha$ and assume for the moment that this is a homeomorphism.
This produces a new tiling $\mathcal{T}_{\alpha}$ whose patches are defined from patches in $\mathcal{T}$ by deforming each tile in $\mathcal{T}$ using $H_\alpha$. Taking the orbit closure of this tiling one obtains $\Omega_x^{\alpha}$, the tiling space defined by this deformation. Let $\Lambda_\mathcal{T}$ be the vertex set of $\mathcal{T}$ and consider $\Lambda_\mathcal{T}^{\alpha}$ obtained by applying $H_\alpha$ to each element of $\Lambda_\mathcal{T}$. If $v_1, v_2\in\Lambda_\mathcal{T}$ are two vertices in $\mathcal{T}$ which are joined by an edge, then the vertices $H_\alpha(v_1),H_\alpha(v_2)\in\Lambda^{\alpha}_\mathcal{T}$ are joined by an edge defined by the vector $\int_{v_1}^{v_2}\alpha$.  Furthermore, by construction, if $\mathcal{C}_\mathcal{T}$ is the set of control points of $\mathcal{T}$, then $\mathcal{C}_\mathcal{T}^{\alpha}:= H_{\alpha}(\mathcal{C}_\mathcal{T})$ is the set of control points of $\mathcal{T}_{\alpha}$.

The set of classes $[\alpha]\in H^1(\Omega_x;\mathbb{R}^d)$ for which there is a well-defined deformation as above, that is, for which there is a representative $\alpha$ such that $H_{\alpha}$ is a homeomorphism, are classified in \cite[\S 7-\S 8]{JulienSadun:deformations} as follows. Let $C_\mu:H^1(\Omega_x;\mathbb{R}^d)\rightarrow M_{d\times d}$ be the asymptotic cycle defined by averaging representatives of classes:
$$C_\mu:[\alpha]\mapsto C_\mu([\alpha]):=\lim_{R\rightarrow \infty}\frac{1}{\mathrm{Vol}\, B_R(0)}\int_{B_R(0)}\alpha(t)\, dt,$$
for $[\alpha]\in H^1(\Omega_x;\mathbb{R}^d)$. That $C_\mu$ is independent of representative follows immediately from the fact that the limit in the definition when $\alpha = d\omega$ is zero. This cycle goes also by the name of the Ruelle-Sullivan map, and in general it depends on a choice of invariant measure, but in all cases in this paper there is no other choice but $\mu$. Then $\Omega_x^{\alpha}$ is well-defined through the deformation $H_{\alpha}$ as above for some representative $\alpha$ if $C_\mu([\alpha])$ is invertible \cite[Theorem 8.1]{JulienSadun:deformations}. A representative $\Rd$ of $[\Rd]$ will be called a \textbf{good representative} if $H_{\Rd}$ is a homeomorphism.
Thus the choice of classes of deformation parameters $\Rd$ will be restricted to the open submanifold of $H^1(\Omega_x;\mathbb{R}^d)$ defined by
\begin{equation}
  \label{eqn:moduli}
  \mathcal{M}_x = \left\{ [\Rd]\in H^1(\Omega_x;\mathbb{R}^d):\mathrm{det}\,C_\mu([\Rd])\neq 0 \right\}.
\end{equation}
This set can be seen as an open submanifold as follows. The geometry of the tiles in the tilings of $\Omega_x$ define a canonical class $[\Rd]_*\in\mathcal{M}_x$, and since $[\Rd]\mapsto \mathrm{det}\,C_\mu([\Rd])$ is a composition of continuous maps, there is a neighborhood of $[\Rd]_*$ for which $\mathrm{det}\,C_\mu(\cdot)$ is nonzero. So $\mathcal{M}_x$ carries a natural absolutely continuous measure inherited from the Lebesgue measure on the vector space $H^1(\Omega_x;\mathbb{R}^d)$. Note immediately that $\mathcal{M}_x$ can be seen to contain a copy of $GL(d,\mathbb{R})$, which are the classes admitting constant representatives which are invertible. These are tiling spaces whose elements are tilings of the form $A\cdot \mathcal{T}$, where $A\in GL(d,\mathbb{R})$ and $\mathcal{T}\in \Omega_x$.
As such, $[\Rd]$ can be identified with the element $[\Rd] \in \mathrm{Hom}(\mathbb{R}^d, H^1(\Omega_x;\mathbb{R}))$ as follows. By definition, an element $[\Rd]\in H^1(\Omega_x;\mathbb{R}^d)$ can be written as $\sum_i a_ic_i\otimes v_i$ with $\{c_i\}\subset H^1(\Omega_x;\mathbb{R})$ is a basis, $a_i\in\mathbb{R}$ and $v_i\in\mathbb{R}^d$ for all $i$, and this defines a map $[\Rd]:\lambda\mapsto [\Rd](\lambda):= \sum_i a_i c_i\langle v_i,\lambda\rangle \in H^1(\Omega_x;\mathbb{R})$ for any $\lambda\in\mathbb{R}^d$, which is a homomorphism.
From this point of view the deformation class $[\Rd]$ also parametrizes inclusions\footnote{A dimension count helps: the Grassmanian $G_d(H^1(\Omega_x;\mathbb{R}))$ has dimension $d(\mathrm{dim}\,H^1(\Omega_x;\mathbb{R})- d) = \mathrm{dim}\,H^1(\Omega_x;\mathbb{R}^d)- d^2 $, so an element $H^1(\Omega_x;\mathbb{R}^d)$ defines first a choice of $d$-dimensional subspace of $H^1(\Omega_x;\mathbb{R})$ (parametrized by the Grassmanian), and then a linear transformation inside this subspace, given by a $d\times d$ matrix.} of $\mathbb{R}^d$ into $H^1(\Omega_x;\mathbb{R})$. In \cite{ClarkSadun:shape}, $[\Rd]$ is called the \emph{shape vector}.

It should be noted that the property of unique ergodicity is not altered through any deformations given by good representatives of $[\Rd]\in \mathcal{M}_x$ and also that the renormalization maps $\Phi_x$ satisfy $\Phi_x(\Omega_x^{\Rd}) = \Omega_{\sigma(x)}^{\Rd_x^1}$ where $[\Rd_x^1] = (\Phi_x^{-1})^*[\Rd]$ and, more generally,
\begin{equation}
  \label{eqn:equivDeform}
  \Phi_x^{(n)}(\Omega_x^{\Rd}) = \Omega_{\sigma^n(x)}^{\Rd_x^n}
\end{equation}
where $[\Rd_x^n] = (\Phi_x^{(n)*})^{-1}[\Rd]$.

Let $v_1,\dots, v_{r_x}\in \Lambda_x$ be a set of generators for $\Gamma_x$. Putting the $v_i$ as columns, there is a $d\times r_x$ matrix $V_{x}$ such that any return vector $\tau$ satisfies $\tau = V_{x}\alpha_x(\tau)$ for a unique $\alpha_x(\tau)\in\mathbb{Z}^{r_x}$, which in particular identifies $V_x$ as an element of $\mathbb{Z}^{r_x *}\otimes \mathbb{R}^d$. As such, the deformation $\Rd_x\mapsto \Rd$ induces a deformation of $V_{x}\mapsto V^{\Rd}_{x}\in \mathbb{Z}^{r_x *}\otimes \mathbb{R}^d$ through the map $\beta_x$ in (\ref{eqn:vectorMap}) as follows. If $[\Rd] = \sum_{i=1}^{\beta_1^x} \alpha_ic_i\otimes v_i\in H^1(\Omega_x;\mathbb{R}^d)$ for a basis $\{c_i\}$ of $H^1(\Omega_x;\mathbb{Z})$, $\alpha_i\in\mathbb{R}$ and $v_i\in\mathbb{R}^d$ for each $i$, then $\bar\beta_x([\Rd]) = \sum_{i=1}^{\beta_1^x} \alpha_i \beta_x(c_i)\otimes v_i = V^{\Rd}_x \in \mathbb{Z}^{r_x *}\otimes \mathbb{R}^d$.


It follows that for $[\Rd]\in \mathcal{M}_x$ with good representative $\Rd$ and $\tau\in \Lambda_x^{(0)}$ there is a corresponding vector $\tau^{\Rd}\in\Lambda^{(0)}_{x,\Rd}$ and thus for a generating set $\{\tau_i\}$ of $\Gamma_x$ there is a deformed generating set $\{\tau_i^{\Rd}\}$ of $\Gamma_x^{\Rd}$ and a corresponding matrix $V^{\Rd}_x$ (whose columns are the generating set of return vectors of the deformed tiling) satisfying $\tau^{\Rd} = V^{\Rd}_x \alpha_x(\tau)$ whenever $\tau = V_x \alpha_x(\tau)$. As such, recalling (\ref{eqn:equivDeform}), the deformed version of the return vector cocycle (\ref{eqn:vectorCocycle}) is
\begin{equation}
  \label{eqn:vectorCocycle2}
  \begin{split}
  &\theta_{(-n)_x}^{-1}V_{x}^{\Rd} =  V_{\sigma^n(x)}^{\Rd_x^n}G_{\sigma^{(n-1)}(x)}\cdots G_{\sigma(x)} G_x, \hspace{.5in}\mbox{ and thus } \\
&\theta_{(-n)_x}^{-1}\tau^{\Rd} = \theta_{(-n)_x}^{-1} V_{x}^{\Rd}\alpha_x(\tau) =  V_{\sigma^n(x)}^{\Rd^n_x} G_x^{(n)} \alpha_x(\tau) = V_{\sigma^n(x)}^{\Rd^n_x} \alpha_{\sigma^n(x)}\left(\theta_{(-n)_x}^{-1}\tau^{\Rd}\right).
  \end{split}
\end{equation}

Let $S_1,\dots, S_N$ be a collection of compatible substitution rules on tiles $t_1,\dots, t_M$ and suppose $\mu$ is a $\sigma-$invariant, ergodic probability measure on $\Sigma_N$ such that $\log\|\Phi_x^*\|\in L^1$, where $\Phi_x^*$ is the induced maps on $H^1(\Omega_x;\mathbb{R})$. Oseledets theorem gives a decomposition $H^1(\Omega_x;\mathbb{R}) = Z^+_x\oplus Z^-_x$, where $Z^+_x$ is the Oseledets space corresponding to positive Lyapunov exponents. Consider now the action $\Phi_{\sigma^{-1}(x)}^*:H^1(\Omega_x;\mathbb{R}^d)\rightarrow H^1(\Omega_{\sigma^{-1}(x)};\mathbb{R}^d)$. Since $\Phi_{\sigma^{-1}(x)}^*$ acts on the first component of $[\eta]\otimes v \in H^1(\Omega_x;\mathbb{R}^d) = H^1(\Omega_x;\mathbb{R})\otimes \mathbb{R}^d$, it follows that if
$$\frac{\log\|\Phi_x^*[\eta]\|}{n}\longrightarrow \omega_*\hspace{.6in}\mbox{ then }\hspace{.6in}\frac{\log\|\Phi_x^*([\eta]\otimes e_i)\|}{n}\longrightarrow \omega_*$$
for all basis vectors $e_i\in\mathbb{R}^d$. Thus there is an Oseledets decomposition of $H^1(\Omega_x;\mathbb{R}^d)$ as
$$H^1(\Omega_x;\mathbb{R}^d) = \hat{Z}^-_x\oplus \hat{Z}^+_x = Z^-_x \otimes \mathbb{R}^d\oplus Z^+_x\otimes \mathbb{R}^d $$
obtained from the Oseledets decomposition of $H^1(\Omega_x;\mathbb{R})$. Note that $\mathrm{dim}\, \hat{Z}^+_x \geq d^2$, since $Z^+_x$ is always at least $d$-dimensional given that the constant forms $dx_1,\dots, dx_d$ represent $d$-linearly independent elements of $Z^+_x$. Finally, note that the inclusions $[\Rd]:\mathbb{R}^d\rightarrow H^1(\Omega_x;\mathbb{R})$ defined by $\lambda\mapsto [\Rd](\lambda)$ as above are equivariant with respect to the induced maps on cohomology: if $[\Rd] = \sum a_ic_i\otimes v_i$, then
$$(\Phi_x^*[\Rd])(\lambda) =  \sum_i a_i(\Phi_x^*c_i)\langle v_i,\lambda \rangle .$$
\section{A cohomological quantitative Veech criterion}
\label{sec:CQVC}
Prior to this section, especially in \S \ref{sec:renorm}-\ref{sec:criterion}, results related to twisted ergodic integrals were obtained for undeformed tiling spaces $\Omega_x$ given a tower structure defined by $x\in \Sigma_N$ through the random substitution tiling construction from \S \ref{sec:RandSub}. It would be sensible now to revisit some of those results, especially Propositions \ref{prop:TracesBound}, \ref{prop:twistedBound} and \ref{prop:Veech}, in light of the results on deformations from the previous section.
\subsection{Propositions \ref{prop:TracesBound} and \ref{prop:twistedBound} revisited}
Let $S_1,\dots, S_N$ be a collection of uniformly expanding compatible substitution rules on $M$ prototiles, $x\in\Sigma_N$, $\Omega_x$ be a tiling space (assuming $\mathcal{B}_x$ is minimal), and $[\Rd]\in\mathcal{M}_x$ with good representative $\Rd$. The combinatorics of the tilings in $\Omega_x$ and those in $\Omega_x^{\Rd}$ are the same: for any level-$k$ supertile $\mathcal{P}$ for a tiling in $\Omega_x$ there is a corresponding deformed level-$k$ supertile $\mathcal{P}_{\Rd}$ for tilings in $\Omega_x^{\Rd}$ which is tiled by the same number of tiles which tile $\mathcal{P}$. In short, the combinatorics of $\Omega_x$ and $\Omega_x^{\Rd}$ are the same; what has changed is the geometry of the return vectors. That is, if $\Lambda^{(0)}_x$ is the set of return vectors for tilings in $\Omega_x$ (see \S \ref{subsec:RV}) and $\Lambda^{(0)}_{x,\Rd}$ is the set of return vectors for tilings in $\Omega_x^{\Rd}$, then $\Lambda^{(0)}_{x}\neq \Lambda^{(0)}_{x,\Rd}$\footnote{The reason $\Lambda^{(0)}_{x}\neq \Lambda^{(0)}_{x,\Rd}$ goes as follows: deformations of tilings are done by deforming a collection of edges on the tiling. Since a basis of $\Lambda^{(0)}_{x}$ is given by vectors which are a finite sum of edges along said tiling, then there is an element of this basis such that if the edges are then deformed, then their sum will be a different vector, and this vector will be a basis element of $\Lambda^{(0)}_{x,\Rd}$.}, and so $\Gamma_x \neq \Gamma_x^{\Rd} := \mathbb{Z}[\Lambda^{(0)}_{x,\Rd}]$ and $\Gamma_x^{\Rd}$ is still finitely generated by finite local complexity. 

Recall the family of elements $\{ a^k_{x,i,\lambda}\}$ in (\ref{eqn:diagonal}) of the LF algebra $LF(\mathcal{B}^+_x)$, where $k\geq 0$, $1\leq i \leq M$ and $\lambda\in\mathbb{R}^d$. As seen in the previous section, a deformation of $\Omega_x$ by a good representative $\Rd$ of $[\Rd]\in\mathcal{M}_x$ sends the control points $\mathcal{C}_\mathcal{T}$ of a tiling $\mathcal{T}\in\Omega_x$ to control points $H_{\Rd}(\mathcal{C}_\mathcal{T})$ of the tiling $\mathcal{T}_{\Rd}\in\Omega_x^{\Rd}$. As such, since $\{ a^k_{x,i,\lambda}\}$ depended on control points, this gives a family of elements $\{ a^{k,\Rd}_{x,i,\lambda}\}$ for the deformed tiling space. Since Proposition \ref{prop:TracesBound} concerns bounding traces applied to the elements $\{ a^k_{x,i,\lambda}\}$, the deformed version of this will consider bounding traces applied to the elements $\{ a^{k,\Rd}_{x,i,\lambda}\}$. Note that the traces are independent of deformations since what is affected by deformations are the elements $a^{k,\Rd}_{x,i,\lambda}$, which are all elements of the same LF algebra $LF(\mathcal{B}^+_x)$ and the traces are unaffected by the deformations since they are linear functionals on this algebra.

Recall that $\tau\in \mathcal{W}^{x_k}_{i,j}\subset \Lambda^{(0)}_{\sigma^{-k}(x)}$ if and only if $\theta_{x_{k-1}}^{-1}\cdots \theta_{x_{1}}^{-1}\tau \in \mathcal{U}_{i,j,k}^x\subset \Lambda^{(k)}$. That is, $\tau$ describes how tiles are placed in level-$1$ supertiles of $\Omega_{\sigma^{-k}(x)}$ if and only if $\theta_{x_{k-1}}^{-1}\cdots \theta_{x_{1}}^{-1} \tau$ describes how level-$(k-1)$ supertiles lie in level-$k$ supertiles of $\Omega_x$. In the proof of Proposition \ref{prop:TracesBound}, the sets $\mathcal{U}_{i,j,k}^x$ were used in products $\langle\lambda, u \rangle$ for $u\in \mathcal{U}_{i,j,k}^x$ and $\lambda\in\mathbb{R}^d$. By (\ref{eqn:vectorCocycle2}), if $u = \theta_{(-k)_{\sigma^{-k}(x)}}^{-1}\tau $, in the undeformed space, then in the deformed space the relevant product is
\begin{equation}
  \label{eqn:DeformedProduct}
  \left\langle   \lambda, V_x^{\Rd} G^{(n)}_{\sigma^{-k}(x)} \alpha_{\sigma^{-k}(x)}(\tau) \right\rangle,
  \end{equation}
 where $\alpha_{\sigma^{-k}(x)}(\tau)$ is the address of $\tau$. With all of this in mind, the analogue of Proposition \ref{prop:TracesBound} in the deformed setting is the following proposition; the proof follows the proof of Proposition \ref{prop:TracesBound} line by line except that one needs to use the deformed products (\ref{eqn:DeformedProduct}) in the relevant places, e.g. (\ref{eqn:traceExp2}) - (\ref{eqn:traceExp3}).
\begin{prop*}
  Let $S_1,\dots, S_N$ be a family of uniformly expanding compatible substitution rules on the prototiles $t_1,\dots, t_M$ and a positively simple word $w = w^-w^+$. Suppose  $x\in\Sigma_N$ has infinitely many occurences of the word $w$, both in the future and the past. Denote by $k_n\rightarrow \infty$ the maximal sequence of return times of $x$ to $C([w^-.w^+])$ under $\sigma^{-1}$. Pick $[\Rd]\in\mathcal{M}_x$ with good representative $\Rd$. Then there exists a $\ell\in\{1,\dots, M\}$, a finite set of addresses $\bar{\Lambda}_x\subset   \mathbb{Z}^{r_x}$, $c>0$, and $k_x>0$ such that for any $m\in\mathbb{N}$
  \begin{equation}
    \label{eqn:twistBnd1}
    \begin{split}
      &\left|\mathfrak{t}^x_{m,j}(a^{m,\Rd}_{x,i,\lambda})\right|\\
      &\hspace{.4in}\leq \left\{
      \begin{array}{ll}
        \displaystyle\left\| \Theta_x^{(m)}\right\|\prod_{n<N}\left( 1 - c\max_{\tau\in \bar\Lambda_x} \| \left\langle\tau, \bar{G}^{(k_n+k_x)}V_x^{\Rd *}\lambda\right\rangle \|^2_{\mathbb{R}/\mathbb{Z}}\right) & \mbox{ if }\,k_{N}\leq m < k_N+k_x  \\
        \displaystyle\left\| \Theta_x^{(m)}\right\|\prod_{n \leq N}\left( 1 - c\max_{\tau\in \bar\Lambda_x} \|\left\langle\tau, \bar{G}^{(k_n+k_x)}V_x^{\Rd *}\lambda\right\rangle \|^2_{\mathbb{R}/\mathbb{Z}}\right) & \mbox{ if }\, k_N+k_x\leq m \leq k_{N+1} 
      \end{array} \right.
    \end{split}
  \end{equation}
  for any $i,j$.
\end{prop*}
Given this, the following proposition follows in the same way that Proposition \ref{prop:twistedBound} follows from Proposition \ref{prop:TracesBound}.
\begin{proposition}
\label{prop:twistedBound2}
  Let $S_1,\dots, S_N$ be a collection of uniformly expanding compatible substitution rules on $M$ prototiles. Let $\mu$ be a positively simple, $\sigma$-invariant ergodic probability measure for which the trace cocycle is integrable. For every positively simple word $w = w^-w^+$ such that $\mu(C([w^-.w^+]))>0$ there exists a finite set of addresses $\Lambda_w\subset \mathbb{Z}^{r_x}$ and $c_w>0$ such that for any $\epsilon\in(0,1)$, for $\mu$-almost every $x\in\Sigma_N$ there is a subsequence $k_n\rightarrow \infty$, $z>1$ and $\bar{K}_{\epsilon,x}>0$ such that for any $[\Rd]\in\mathcal{M}_x$ with good representative $\Rd$, any level-$y$ bounded TLC function $f:\Omega_x^{\Rd}\rightarrow \mathbb{R}$, $\mathcal{T}\in\Omega_x^{\Rd}$ and $\lambda\in\mathbb{R}^d$
  $$\left|\mathcal{S}^\mathcal{T}_R(f,\lambda)\right|\leq C_{x,f,\epsilon} R^{d+\epsilon} \prod_{m_0<n<N_\epsilon(R)} \left( 1-c_w \max_{\tau\in \Lambda_w}\left\| \left\langle\tau, \bar{G}^{(k_n+|w^+|)}V_x^{\Rd *}\lambda\right\rangle\right\|^2_{\mathbb{R}/\mathbb{Z}}  \right)+ \mathcal{O}(R^{d-1})$$
 for all $R>C_{\epsilon, x}z^{y}$, and $N_\epsilon(R)$ is the largest return of $x$ to $C([w^-.w^+])$ which is less than $\frac{d\log R - \bar{K}_{\epsilon,x}}{\eta_1+\epsilon}$ and $m_0$ is the smallest integer which satisfies $k_{m_0}\geq y$.
\end{proposition}
Now Lemma \ref{lem:Postal} can be used as in the undeformed version, that is, the deformed analogue of (\ref{eqn:oftheForm}) is
\begin{equation}
  \label{eqn:oftheForm2}
  \begin{split}
    &\prod_{y'<n<N_\varepsilon(R)} \left( 1-c_w \max_{\tau\in \Lambda_w}\left\| \left\langle\tau, \bar{G}^{(k_n+|w^+|)}V_x^{\Rd *}\lambda\right\rangle\right\|^2_{\mathbb{R}/\mathbb{Z}}  \right) \\
    &\hspace{2.1in}\leq \prod_{y'<n<N_\varepsilon(R)} \left( 1-c_w' \left\| \bar{G}^{(k_n+|w^+|)}_x V_{x}^{\Rd *}  \lambda    \right\|^2_{\mathbb{R}^{r_x}/\mathbb{Z}^{r_x}}  \right)
  \end{split}
\end{equation}
provided that $\Lambda_w$ generates $\mathbb{Z}^{r_x}$.
\subsection{The criterion for deformed spaces}
For any positively simple postal word $w = w^-w^+$ with $\mu(C([w]))>0$, any $\rho,\delta\in(0,1/2)$, $[\Rd]\in\mathcal{M}_x$ and $\lambda\in\mathbb{R}^d$, define
\begin{equation}
  \label{eqn:density}
  \mathcal{D}_N^x(w,\rho,[\Rd],\lambda):=  \frac{\displaystyle\sum_{j=1}^{N}\mathbbm{1}_{C([w^-.w^+])}\left(\sigma^{-j}(x)\right)\mathbbm{1}_{B_\rho\left(H^1(\Omega_{\sigma^{-j}(x)};\mathbb{Z})\right)}\left(\Phi_x^{(j)*}[\Rd] (\lambda)  \right)}{\displaystyle\sum_{j=1}^{N}\mathbbm{1}_{C([w^-.w^+])}(\sigma^{-j}(x))}.
\end{equation}
The quantity $\mathcal{D}_N^x(w,\rho,[\Rd],\lambda)$ measures the density of the iterates of $\sigma^{-j}(x)$ up to $N$ which satisfy two simultaneous properties:
\begin{enumerate}
\item that the iterate on the base dynamics lands in the special compact set $C([w^-.w^+])$ for the positively simple word $w = w^-w^+$; and
\item the the fiber dynamics $\Phi^{(n)*}_x[\Rd](\lambda)$ land at most $\rho$ from a point in the lattice $H^1(\Omega_{\sigma^{-j}(x)};\mathbb{Z})$.
\end{enumerate}
Since $x$ is taken to be $\mu$-generic and $\mu(C(w^-.w^+))>0$, the iterates $\sigma^{-j}(x)$ will visit the special compact set with asymptotic frequency $\mu(C(w^-.w^+))$. Thus the density quantity $\mathcal{D}_N^x(w,\rho,[\Rd],\lambda)$ is really focusing on (ii) above, the behavior of the fiber dynamics. Furthermore, if $k_n\rightarrow \infty$ is a maximal sequence of return times to $C(w^-.w^+)$ for $\sigma^{-1}$, then it follows from the definition that
\begin{equation}
  \label{eqn:Accdensity}
  \mathcal{D}^x_{k_n}(w,\rho,[\Rd],\lambda) = \frac{1}{n}\sum_{j=1}^{n} \mathbbm{1}_{B_\rho\left(H^1(\Omega_{\sigma^{-k_j}(x)};\mathbb{Z})\right)}\left(\Phi_x^{(k_j)*}[\Rd] (\lambda)  \right)
\end{equation}
and $\mathcal{D}^x_{N}$ is constant for $N\in \{k_n,\dots, k_{n+1}-1\}$. Thus if $\mathcal{D}_{k_n}(w,\rho,[\Rd],\lambda)<1-\delta$ for all $n$ large enough, then $\mathcal{D}_{N}(w,\rho,[\Rd],\lambda)<1-\delta$ for all $N$ large enough.

These observations will be particularly relevant in the following section on the Erd\H{o}s-Kahane method to exclude bad deformation parameters. That is, instead of considering the cocycle over $\sigma^{-1}$, it will be convenient to study the induced cocycle over the induced system $\sigma^{-r_{C_w}}:C(w^-.w^+)\rightarrow C(w^-.w^+)$ of returns to $C(w^-.w^+)$.
\begin{lemma}
\label{lemma:OFK}
  Let $\mu$ be a $\sigma$-invariant ergodic probability measure such that $\log \|\Phi_x^*\|\in L^1$, where $\Phi_x^*$ is the induced action on the first cohomology. Suppose $x$ is Oseledets regular for the cohomology cocycle. Then it is Furstenberg-Kesten regular for the trace cocycle.
\end{lemma}
\begin{proof}
  The largest exponents for the cocycle on the first cohomology space correspond to the classes represented by the $d$ constant, pattern equivariant 1-forms $dx_1,\dots, dx_d$. The Lyapunov exponents are
  $$\lambda_{top} = \lim_{n\rightarrow\infty} \frac{\log |\theta_{x_{1}}^{-1}\cdots \theta_{x_{n}}^{-1}|}{n}$$
  which exists by assumption. The trace cocycle by definition tracks the number of prototiles in level-$n$ supertiles as $n$ grows, which is proportional to the volume of level-$n$ supertiles. Thus the top exponent of the trace cocycle is given by
  $$\lim_{n\rightarrow\infty} \frac{\log \left\|\Theta_x^{(n)}\right\|}{n} = \lim_{n\rightarrow \infty}\frac{\log |\theta_{x_{1}}^{-d}\cdots \theta_{x_{n}}^{-d}|}{n} = d \lambda_{top},$$
  so $x$ is Furstenberg-Kesten regular, and $\eta_1 = d\lambda_{top}$.
\end{proof}

The following is a quantitative, non-stationary version of \cite[Theorem 4.1]{ClarkSadun:shape}.
\begin{proposition}[Cohomological quantitative Veech criterion]
  \label{prop:CQVC}
  Let $S_1,\dots, S_N$ be a collection of compatible substitution rules on tiles $t_1,\dots, t_M$ and $\mu$ a positively simple, postal, $\sigma$-invariant ergodic probability measure for which the cocycle on $H^1$ is integrable. For any positively simple postal word $w = w^-w^+$ with $\mu(C([w]))>0$, for any Oseledets-regular $x\in\Sigma_N$, if for $[\Rd]\in H^1(\Omega_x;\mathbb{R}^d)$ there exists $\rho,\delta\in (0,\frac{1}{2})$ and $\vartheta, N_0(x)>1$ such that any $B\geq 2$ and $\lambda \in \mathcal{A}_B$,  $\mathcal{D}_N^x(w,\rho,[\Rd],\lambda) <1-\delta$ 
  for all $N>N_0(x)+\vartheta\log B$, then there is a $\gamma'>0$ such that for any good representative $\Rd$ of $[\Rd]$ and any Lipschitz function $f:\Omega_x^{\Rd}\rightarrow \mathbb{R}$ and $\mathcal{T}\in \Omega_x^{\Rd}$,
  $$\left|\int_{C_R(0)}e^{-2\pi \imath \langle \lambda, t\rangle} f\circ \varphi_t(\mathcal{T})\, dt\right|\leq C_{x,f,\varepsilon} R^{d-\gamma'}+\mathcal{O}(R^{d-1})$$
  for all $R>N'(x)\cdot B^\theta$, with the corresponding lower bound on lower local dimension for $d^-_f(\lambda)$ for all $\lambda\in\mathcal{A}_B$.
\end{proposition}
\begin{proof}
If $\| \Phi^{(j)*}_x[\Rd](\lambda)\|\leq \rho$ then by (\ref{eqn:betaEquiv}),
\begin{equation*}
  \begin{split}
    \|\beta_{\sigma^{-j}(x)}\|^{-1}\|\beta_{\sigma^{-j}(x)}\left( \Phi^{(j)*}_x[\Rd](\lambda)\right)\|&=\|\beta_{\sigma^{-j}(x)}\|^{-1}\|\bar G_x  \beta_x([\Rd](\lambda))\| \\
    &=\|\beta_{\sigma^{-j}(x)}\|^{-1}\|\bar G_x  V_x^{\Rd*}\lambda\| \leq \rho,
  \end{split}
\end{equation*}
assuming $\sigma^{-j}(x)\in C([w^-.w^+])$. Thus $\Phi^{(j)*}_x[\Rd](\lambda)\in B_\rho\left( H^1(\Omega_{\sigma^{-j}(x)};\mathbb{Z})\right)$ implies that $G_x  V_x^{\Rd*}\lambda \in B_{\rho'}(\Gamma^*_\star)$ for $\rho' = \|\beta_{\sigma^{-j}(x)}\|\rho.$

It folows that, by (\ref{eqn:density}), if $\mathcal{D}_N^x(w,\rho,[\Rd],\lambda)<1-\delta$ for all $N$ large enough, then the same holds for the analogous quantity defined in the deformed case at the level of return vectors (which is analogous to (\ref{eqn:densityVeech}))
  \begin{equation}
    \label{eqn:densityVeech2}
    \frac{\displaystyle\sum_{j=1}^{N}\mathbbm{1}_{C([w^-.w^+])}\left(\sigma^{-j}(x)\right)\mathbbm{1}_{B_{\rho'}(\Gamma^*_\star)}\left( \bar{G}^{(j)}_x V_{x}^{\Rd *}  \lambda\right)}{\displaystyle\sum_{j=1}^{N}\mathbbm{1}_{C([w^-.w^+])}(\sigma^{-j}(x))}<1-\delta'
  \end{equation}
using $\rho' = \|\beta_{\sigma^{-j}(x)}\|\rho$ which, by compactness is bounded. Given Proposition \ref{prop:twistedBound2} and Lemma \ref{lemma:OFK}, the same calculation as in (\ref{eqn:VeechBnd}) implies that if the density condition (\ref{eqn:densityVeech2}) is satisfied for all $N$ large enough, then the quantitative weak mixing conclusion of Proposition \ref{prop:Veech} holds for the deformed tiling space $\Omega_x^{\Rd}$.
\end{proof}
\section{The Erd\H{o}s-Kahane method \`a la Bufetov-Solomyak}
\label{sec:dimension}
In this section it is shown that the typical class $[\Rd]\in\mathcal{M}_x$ admits a representative which satisfies the conditions in the cohomological quantitative Veech criterion (Proposition \ref{prop:CQVC}). The argument very closely follows that of Bufetov-Solomyak in \cite[\S 6]{BS:translation}, which they call the \emph{vector form of the Erd\H{o}s-Kahane method}, and it is a tool to bound the Hausdorff dimension of a set of bad values of $\Rd\in\mathcal{M}_x$. It should be noted that the case when $d=1$ is already done in \cite[\S 6]{BS:translation}, so what is relevant here is the case $d>1$. I am grateful to Boris Solomyak who showed me how to upgrade the method in \cite[\S 6]{BS:translation} to an effective version, which is crucially needed for the uniform bounds in Theorem \ref{thm:main2}.

In this section it is assumed that $\mu$ is a positively simple, postal, $\sigma$-invariant ergodic probability measure for which the return vector cocycle and the trace cocycle are integrable. Throughout $x$ will denote a $\mu$-generic and Oseledets-regular point (for both cocycles) and the assumption on $d\mu(x|x^-)$ on the main theorem will come up, so it is also assumed in this section. Let $r_\mu$ be the rank of $\Gamma_x$ for $\mu$-almost every $x$. Recall that by \S \ref{subsec:RV} there is a collection of subsets $\{U_i^{r_\mu}\}$ of $\Sigma_N$ with the property that for each $i$ there is a collection of vectors $v_1,\dots, v_{r_\mu}$ such that $\Gamma_x$ is generated by those vectors for any $x\in U_i^{r_\mu}$. By the hypotheses on $\mu$, without loss of generality it can be assumed that there is an $i$ such that $U_i^{r_\mu} = C([w^-.w^+])$, where $w = w^-w^+$ is a positively simple word and postal. As such, $\mu(U_i^{r_\mu})>0$, and set $U_\star = U_i^{r_\mu}$. In this section and the next, classes will be denoted by $\Rd$ and not $[\Rd]$ as in previous sections.

Let $k_n\rightarrow \infty$, $n\in\mathbb{N}$, be a maximal sequence of return times to $U_\star$ for $x$ under $\sigma^{-1}$. Since there is a local choice of generators $v_1,\dots, v_{r_\mu}$ for all $\Gamma_x$ with $x\in U_\star$, there is a natural identification of all fibers $\mathfrak{R}_x$ and their lattices $\alpha(\Gamma_x)\subset \mathfrak{R}_x$ for all $x\in U_\star$, and in particular for $\sigma^{k_n}(x)$. Under this identification, the vector space is denoted by $\mathfrak{R}_\star$ and the corresponding lattice is $\Gamma_\star\subset \mathfrak{R}_\star$. This also holds for the dual spaces and lattices $\Gamma_\star^*\subset \mathfrak{R}_\star^*$. Following the discussion around (\ref{eqn:Accdensity}), the main player in this section is not the cocycle over $\sigma^{-1}$ but the induced cocycle of the return map $\sigma^{-r_{U_\star}}:U_\star\rightarrow U_\star$, which is why a maximal sequence of return times to $U_\star$ has been identified.

Let $V_\star$ be the $ d\times r_\mu$ matrix associated with $U_\star$ as in (\ref{eqn:vectorCocycle}). That is, for any $x\in U_\star$ and $\tau\in \Lambda_x$ there is a unique $\alpha_\star(\tau)\in\Gamma_\star$ such that $\tau = V_\star \alpha_\star(\tau)$. For $\lambda\in\mathbb{R}^d$, let $K_x^n(\lambda)\in H^1(\Omega_x;\mathbb{Z})$ be the closest lattice point to $\Phi_x^{(k_n)*}\Rd(\lambda)$ and set
\begin{equation}
  \label{eqn:Kepsilon}
\Phi_x^{(k_n)*} \Rd( \lambda) = K_x^n(\lambda) + \varepsilon_x^n(\lambda),\;\;\;\; \|\varepsilon_x^n(\lambda)\| := \left\|\Phi_x^{(k_n)*}\Rd( \lambda) \right\|_{H^1(\Omega_x;\mathbb{R})/H^1(\Omega_x;\mathbb{Z})},
\end{equation}
where $\|\cdot\|_{V/G}$ is the shortest distance of an element $v$ in a vector space $V$ to a point in a lattice $G\subset V$. For any $B>1$, let
$$\mathcal{M}_x^B:= \{\Rd  \in\mathcal{M}_x:  B^{-1}\leq \|\Rd\| \leq B\}\;\;\;\;\mbox{ and }\;\;\;\; \mathcal{A}_B:= \{\lambda  \in\mathbb{R}^d:  B^{-1}\leq \|\lambda\| \leq B\}$$
and furthermore, given $\rho,\delta>0, B>1$ and $x\in\Sigma_N$, let
\begin{equation}
  \begin{split}
    E_N^x(\rho,\delta, B) &:= \left\{(\Rd,\lambda) \in \mathcal{M}_x^B\times \mathcal{A}_B: \;\mathcal{D}_{k_N}^x(w,\rho,\delta,\Rd,\lambda) \geq 1-\delta\right\} \\
    \mathcal{E}_N(\rho,\delta, B) &:= \left\{\Rd\in\mathcal{M}_x^B:\, \mbox{ there is }\lambda\in\mathcal{A}_B\mbox{ such that } (\Rd,\lambda) \in E^x_N(\rho,\delta, B)\right\} \\
    &= \pi_x(E_N^x(\rho,\delta, B)),
  \end{split}
\end{equation}
where $\pi_x$ is the projection onto the $\mathcal{M}_x$-coordinate, and set for $\vartheta>0$
  $$\mathfrak{B} = \mathfrak{B}(\rho,\delta,\vartheta):=  \bigcap_{N_0>0}\bigcup_{B >1}\bigcup_{N_0\geq +\vartheta \log B}\mathcal{E}_N(\rho,\delta, B).$$
  
  Denote by $K_n = K_x^n(\lambda)$ and $\varepsilon_n = \varepsilon_x^n(\lambda)$. Let $\Pi_x^+:H^1(\Omega_x;\mathbb{R})\rightarrow Z_x^+$ be the projection onto the unstable space of the fiber over $x$ (see sentence after (\ref{eqn:betaEquiv})). Then
  $$\Phi_x^{(k_n)*}\Pi_{x}^+\Rd(\lambda) = \Pi_{\sigma^{-k_n}(x)}^+\Phi_x^{(k_n)*}\Rd(\lambda)  = \Pi^+_{\sigma^{-k_n}(x)}K_n + \Pi^+_{\sigma^{-k_n}(x)}\varepsilon_n $$
  and thus
  \begin{equation}
    \label{eqn:EK1}
    \Pi_{x}^+\Rd(\lambda) = \Phi_x^{(k_n)*-1}\left(\Pi^+_{\sigma^{-k_n}(x)}K_n\right) + \Phi_x^{(k_n)*-1}\left( \Pi^+_{\sigma^{-k_n}(x)}\varepsilon_n \right)
  \end{equation}
 Since $x$ is Oseledets regular, for any $\epsilon\in(0,\eta_*)$ and all $n$ sufficiently large,
  $$\left\| \Phi_x^{(k_n)*-1} \cdot \Pi^+_{\sigma^{-k_n}(x)}\right\|\leq e^{-(\eta_*-\epsilon)n},$$
  where $\eta_*$ is the smallest positive Lyapunov exponent for the induced cocycle on $H^1(\Omega_x;\mathbb{R})$ by the induced return system to $U_\star$. Therefore, from (\ref{eqn:EK1}), since $\varepsilon_n<1$,
  \begin{equation}
    \label{eqn:EK2}
    \left\| \Pi_{x}^+\Rd(\lambda) - \Phi_x^{(k_n)*-1}\left(\Pi^+_{\sigma^{-k_n}(x)}K_n\right) \right\|< e^{-(\eta_*-\epsilon)n}
  \end{equation}
for all $n$ large enough.
  
Set
\begin{equation}
  \label{eqn:EK3}
  \begin{split}
    &W_n = W_n(x):= \log \left\|  \Phi_{\sigma^{-k_{n+1}+1}(x)}^{*} \cdots \Phi_{\sigma^{-k_n}(x)}^{*} \right\|,\;\;\;\;M_n:= (2+\exp(W_{n+1}))^{\beta_\mu} \\ &\hspace{2in}\mbox{ and }\;\; \rho_n:= \frac{1}{2(1+\exp(W_{n+1}))},
  \end{split}
\end{equation}
where $\beta_\mu = \mathrm{dim}\,H^1(\Omega_x;\mathbb{R})$, which is constant for $\mu$-almost every $x$.
Define
$$ \bar{E}_N^x(\delta,B) := \left\{ (\Rd,\lambda) \in \mathcal{M}_x^B\times \mathcal{A}_B : \mathrm{card}\{n\leq N :\mbox{max}\{\|\varepsilon_n\|,\|\varepsilon_{n+1}\|\}\geq \rho_n\}<\delta N \right\},$$
\begin{lemma}
  \label{lem:badSets}
  For $\mu$-almost every $x$, there is an $N_*(x)$ such that given $\delta>0$ small enough:
  $$E_N^x(\varrho,\delta/4,B)\subset \bar{E}_N^x(\delta,B)$$
  for all $N>N_*(x)$, where $\varrho$ depends on $\delta$.
\end{lemma}
\begin{proof}
  The following result, which combines Proposition 6.1 in \cite{BS:genus2} and Corollary 4.3 in \cite{BS:translation}, will be used. It is here that the hypothesis involving $\mu(x|x^-)$ of the main theorem is needed.
  \begin{proposition}
    \label{prop:BSlargeDev}
  There exists a positive constant $L$ such that for $\mu$-almost every $x$, there is an $N_*$ (depending only on $x$) such that for any $\delta>0$ small enough:
  $$\max\left\{ \sum_{n\in \Psi}W_{n+1}: \Psi\subset \{1,\dots, N\},\, |\Psi|\leq \delta N \right\}\leq L \log(1/\delta)\cdot \delta N.$$
for all $N>N_*$. Moreover, for any $C>0$
  $$\mbox{card}\{n\leq N:W_{n+1}>CL\cdot\log(1/\delta)\}\leq \frac{\delta N}{C}.$$
  \end{proposition}

  Let $K = 2L\log (1/ \delta)$, where $L$ is obtained from Proposition \ref{prop:BSlargeDev}, define $\varrho = (2+2e^K)^{-1}$, and suppose $(\Rd,\lambda)\not\in  \bar{E}_N^x(\delta, B)$. Then there is a subset $\Psi$ of $\{1,\dots, N\}$ of cardinality greater than or equal to $\delta N$ such that $\max\{\|\varepsilon_n\|,\|\varepsilon_{n+1}\|\}\geq \rho_n$
  for all $n\in\Psi$. Since $\rho_n<\varrho$ is equivalent to $W_{n+1}>K$, by the second bound in Proposition \ref{prop:BSlargeDev}, there are no more than $\delta N /2$ integers $n\leq N$ for which $W_{n+1}>K$ for all $N$ large enough. Thus
  $$\mathrm{card}\{n\leq N: \max\{   \|\varepsilon_n\|,\|\varepsilon_{n+1}\|\}\geq \varrho \}\geq \frac{\delta N}{2},$$
  so there are no less than $\delta N /4$ integers $n\leq N$ with $\|\varepsilon_n\|>\varrho$, and so it follows that $(\Rd,\lambda) \not\in E_N^x(\varrho,\delta/4,B)$.
\end{proof}
\begin{proposition}
  \label{prop:dimension}
  For any $\epsilon_*>0$ there exists a $\delta_0>0$ such that for all $\delta\in (0,\delta_0)$, there is a $\varrho>0$ such that for $\mu$-almost every $x$ and all $\vartheta$ large enough:
  $$\mathrm{dim}(\mathfrak{B}_x(\varrho, \delta, \vartheta)) \leq \mathrm{dim}\,\mathcal{M}_x - \mathrm{dim}\, Z^+_x+d+\epsilon_*.$$
\end{proposition}
\begin{proof}
  In light of Lemma \ref{lem:badSets} it suffices to show for all $\delta>0$ small enough, $\vartheta>0$ large enough, setting $\beta_{\epsilon_*} := \mathrm{dim}\,\mathcal{M}_x-\mathrm{dim}\,Z^+_x+d+\epsilon_* = d\cdot\mathrm{dim}\,Z^-_x + (d-1)\mathrm{dim}\,Z^+_x +d+\epsilon_*$, that
  $$\mathrm{dim}(\mathfrak{B}'_x(\delta, \vartheta)) \leq \beta_{\epsilon_*},\,\,\;\;\;\;\; \mbox{ where }\,\,\;\;\;\;\; \mathfrak{B}'_x(\delta,\vartheta):= \bigcap_{N_0>0}\bigcup_{B>1}\bigcup_{N \geq N_0+\vartheta\log B}\bar{\mathcal{E}}_N^x(\delta,B)$$
  and
  \begin{equation}
    \label{eqn:projN}
    \begin{split}
      \bar{\mathcal{E}}_N^x(\delta,B):&= \left\{\Rd\in\mathcal{M}_x^B: \mbox{ there is a $\lambda\in\mathcal{A}_B$ such that $(\Rd,\lambda)\in\bar{E}_N^x(\delta,B)$} \right\}\\
      &=\pi_x(\bar{E}_N^x(\delta,B)).
    \end{split}
  \end{equation}
  In order to estimate the dimension of $\mathfrak{B}'_x$, it suffices to construct a cover of $\bar{E}^x_N(\delta,B)$ since the projection map is Lipschitz with Lipschitz constant proportional to $B$ \cite[\S 2.2]{falconer:book}. It is important to remark here that in order to be able to bound dimension using covering arguments, it is necessary for the $\lambda$ coordinate in $\bar{E}_N^x(\delta,B)$ be bounded below in norm (in this case by $B^{-1}$); otherwise, we would get that $\bar{\mathcal{E}}_N^x(\delta,B) = \mathcal{M}_x^B$ which would not be very useful.
    By definition,
\begin{equation}
  \label{eqn:localProduct}
  \bar{E}_N^x(\delta,B)\subset  \Pi^-_x(\mathcal{M}_x^B) \times \Pi^+_x(\pi_x(\bar{E}_N^x(\delta,B)))\times \mathcal{A}_B.
\end{equation}
In slight abuse of notation, the set $\Pi^\pm_x(\bar{E}_N^x(\delta,B))$ will be short for $\Pi^\pm_x(\pi_x(E_N^x(\delta,B)))\times \mathcal{A}_B$. To produce the desired dimension estimates, one needs to find a cover of the set $\Pi^+_x(\bar{E}_N^x(\delta,B))$.

To this end, for any $0\neq c\in Z^+_x \subset H^1(\Omega_x;\mathbb{R})$, $B>1, \epsilon>0$ and $A\subset \Pi^+_x \mathcal{M}_x\times\mathcal{A}_B$, define
\begin{equation}
  \label{eqn:locus}
  \begin{split}
    \mathbb{E}^{-1}_B(c) &:= \left\{ (\Rd,\lambda)\in \Pi^+_x\mathcal{M}_x^B\times \mathcal{A}_B: \Rd(\lambda) = c\right\} \hspace{.2in}\mbox{ and } \\ \partial_\epsilon^B A &:= \left\{(\Rd,\lambda)\in \Pi^+_x\mathcal{M}_x^B\times\mathcal{A}_B :\mathrm{dist}\, ((\Rd,\lambda), A)\leq \epsilon \right\}
  \end{split}
\end{equation}
and note that $\mathrm{dim}\, \mathbb{E}^{-1}_B(c)= (d-1)d_\mu^+ +d $, where $d^+_\mu = \mathrm{dim}\, Z^+_x$. Moreover, by compactness, there exists a $F^B_x>0$ such that for any $\eta>0$ 
$$\Pi^+_xB_\eta(c)\subset \mathbb{E}(\partial_{F_x^B\eta}^B(\mathbb{E}^{-1}_B(c))),$$
where $B_\eta(c)\subset H^1(\Omega_x;\mathbb{R})$ is the ball of radius $\eta$ around $c$ and $\mathbb{E}$ is the evaluation map $(\Rd,\lambda)\mapsto \Rd(\lambda)\in H^1(\Omega_x;\mathbb{R})$.

Suppose $(\Rd,\lambda) \in \bar{E}_N^x(\delta,B)$, and let $K_n$ and $\varepsilon_n$ be the corresponding sequences as defined in (\ref{eqn:Kepsilon}). From (\ref{eqn:EK2}) it follows that $\Pi_x^+ \Rd(\lambda) $ is in the ball of radius $e^{-(\eta_*-\epsilon)N}$ centered at $P_N(\Rd,\lambda):= \Phi_x^{(k_N)*-1}\left(\Pi^+_{\sigma^{-k_N}(x)}K_N\right)$, and so
\begin{equation}
  \label{eqn:coverCovers}
  \mathrm{dist}\left(\left(\Pi^+_x\Rd,\lambda\right),\mathbb{E}^{-1}_B\left(\{P_N(\Rd,\lambda)\}\right)\right)\leq F_x^B e^{-(\eta_*-\epsilon)N}.
\end{equation}
Therefore $(\Rd,\lambda)$ is in one of at most $K_{B,x} \left(F_x^B e^{-(\eta_*-\epsilon)N}\right)^{-[(d-1)d_\mu^++d]}$ balls of radius $F_x^B e^{-(\eta_*-\epsilon)N}$ which cover $\mathbb{E}^{-1}_B(\{P_N(\Rd,\lambda)\})\subset \Pi^+_x\mathcal{M}^B_x\times\mathcal{A}_B$. The goal is now to consider all possible sequences of $\{K_n\}$ that can appear for any such $(\Rd,\lambda)\in \bar{E}_N^x(\delta,B)$, which will give all the possible sets of the form $\partial_{\epsilon_N}\mathbb{E}_B^{-1}(\{P_N(\Rd,\lambda)\})$ which will cover $\Pi^+_x(\bar{E}_N^x(\delta,B))$, for $\epsilon_N = e^{-(\eta_*-\epsilon)N}$.

Let $\Psi_N$ be the set of subsequences of $n\in \{1,\dots, N\}$ for which $\max\{\|\varepsilon_n\|,\|\varepsilon_{n+1}\| \}\geq \rho_n$, and note that each such subsequence has at most $\delta N$ elements in it and moreover that there are at most $\displaystyle \sum_{i<\delta N}\left(\begin{array}{c} N \\ i\end{array} \right)$ such sets. By \cite[Lemma 6.1]{BS:translation}, for a fixed $\Psi_N$ the number of subsequences $\{K_n\}$ is at most $\displaystyle \mathbb{M}_N:= \prod_{n\in\Psi_N}M_n$ times the number of possible starting values for $K_0$, which is bounded independently of $N$ by compactness since $(\Rd,\lambda)\in\mathcal{M}_x^B\times\mathcal{A}_B$ (this number is proportional to $B^{d(1+ \beta_\mu)}$). By (\ref{eqn:EK3}) and Proposition \ref{prop:BSlargeDev}, since $M_n\leq 3^{\beta_\mu}e^{\beta_\mu W_{n+1}}$, for all $N>N_*(x)$,
  $$\mathbb{M}_N\leq K  3^{\beta_\mu \delta N}\exp\left( \beta_\mu \sum_{n\in\Psi_N}W_{n+1}\right)\leq K 3^{\beta_\mu\delta N}\exp(\beta_\mu L \log(1/\delta)(\delta N)).$$
  It follows from (\ref{eqn:EK2}) that the number of balls of radius $F_x^B e^{-(\eta_*-\epsilon)N}$ needed to cover $\Pi^+_x(\pi_x\bar{E}_N^x(\delta,B))$ is no more than
  \begin{equation*}
    \begin{split}
      K' B^{d(1+\beta_\mu)}  &\sum_{i<\delta N}\left(\begin{array}{c} N \\ i\end{array} \right)3^{\beta_\mu \delta N} \exp(\beta_\mu L \log(1/\delta)(\delta N)) \\
        &\hspace{.5in}\leq K' B^{d(1+\beta_\mu)}\mathrm{exp}\, \left(\left( H(\delta)\log 2 + \delta\beta_\mu\log 3 + \beta_\mu L \delta \log(1/\delta)\right)N\right),
    \end{split}
  \end{equation*}
 where the right-hand side follows from the bound of the sum of binomial coefficients up to $\delta N$ being bounded by $2^{H(\delta)N}$, where $H(\delta)$ is the binary entropy function \cite[\S 11.1]{CT:book}. Therefore, since $H(\delta)$ and $\delta\log(1/\delta)$ go to zero as $\delta\rightarrow 0^+$, there exists a $\delta_0$ such that $\delta\in(0,\delta_0)$ implies that
    \begin{equation}
      \label{eqn:EK4}
      H(\delta)\log 2 + \delta\beta_\mu\log 3 + \beta_\mu L \delta \log(1/\delta)\leq \frac{\epsilon_*}{2}(\eta_*-\epsilon).
    \end{equation}
    The set $ \Pi^-_x(\mathcal{M}_x^B)$ can be covered with roughly $\exp( (\eta_*-\epsilon)(\beta_{\epsilon_*}-d-\epsilon_* - (d-1)d^+_\mu)N)$ balls of radius $e^{-(\eta_*-\epsilon)N}$, whereas by (\ref{eqn:locus})-(\ref{eqn:EK4}) the set $\Pi^+_x(E_N^x(\delta,B))$ can be covered by roughly $ \exp( (\eta_*-\epsilon)N\beta'_{\epsilon_*})$ balls of radius $e^{-(\eta_*-\epsilon)N}$, where $\beta'_{\epsilon_*}=(d-1)d^+_\mu+d+\epsilon_*/2$. So there exists a $\Delta> d(1+\beta_\mu)$ depending only on the dimensions such that by (\ref{eqn:localProduct}), the $\beta_{\epsilon_*}$-dimensional Hausdorff measure of $\mathfrak{B}'_x$
    \begin{equation*}
      \begin{split}
        \mathcal{H}^{\beta_{\epsilon_*}}\left( \mathfrak{B}_x'(\delta,\vartheta)\right) &\leq K''\cdot \limsup_{N_0\rightarrow \infty} \sum_{B=2}^\infty B^{\Delta}\sum_{N\geq N_0+\vartheta \log B} e^{(\eta_*-\epsilon)\beta'_{\epsilon_*}N}e^{-(\eta_*-\epsilon)\beta_{\epsilon_*}N} \\
        &= K''\cdot \limsup_{N_0\rightarrow \infty} \sum_{B=2}^\infty B^{\Delta}\sum_{N\geq N_0+\vartheta \log B} e^{-(\eta_*-\epsilon)N\epsilon_*/2 } \\
        &\leq K''\cdot \limsup_{N_0\rightarrow \infty} \sum_{B=2}^\infty B^{\Delta}e^{-(\eta_*-\epsilon)(N_0+\vartheta \log B)\epsilon_*/2 } \\
        &= K''\cdot \limsup_{N_0\rightarrow \infty} \sum_{B=2}^\infty B^{\Delta-\vartheta(\eta_*-\epsilon)\epsilon_*/2}e^{-(\eta_*-\epsilon)N_0\epsilon_*/2 } = 0
      \end{split}
    \end{equation*}
    as long as $\vartheta>\frac{2\Delta}{(\eta_*-\epsilon)\epsilon_*}$, and so it follows that $\dim\left( \mathfrak{B}_x'(\delta,\vartheta)\right)\leq \beta_{\epsilon_*}$.
\end{proof}
\section{Some proofs}
\label{sec:proofs}
\begin{proof}[Proof of Theorem \ref{thm:main}]
  Suppose for a Oseledets regular $x$ the deformation parameter $\Rd\not\in \mathfrak{B}_x(\rho,\delta,\vartheta)$ for some $\delta,\rho,\vartheta$. This means that there is a $N_0$ such that for all $B\geq 2$ and $N\geq N_0+\vartheta\log B$, $\Rd\not\in \mathcal{E}_N(\rho,\delta, B)$, which in turn implies that for all $\lambda\in \mathcal{A}_B$, $(\Rd,\lambda)\not\in E_N^x(\rho,\delta,B)$, for all $N> N_0+\vartheta \log B$, i.e. $\mathcal{D}_{k_N}^x(w,\rho, \Rd, \lambda)<1-\delta$, where $k_n\rightarrow \infty$ is a maximal sequence of return times of $\sigma^{-1}$. For any $\varepsilon$ small enough there is a $N_\varepsilon(x)>0$ such that for all $n>N_\varepsilon(x)$, $(\mu(C_w)-\varepsilon)k_n<n<(\mu(C_w)+\varepsilon)k_n$ and so there is a $N'_\varepsilon(x)>0$ (depending on $N_\varepsilon$) such that for all $k>\max\{N'_\varepsilon,\frac{N_0+\vartheta\log B}{\mu(C_w)-\varepsilon}\}$, for all $B\geq 2$ and for all $\lambda\in \mathcal{A}_B$, $\mathcal{D}_{k}^x(w,\rho, \Rd, \lambda)<1-\delta$.

  If $\mathrm{dim}\, Z^+_x>d$ then the dimension of $\mathfrak{B}_x(\rho,\vartheta,B)$ is strictly less than that of $\mathcal{M}_x$, and so for almost every $\Rd\in \mathcal{M}_x$, for any $B\geq 2$, the density condition $\mathcal{D}_N^x(w,\rho, \Rd, \lambda)<1-\delta$ holds for all $\lambda\in\mathcal{A}_B$, for $N\geq N_0'+\vartheta'\log B$. The result then follows from the cohomological quantitative Veech criterion, Proposition \ref{prop:CQVC}.
\end{proof}
\begin{proof}[Proof of Corollary \ref{cor:main}]
  By considering a high power of a single substitution, it can be assumed by primitivity that each tile is subdivided into enough tiles so that there are enough translation vectors between tiles of the same type which generate the group of return vectors, and so postal words are trivial (it can be assumed the corresponding $x$ is a fixed point of $\sigma$). Thus if $x\in\Sigma_N$ is fixed for the shift map and represents this substitution, the measure $\delta_x$ is a positively simple, postal ergodic probability measure. Thus, if the unstable space of $H^1(\Omega_x;\mathbb{R})$ with respect to the induced action by the substitution has dimension greater than $d$, Theorem \ref{thm:main} applies.
  \end{proof}

\begin{proof}[Proof of Lemma \ref{lem:dimension}]
  By definition:
  \begin{equation}
    \label{eqn:L2bound}
    \begin{split}
      \left\| \mathcal{S}_R(f,\lambda) \right\|^2 & = \left\langle\int_{C_R(0)} e^{-2\pi \imath \langle \lambda,\tau \rangle}f\circ\varphi_\tau\, d\tau, \int_{C_R(0)} e^{-2\pi \imath \langle \lambda,s \rangle}f\circ\varphi_s\, ds \right\rangle \\
      &= \int_\Omega  \int_{C_R(0)}\int_{C_R(0)} e^{-2\pi \imath \langle \lambda,\tau-s \rangle}f\circ\varphi_\tau(\mathcal{T})\, d\tau  f\circ\varphi_s(\mathcal{T})\, ds\,    d\mu(\mathcal{T}) \\
      &=  \int_{C_R(0)}  \int_{C_R(0)} e^{-2\pi \imath \langle \lambda,\tau-s \rangle}\langle f\circ\varphi_{\tau-s},f\rangle\, d\tau \, ds \\
      &=  \int_{C_R(0)}  \int_{C_R(0)} e^{-2\pi \imath \langle \lambda,\tau-s \rangle}\left(\int_{\mathbb{R}^d}e^{2\pi \imath \langle\tau-s,z\rangle}\, d\mu_f(z) \right)\, d\tau\, ds \\
      &=  \int_{\mathbb{R}^d} \int_{C_R(0)}  \int_{C_R(0)} e^{2\pi \imath \langle z-\lambda,s-\tau \rangle} \, d\tau\, ds\,  d\mu_f(z)\\
      &= \int_{\mathbb{R}^d}\prod_{i=1}^d\frac{\sin^2(2\pi (z_i-\lambda_i)R)}{\pi^2(z_i-\lambda_i)^2} \,  d\mu_f(z).
    \end{split}
  \end{equation}
  There exists a $C_2$ such that for $r = \frac{1}{2R}$,
  $$C_2R^2\leq \frac{\sin^2(2\pi w R)}{\pi^2w^2}$$
  for $w\in[-r,r]$. Since $B_r(\lambda)\subset C_r(\lambda)$, by (\ref{eqn:L2bound}), the hypothesis implies that
  \begin{equation}
    \label{eqn:HofBnd}
    \begin{split}
      \mu_f(B_r(\lambda)) &\leq \mu_f(C_r(\lambda)) = \int_{C_r(\lambda)}\, d\mu_f(z) \leq\frac{1}{(C_2R)^{d}} \int_{C_r(\lambda)} \prod_{i=1}^d\frac{\sin^2(2\pi (z_i-\lambda_i)R)}{\pi^2(z_i-\lambda_i)^2} \,  d\mu_f(z)\\
      &\leq \frac{1}{(C_2R^2)^{d}} \int_{\mathbb{R}^{d}}\prod_{i=1}^d\frac{\sin^2(2\pi (z_i-\lambda_i)R)}{\pi^2(z_i-\lambda_i)^2} \, d\mu_f(z) = \frac{1}{(C_2R^2)^{d}} \|\mathcal{S}_R(f,\lambda)\|^2 \\
      & \leq \frac{C_1^2}{C_2^dR^{2d}} R^{2(d-\alpha)} \leq C r^{2\alpha}
    \end{split}
  \end{equation}
  for any $r<(2R_0)^{-1}$.
\end{proof}

\section{Uniform rates of weak mixing}
\label{sec:uniform}
The first step to prove uniform rates of weak mixing is to prove effective bounds of the correlation (\ref{eqn:avgCorr}) in Theorem \ref{thm:main2}. This will first be done before handling the uniform rates of weak mixing.

Let $\Omega$ be a tiling space corresponding to repetitive aperiodic tilings of finite local complexity with a uniquely ergodic $\mathbb{R}^d$ action. For $f\in L^2$, H\"older's inequality gives the bound
\begin{equation}
  \label{eqn:Holder}
  \left| \int_{C_R(0)}e^{2\pi \imath \langle\lambda,t\rangle}\langle f\circ \varphi_t, f\rangle \, dt \right|\leq \|f\|_{L^2}\left\|\mathcal{S}_R(f,-\lambda)\right\|_{L^2}.
\end{equation}
By the polarization identity, in order to bound correlations of the form $\langle f\circ \varphi_t,g\rangle$, it suffices to bound self-correlations of the form $\langle f\circ \varphi_t,f\rangle$. For such correllations, using that $\langle f\circ \varphi_t,f\rangle = \hat{\mu}_f(t)$ (Bochner's theorem),
\begin{equation}
  \label{eqn:correllation}
  \begin{split}
    \int_{C_R(0)}|\langle f\circ \varphi_t,f\rangle|^2\, dt &= \int_{\mathbb{R}^d}\mathbbm{1}_{C_R(0)} \langle f\circ \varphi_t,f\rangle \overline{\langle f\circ \varphi_t,f\rangle}\, dt \\
    &= \int_{\mathbb{R}^d}\int_{C_R(0)}e^{2\pi \imath \langle\lambda,t\rangle}\langle f\circ \varphi_t, f\rangle \, dt\, d\bar{\mu}_f(\lambda).
  \end{split}
\end{equation}
It follows that in order to obtain a rate of growth of the correllations (\ref{eqn:correllation}) then by (\ref{eqn:Holder}) one has to bound the twisted integrals of $f$ for all spectral parameters $\lambda\in\mathbb{R}^d$. For some $\delta\in(0,1)$ this problem will be subdivided in three parts:
\begin{enumerate}
\item for $\delta\leq \|\lambda\| \leq \delta^{-1}$, the control of the twisted integrals is obtained from the results of Theorem \ref{thm:main},
\item for $\|\lambda\| >\delta^{-1}$, the twisted integrals will be bounded using integration by parts, assuming that the functions are sufficiently smooth in the leaf direction,
\item for $\|\lambda\|\leq \delta$, the twisted integrals will be bounded using the bounds on the growth of ergodic integrals of $f$, assuming that such behavior is known.
\end{enumerate}
To this end, let
\begin{equation*}
    \mathcal{A}_\delta= \{\lambda \in \mathbb{R}^d: \delta\leq  \|\lambda\|\leq \delta^{-1} \},\hspace{1in}  O_{\delta} = \{\lambda\in\mathbb{R}^d: \|\lambda\|\geq \delta^{-1} \},
\end{equation*}
so by (\ref{eqn:correllation}) and (\ref{eqn:Holder}) it follows that
\begin{equation}
  \label{eqn:breakdown}
  \int_{C_R(0)}|\langle f\circ \varphi_t,f\rangle|^2\, dt \leq \|f\|_{L^2} \sum_{S\in \{B_\delta(0),\mathcal{A}_\delta, O_\delta\} }\int_{S} \left\|\mathcal{S}_R(f,-\lambda)\right\|_{L^2}\, d\bar{\mu}_f(\lambda).
  \end{equation}

Now bounds on the three components of (\ref{eqn:breakdown}) will be obtained; let $f$ be of zero average. The main result of \cite{T:TTT} implies that for any $\varepsilon>0$ there is a $C_\varepsilon>0$ such that
$$\|\mathcal{S}_R(f,0)\| \leq C_\varepsilon\max\left\{R^{\frac{\lambda_2}{\lambda_1}+\varepsilon},R^{(d-1)+\varepsilon}\right\} $$
for all $R>1$, where $\lambda_1>\lambda_2$ are the top two Lyapunov exponents of the trace cocycle (the spectral gap follows\footnote{To see how Horan's result applies, note that the assumptions here about the positively simple measure imply that there is a $k>0$ such that for a positive measure set $Y$ of points in $\Sigma^+_N$, the cocycle at time $k$ has strictly positive entries (the matrices $Q^\pm$ in Definition \ref{def:posSimple}), and thus maps the positive cone into itself, verifying the hypotheses of \cite[Corollary 2.16]{horan:lyapunov}.} from \cite[Corollary 2.16]{horan:lyapunov}). Thus, using (\ref{eqn:HofBnd}) with this estimate it follows that
\begin{equation*}
  \bar{\mu}_f(B_\delta(0))\leq \frac{\delta^{2d}}{C_2} \|\mathcal{S}_{\delta^{-1}}(f,0)\|^2 \leq C_{3,\varepsilon} \delta^{2d} \max\left\{\delta^{-\frac{\lambda_2}{\lambda_1}-\varepsilon},\delta^{1-d-\varepsilon}\right\}
\end{equation*}
for any $\delta\in(0,\frac{1}{2})$. So it follows that for $R>2$:
\begin{equation}
  \label{eqn:originBnd}
  \begin{split}
    \int_{B_\delta(0)}\|\mathcal{S}_R(f,-\lambda)\, \|_{L^2}\, d\bar{\mu}_f(\lambda) & \leq \|f\|_\infty (2R)^d \bar{\mu}_f(B_\delta(0)) \\
    &\leq C_{4,\varepsilon,f}R^d\max\left\{\delta^{2d-\frac{\lambda_2}{\lambda_1}-\varepsilon},\delta^{1+d-\varepsilon}\right\}.
  \end{split}
\end{equation}

The bound for $\mathcal{A}_\delta$ is given by the results of Theorem \ref{thm:main}. More precisely, we have that there exists a $\vartheta>1$ such that for almost every $\Rd\in \mathcal{M}_x$, $\delta\in(0,\frac{1}{2})$,  $\lambda\in \mathcal{A}_\delta$, $f:\Omega_x^{\Rd}\rightarrow \mathbb{R}$ of zero average and $\mathcal{T}\in  \Omega_x^{\Rd}$, $\left|\mathcal{S}_R^\mathcal{T}(f,\lambda)\right|\leq C_{f,x} R^{d-\alpha_\mu}$ for all $R\geq \delta^{-\vartheta}$ (without loss of generality here, by increasing $\vartheta$, it can be assumed that $R_0(x)=1$). This implies that if $\varepsilon\in(0,\alpha_\mu)$ then
\begin{equation}
  \label{eqn:goodTwist}
    \left|\mathcal{S}_R^\mathcal{T}(f,\lambda)\right|\leq C_{f,x}R^{d-\alpha_\mu}\leq C_{f,x} \delta^{\varepsilon \vartheta} R^{d-\alpha_\mu+\varepsilon}
\end{equation}
whenever $R\geq \delta^{-\vartheta}$.

Finally, the case $\lambda\in O_\delta$ will be considered. Let $f:\Omega_x^{\Rd}\rightarrow \mathbb{R}$ be such that $X_\ell f$ is Lipschitz for any index $\ell$. Since $\lambda\in O_\delta$ there is an index $\ell$ such that $|\lambda_\ell|\geq \delta^{-1}d^{-1/2}$. Since $\frac{1}{c}\partial_xe^{cx} = e^{cx}$ and $dt = dt_1\wedge\cdots \wedge dt_d$, the twisted integral can be written as
\begin{equation}
  \label{eqn:IBP}
  \begin{split}
    \mathcal{S}_R^{\mathcal{T}}(f,-\lambda) &= \int_{-R}^R\cdots \int_{-R}^R e^{2\pi \imath \langle\lambda,t\rangle}f\circ\varphi_{t}(\mathcal{T})\, dt \\
    &= \frac{1}{2\pi \imath \lambda_\ell}\int_{-R}^R\cdots \int_{-R}^R\left( \frac{\partial}{\partial t_\ell} e^{2\pi \imath \langle \lambda, t\rangle }\right)f\circ\varphi_{(t_1,\dots, t_d)}(\mathcal{T})\, dt_\ell\wedge \star dt_\ell ,
  \end{split}
\end{equation}
where $\star$ is the Hodge-$\star$ operator taking $1$-forms to $d-1$-forms. Doing integration by parts, it follows that
\begin{equation}
  \label{eqn:1iteration}
  \begin{split}
    \mathcal{S}_R^{\mathcal{T}}(f,-\lambda) & = \frac{1}{2\pi \imath \lambda_\ell}\left[ \int_{\partial_\ell^- C_R(0)} e^{2\pi \imath \langle \lambda, t\rangle}f\circ \varphi_{t}(\mathcal{T})\, \star dt_\ell - \int_{\partial_\ell^+ C_R(0)} e^{2\pi \imath \langle \lambda, t\rangle}f\circ \varphi_{t}(\mathcal{T}) \star dt_\ell \right]\\
  &  \hspace{1in}+ \frac{1}{2\pi \imath \lambda_\ell}\mathcal{S}_R(X_\ell f,-\lambda),
  \end{split}
\end{equation}
where $\partial^\pm_\ell C_R(0)$ are the two $d-1$ dimensional faces of $C_R(0)$ given when $t_\ell = \pm R$. Therefore the twisted integral can be bounded by
\begin{equation*}
  \begin{split}
    \|\mathcal{S}_R(f,-\lambda)\|_{L^2} &\leq \frac{1}{2\pi |\lambda_\ell|}\left( 2(2R)^{d-1}\|f\|_{L^2} + \|\mathcal{S}_R(X_\ell f,-\lambda)\|_{L^2}\right) \\
    &\leq \frac{\delta}{2\pi \sqrt{d}}\left( 2(2R)^{d-1}\|f\|_{L^2} + \|\mathcal{S}_R(X_\ell f,-\lambda)\|_{L^2}\right).
  \end{split}
\end{equation*}
Since $X_\ell f$ is Lipshitz, its twisted integral is bounded for $R\geq \delta^{-\vartheta}$ as in (\ref{eqn:goodTwist}), and so of $\lambda\in O_\delta$,
\begin{equation}
  \label{eqn:Opart}
    \|\mathcal{S}_R(f,-\lambda)\|_{L^2} \leq \frac{\delta}{2\pi \sqrt{d}}\left( 2(2R)^{d-1}\|f\|_{L^2} + C_{Xf,x}\|\lambda\|^{-\varepsilon \vartheta}R^{d-\alpha_\mu+\varepsilon}\right).
\end{equation}

Now everything needs to come together in (\ref{eqn:breakdown}). Using (\ref{eqn:originBnd}), (\ref{eqn:goodTwist}) and (\ref{eqn:Opart}) it follows that
\begin{equation*}
  \begin{split}
    &\|f\|_{L^2}^{-1}\int_{C_R(0)}|\langle f\circ \varphi_t,f\rangle|^2\, dt \leq  C_{4,\varepsilon,f}R^d\max\left\{\delta^{2d-\frac{\lambda_2}{\lambda_1}-\varepsilon},\delta^{1+d-\varepsilon}\right\} +  C_{f,x} \delta^{\varepsilon \vartheta} R^{d-\alpha_\mu+\varepsilon} \\
    &\hspace{2.4in}+\frac{\delta}{2\pi \sqrt{d}}\left( 2(2R)^{d-1}\|f\|_{L^2} + C_{Xf,x}R^{d-\alpha_\mu+\varepsilon}\right)
  \end{split}
\end{equation*}
for $R\geq \delta^{-\vartheta}$. Picking $\delta = R^{-\alpha}$ for $\alpha\in (0,\vartheta^{-1})$, then this is satisfied for any $R>1$. Thus there is a constant $C_{f,Xf,x,\varepsilon}$ such that
\begin{equation}
  \label{eqn:finalCorr}
  \begin{split}
    &\int_{C_R(0)}|\langle f\circ \varphi_t,f\rangle|^2\, dt \\
    &\hspace{.5in}\leq  C_{f,Xf,x,\varepsilon}\max\left\{R^{d-\alpha\left(2d-\frac{\lambda_2}{\lambda_1}-\varepsilon\right)},R^{d-\alpha(1+d-\varepsilon)},R^{d-\alpha_\mu+\varepsilon(1-\alpha\vartheta)},R^{d-1-\alpha},R^{d-\alpha_\mu+\varepsilon-\alpha}\right\} 
  \end{split}
\end{equation}
for all $R>1$.
\begin{proof}[Proof of (\ref{eqn:avgCorr}) in Theorem \ref{thm:main2}]
Let $S_1,\dots, S_N$ be a collection of $N$ uniformly expanding compatible substitution rules on the $M$ prototiles $t_1,\dots, t_M$ of dimension $d$ and $\mu$ a $\sigma$-invariant positively simple, postal ergodic probability measure on $\Sigma_N$ such that $\log \|\Phi_x^*\|\in L^1_\mu$, where $\Phi^*_x$ is the map on the first cohomology of $\Omega_x$ induced by the shift, and assume that $\mathrm{dim}(Z^+_x)>d$. Pick an Oseledets regular $x$ and the $\theta$ afforded by Theorem \ref{thm:main} and $[\Rd]\in \mathcal{M}_x$ with good representative $\Rd$ such that the conclusions of Theorem \ref{thm:main} hold. Let $f,g: \Omega_x^{\Rd}\rightarrow \mathbb{R}$ be such that $X_\ell f, X_\ell g$ are Lipschitz for any index $\ell$. The bound (\ref{eqn:finalCorr}) above implies the existence of an $\alpha_\mu'>0$ such that 
    \begin{equation}
      \label{eqn:avgCorr2}
  \int_{C_R(0)}|\langle f\circ \varphi_t,g\rangle|^2\, dt \leq C_{f,g,x} R^{d-\alpha'_\mu}
    \end{equation}
    for all $R>1$. Finally, by using the Cauchy-Schwartz inequality in (\ref{eqn:avgCorr2}), (\ref{eqn:avgCorr}) follows.
\end{proof}
 \subsection{Proof of uniform bounds for twisted integrals in Theorem \ref{thm:main2}}Let $S_1,\dots, S_N$ be a collection of $N$ uniformly expanding compatible substitution rules on the $M$ prototiles $t_1,\dots, t_M$ of dimension $d$ and $\mu$ a $\sigma$-invariant measure satisfying the hypotheses of Theorem \ref{thm:main}. Pick an Oseledets regular $x$ and $\Rd\in \mathcal{M}_x$ such that the conclusions of Theorem \ref{thm:main} hold, and so by the previous subsection, (\ref{eqn:avgCorr}) also holds. The argument here follows the argument of Venkatesh in \cite[Lemma 3.1]{venkatesh:sparse}.

 First, some estimates need to be worked out. For a Lipschitz function $f$ with $X_if$ Lipshitz for all $i$, $\mathcal{T}\in\Omega_x$, $t\in\mathbb{R}^d$ and $H>0$ define
$$D_H^\mathcal{T}(f,t) := \int_{C_H(0)}e^{-2\pi \imath \langle \lambda,x  \rangle}f\circ\varphi_{x+t}(\mathcal{T})\, dx,$$
and note that
$$D_H^\mathcal{T}(f,t)^2 = \int_{C_H(0)^2}e^{-2\pi \imath \langle \lambda,x+y  \rangle}f\circ\varphi_{x+t}(\mathcal{T}) f\circ\varphi_{y+t}(\mathcal{T})\, dx \, dy,$$
so
$$\left|\int_{C_R(0)} D_H^\mathcal{T}(f,t)^2\, dt\right| \leq \left|\int_{C_H(0)^2} \int_{C_R(0)} (f\circ\varphi_{x} f\circ\varphi_{y})\circ\varphi_t(\mathcal{T})\, dt\, dx\, dy\right|.$$
In addition, by \cite{T:TTT},
\begin{equation}
  \label{eqn:corBnd1}
  \begin{split}
    \int_{C_R(0)} (f\circ\varphi_x\cdot f\circ\varphi_y)\circ\varphi_t(\mathcal{T})\, dt &= (2R)^d \mu(f\circ\varphi_x\cdot f\circ\varphi_y) + \mathcal{O}\left(\max\left\{R^{d\frac{\lambda_2}{\lambda_1}},R^{d-1}\right\}\right) \\
   &= (2R)^d\langle f\circ\varphi_{x-y},f\rangle  + \mathcal{O}\left(\max\left\{R^{d\frac{\lambda_2}{\lambda_1}},R^{d-1}\right\}\right),
  \end{split}
\end{equation}
since $f\circ\varphi_x\cdot f\circ\varphi_y$ is Lipschitz whenever $f$ is, and where $\lambda_1>\lambda_2$ are the top two Lyapunov exponents of the trace cocycle (the inequality is strict by \cite[Corollary 2.16]{horan:lyapunov}). 
Now, the twisted integral of $D_H^{\mathcal{T}}(f,t)$ is
\begin{equation}
  \label{eqn:CS1}
  \begin{split}
    \int_{C_R(0)} e^{-2\pi \imath \langle \lambda,t \rangle} D_H^{\mathcal{T}}(f,t) \, dt &= \int_{C_H(0)}\int_{C_R(0)}e^{-2\pi \imath \langle\lambda, s+t\rangle}f\circ \varphi_{s+t}(\mathcal{T})\, dt\, ds \\
    &= \int_{C_H(0)}\int_{\varphi_{s}(C_R(0))} e^{-2\pi \imath \langle \lambda, z\rangle} f\circ \varphi_z(\mathcal{T})\, dz\, ds
  \end{split}
\end{equation}
after using Fubini and changing variables $z = s+t$.
\begin{lemma}
  \label{lem:VenkBound}
  The difference between the twisted integral in (\ref{eqn:CS1}) and $(2H)^d\mathcal{S}_R^\mathcal{T}(f,\lambda)$ is bounded as
\begin{equation}
  \label{eqn:CS2}
  \begin{split}
    &\left|  (2H)^{d}\int_{C_R(0)} e^{-2\pi \imath \langle\lambda, t\rangle }f\circ \varphi_t(\mathcal{T})\, dt -  \int_{C_R(0)} e^{-2\pi \imath \langle \lambda,t \rangle} D_H^{\mathcal{T}}(f,t) \, dt\right| \\
    & \hspace{.4in}= \left| \int_{C_H(0)}\int_{C_R(0)} e^{-2\pi \imath \langle \lambda, t\rangle} f\circ \varphi_t(\mathcal{T})\, dt - \int_{\varphi_s(C_R(0))}e^{-2\pi \imath \langle \lambda, t\rangle}f\circ \varphi_t(\mathcal{T})\, dt\, ds\right|\\
    &\hspace{.4in}\leq d 2^{2d+1}R^{d-1}H^{d+1}\|f\|_\infty.
  \end{split}
\end{equation}
\end{lemma}
\begin{proof}
  Let $\{e_1,\dots, e_d\}$ be the standard Euclidean basis for $\mathbb{R}^d$. First , for $i\in\{1,\dots, d\}$, $s_1,s_2,\dots, s_i\in \mathbb{R}$ and $\lambda\in\mathbb{R}^d$, define
  \begin{equation}
    \label{eqn:Ii}
    \begin{split}
      I_{i,R,s_1,s_2,\dots, s_i,\lambda}(\mathcal{T}) &:=  \int_{-R}^R \cdots\int_{-R}^Re^{-2\pi \imath \langle \lambda ,t\rangle} f\circ \varphi_{e_1t_1 + \cdots + e_it_i}(\mathcal{T})\, dt_1\cdots \, dt_i\\
      &\hspace{1in}- \int_{-R+s_i}^{R+s_i} \cdots\int_{-R+s_1}^{R+s_1}e^{-2\pi \imath \langle \lambda ,t\rangle} f\circ \varphi_{e_1t_1 + \cdots + e_it_i}(\mathcal{T})\, dt_1\cdots \, dt_i.
    \end{split}
  \end{equation}
  As such, we have that $|I_{1,R,s_1,\lambda}(\mathcal{T})|\leq 2|s_1|\cdot\|f\|_\infty$. Indeed:
  \begin{equation*}
    \begin{split}
      \left(\int_{-R}^R - \int_{-R+s_1}^{R+s_1}\right) e^{-2\pi \imath \langle \lambda ,t\rangle} f\circ \varphi_t\, dt &= \left(\int_{-R}^{-R+s_1} + \int_{-R+s_1}^R - \int_{-R+s_1}^R - \int_{R}^{R+s_1}\right) e^{-2\pi \imath \langle \lambda ,t\rangle} f\circ \varphi_t\, dt \\
      &= \left(\int_{-R}^{-R+s_1}  - \int_{R}^{R+s_1}\right) e^{-2\pi \imath \langle \lambda ,t\rangle} f\circ \varphi_t\, dt 
    \end{split}
  \end{equation*}
  and thus the bound follows. Note that
  \begin{equation}
    \label{eqn:transfer1}
    \begin{split}
      I_{2,R,s_1,s_2,\lambda}(\mathcal{T}) &= \int_{-R}^RI_{1,R,s_1,\lambda}\circ \varphi_{e_2t_2}(\mathcal{T})\, dt_2 
      +   \int_{-R}^{R} \int_{-R+s_1}^{R+s_1}f\circ \varphi_{e_1t_1 + e_2t_2}(\mathcal{T})\, dt_1dt_2  \\
      &\hspace{2.5in}- \int_{-R+s_2}^{R+s_2} \int_{-R+s_1}^{R+s_1}f\circ \varphi_{e_1t_1 + e_2t_2}(\mathcal{T})\, dt_1dt_2
    \end{split}
  \end{equation}
  and so it follows that
  \begin{equation}
    \label{eqn:transferBound1}
    \begin{split}
      |I_{2,R,s_1,s_2,\lambda}(\mathcal{T})|&\leq 2R|I_{1,R,s_1,\lambda}(\mathcal{T})| + 2|s_2|2R\|f\|_\infty \\
      &\leq 2R2|s_1|\|f\|_\infty + 4R|s_2|\|f\|_\infty \\
      &\leq 4R(|s_1|+|s_2|)\|f\|_\infty.
    \end{split}
  \end{equation}
  Now we can proceed recursively: first, writing as in (\ref{eqn:transfer1})
  \begin{equation}
    \label{eqn:transfer2}
    \begin{split}
      I_{i,R,s_1,\dots, s_i,\lambda}(\mathcal{T}) &= \int_{-R}^RI_{1,R,s_1,\dots,s_{i-1},\lambda}\circ \varphi_{e_it_i}(\mathcal{T})\, dt_1\cdots dt_i \\
      &\hspace{.75in}      +   \int_{-R}^{R} \int_{-R+s_{i-1}}^{R+s_{i-1}}\cdots \int_{-R+s_1}^{R+s_1}f\circ \varphi_{e_1t_1+\cdots + e_it_i}(\mathcal{T})\, dt_1\cdots dt_i  \\
      &\hspace{1.5in}- \int_{-R+s_i}^{R+s_i}\cdots \int_{-R+s_1}^{R+s_1}f\circ \varphi_{e_1t_1+\cdots + e_it_i}(\mathcal{T})\, dt_1\cdots dt_i
    \end{split}
  \end{equation}
  and then bounding as in (\ref{eqn:transferBound1}):
    \begin{equation}
    \label{eqn:transferBound2}
    \begin{split}
      |I_{i,R,s_1,\dots , s_i,\lambda}(\mathcal{T})|
      &\leq 2R|I_{i-1,R,s_1,\dots, s_{i-1},\lambda}(\mathcal{T})| + 2|s_i|(2R)^{i-1}\|f\|_\infty \\
      &\leq 2R\cdot 2^{i-1}R^{i-2}(|s_1|+\dots+|s_{i-1}|)\|f\|_\infty + 2^iR^{i-1}|s_i|\|f\|_\infty \\
      &\leq 2^iR^{i-1}(|s_1|+|s_2|+\cdots+|s_i|)\|f\|_\infty.
    \end{split}
  \end{equation}
    Note that the quantity appearing in (\ref{eqn:CS2}) is
    $$I_{d,R,s_1,\dots, s_d,\lambda}(\mathcal{T}) = \int_{C_R(0)} e^{-2\pi \imath \langle \lambda, t\rangle} f\circ \varphi_t(\mathcal{T})\, dt - \int_{\varphi_s(C_R(0))}e^{-2\pi \imath \langle \lambda, t\rangle}f\circ \varphi_t(\mathcal{T})\, dt$$
    and so (\ref{eqn:transferBound2}) gives the bound
    $$|I_{d,R,s_1,\dots, s_d,\lambda}(\mathcal{T})| \leq 2^dR^{d-1}\|s\|_1\|f\|_\infty$$
    which can be used to obtain the desired bound (\ref{eqn:CS2}):
    \begin{equation}
      \begin{split}
        \left|\int_{C_H(0)}I_{d,R,s_1,\dots, s_d,\lambda}(\mathcal{T})\, ds\right| &\leq \int_{C_H(0)}\,  2^dR^{d-1}\|s\|_1\|f\|_\infty ds= 2^dR^{d-1}\|f\|_\infty \int_{C_H(0)} \|s\|_1\,   ds\\
        &= d2^dR^{d-1}\|f\|_\infty 2^{d-1}H^{d+1} = d2^{2d+1}R^{d-1}H^{d+1}\|f\|_\infty
      \end{split}
    \end{equation}
    completing the proof.
\end{proof}
The twisted integral $\mathcal{S}_R^{\mathcal{T}}(f,\lambda)$ can now be bounded:
\begin{equation}
  \label{eqn:final1}
  \begin{split}
    \left|\mathcal{S}_R^{\mathcal{T}}(f,\lambda)\right| &\leq \left|\mathcal{S}_R^{\mathcal{T}}(f,\lambda) - \frac{1}{(2H)^d}\int_{C_R(0)} e^{-2\pi \imath \langle \lambda,t \rangle} D_H^{\mathcal{T}}(f,t) \, dt\right| + \frac{1}{(2H)^d}\left|\int_{C_R(0)} e^{-2\pi \imath \langle \lambda,t \rangle} D_H^{\mathcal{T}}(f,t) \, dt\right| \\
    &\leq d2^{d+1}R^{d-1}H\|f\|_\infty + \frac{1}{(2H)^d}\left|\int_{C_R(0)} e^{-2\pi \imath \langle \lambda,t \rangle} D_H^{\mathcal{T}}(f,t) \, dt\right|,
    \end{split}
\end{equation}
where the first term comes from the bound in Lemma \ref{lem:VenkBound}. Finally, to bound the last term in (\ref{eqn:final1}), let $\lambda^* = \max\left\{d\frac{\lambda_2}{\lambda_1},d-1\right\}$. Then:
\begin{equation}
  \label{eqn:CS3}
  \begin{split}
    &\frac{1}{(2H)^{2d}}\left|\int_{C_R(0)} e^{-2\pi \imath \langle \lambda,t \rangle} D_H^{\mathcal{T}}(f,t) \, dt\right|^2 = \frac{1}{(2H)^{2d}} \left|\left\langle e^{-2\pi \imath \langle\lambda, \cdot\rangle} D_H^{\mathcal{T}}(f,\cdot),1\right\rangle_{L^2(C_R(0))} \right|^2 \\
    &\hspace{2in}\leq \frac{(2R)^d}{(2H)^{2d}}\int_{C_R(0)}D_H^\mathcal{T}(f,t)^2\, dt \\
    &\hspace{1.8in}=  \frac{(2R)^d}{(2H)^{2d}}\int_{C_R(0)}\int_{C_H(0)^2}e^{-2\pi \imath \langle\lambda,x+y\rangle}f\circ\varphi_{t+x}(\mathcal{T})\cdot f\circ\varphi_{t+y}(\mathcal{T}) \, dx\, dy\, dt \\
    &\hspace{1.6in}\leq \frac{(2R)^d}{(2H)^{2d}} \int_{C_H(0)^2} \left| \int_{C_R(0)}f\circ\varphi_{t+x}(\mathcal{T})\cdot f\circ\varphi_{t+y}(\mathcal{T}) \, dt\right| \, dx\, dy \\
    &\hspace{1.4in}\leq \frac{(2R)^d}{(2H)^{2d}} \int_{C_H(0)^2} \left| (2R)^d\langle f\circ\varphi_{x-y},f\rangle  + \mathcal{O}\left( R^{\lambda^*}\right)\right| \, dx\, dy \\
    &\hspace{1.2in}\leq \frac{R^{2d}}{H^{2d}} \int_{C_H(0)^2} \left| \langle f\circ\varphi_{x-y},f\rangle\right|\, dx\, dy  +C_\varepsilon  R^{d+\lambda^*+\varepsilon} \\
    &\hspace{1in}\leq R^{2d}C_{f,\varepsilon}^2 H^{-\alpha_\mu'/2+2\varepsilon} +C_\varepsilon  R^{d+\lambda^*+\varepsilon} \\
    &\hspace{.8in}\leq C'_{f,\varepsilon}\left(R^{2d}H^{-\alpha_\mu'/2+2\varepsilon}+R^{d+\lambda^*+\varepsilon}\right),
  \end{split}
\end{equation}
for $R>1$, where the first inequality follows from the Cauchy-Schwarz inequality, (\ref{eqn:corBnd1}) was used in the third inequality and (\ref{eqn:avgCorr}) in the fifth. Chosing $H = R^{\frac{4}{4+\alpha_\mu'}}$ it follows that $R^{d}H^{-\alpha_\mu'/4}=R^{d-1}H$, so using in (\ref{eqn:final1}) this choice of $H$:
\begin{equation*}
  \label{eqn:final2}
    \left|\mathcal{S}_R^{\mathcal{T}}(f,\lambda)\right| \leq  C''_{f,r,\varepsilon}\max\left\{R^{d-1+\frac{4}{4+\alpha'_\mu}+\varepsilon},R^{\frac{d+\lambda^*}{2}+\varepsilon}\right\} = C''_{f,r,\varepsilon} \max\left\{R^{d-\beta_1+\varepsilon},R^{d-\beta_2+\varepsilon}\right\}
\end{equation*}
for $R>1$, where
$$\beta_1 = 1-\frac{4}{4+\alpha'_\mu} = \frac{\alpha'_\mu}{4+\alpha'_\mu}\hspace{.6in}\mbox{ and }\hspace{.6in}\beta_2 = \frac{d-\lambda^*}{2}.$$
As such, by Lemma \ref{lem:dimension},
$$d^-_f(\lambda)\geq \min\left\{\frac{2\alpha'_\mu}{4+\alpha'_\mu}, d\left(1-\frac{\lambda_1}{\lambda_2}\right)\right\}$$
for any $\lambda\in\mathbb{R}^d$, which finishes the proof.

\appendix
\section{Twisted cohomology for tiling spaces}
\label{sec:plotTwist}
Let $\bar{\Delta}^k_\mathcal{T}$ the set of smooth ($C^\infty$), $\mathbb{C}$-valued $\mathcal{T}$-equivariant $k$-forms. Let $\eta\in\Delta^1_\mathcal{T}$ be a representative of a class $[\eta] \in H^1(\Omega_\mathcal{T};\mathbb{R})$ and define the operator $d_\eta:\bar{\Delta}^k_\mathcal{T}\rightarrow \bar{\Delta}^{k+1}_\mathcal{T}$
$$d_\eta:\alpha\mapsto d_\eta\alpha:= d\alpha - 2\pi \imath \eta\wedge \alpha. $$
It is immediate to check that
$$d^2_\eta\alpha = (d - 2\pi \imath \eta\wedge )(d\alpha - 2\pi \imath \eta\wedge \alpha) = -2\pi \imath d\eta\wedge \alpha = 0,$$
making $(\bar{\Delta}^k, d^2_\eta)$ a cochain complex with well-defined cohomology.
\begin{definition}
  For any real, closed $\eta\in\Delta^1$, the \textbf{twisted cohomology spaces of $\Omega$} are defined as the quotients
  $$H_\eta^k(\Omega;\mathbb{C}) := \frac{\mathrm{ker}\, d_\eta:\bar{\Delta}^k\rightarrow \bar{\Delta}^{k+1}}{\mathrm{Im}\, d_\eta:\bar{\Delta}^{k-1}\rightarrow \bar{\Delta}^k }.$$
\end{definition}
Any $\lambda\in\mathbb{R}^d$ defines a natural constant, closed, $\mathcal{T}$-equivariant $1$-form $\sum \lambda_i\, dx_i$. For any such constant form the twisted cohomology is denoted by $H_\lambda^*(\Omega;\mathbb{C})$.
\begin{proposition}
  \label{prop:Fact1}
  \begin{enumerate}
  \item The first twisted cohomology $H^1_\eta(\Omega;\mathbb{C})$ is defined, up to unitary equivalence, by the cohomology class $[\eta]\in H^1(\Omega;\mathbb{R})$. 
  \item For a constant 1-form $\lambda\in \mathbb{R}^d$, the twisted differential $d_\lambda$ satisfies $d_\lambda = U_\lambda^{-1} d U_\lambda$, where $U_\lambda$ is the twisted multiplication operator $U_\lambda:\omega(x)\mapsto e^{-2\pi \imath \langle\lambda, x\rangle}\omega(x)$ \textbf{which does not preserve pattern-equivariance}.
  \end{enumerate}
\end{proposition}
The first item in the proposition is essentially due to Forni \cite[Lemma 4.2]{forni:twist}.
\begin{proof}
  For the first claim, consider two closed 1-forms $\eta, \eta'\in\Delta^1_\mathcal{T}$ such that $\eta-\eta' = df$ for some $f\in\bar\Delta_\mathcal{T}^0$. Note that $e^{-2\pi \imath f}\in\bar\Delta^1_\mathcal{T}$, and so for any form $\omega\in\bar{\Delta}_\mathcal{T}^k$:
  $$d_\eta(e^{-2\pi \imath f}\omega) = e^{-2\pi \imath f}(-2\pi \imath df \wedge \omega  + d\omega - 2\pi \imath \eta \wedge \omega) = e^{-2\pi \imath f} d_{\eta'}\omega.$$
  This means that if $\omega$ is $d_{\eta'}$-closed, then $e^{-2\pi \imath f}\omega$ is $d_\eta$-closed. Thus multiplication by $e^{-2\pi \imath f}$ sends the $d_{\eta'}$-closed, pattern equivariant forms to the $d_{\eta}$-closed pattern equivariant forms. Similarly, multiplication by $e^{-2\pi \imath f}$ gives a bijection between $d_{\eta'}$-exact pattern equivariant forms and $d_\eta$-exact pattern equivariant forms.

  For the second claim, a straight-forward computation shows that $d(U_\lambda\omega) = U_\lambda d_\lambda \omega$ for any $\omega\in \Delta_\mathcal{T}^k$, from which the claim follows. 
\end{proof}

\begin{proposition}
 Let $\eta\in\Delta^1$ be a real, closed 1-form. Then $H^0_\eta(\Omega;\mathbb{C})$ is non-trivial if and only if $\eta$ is a constant 1-form $\sum_i\lambda_i \,dx_i$ for some $\lambda\in\mathbb{R}^d$ which is an eigenvalue for the $\mathbb{R}^d$ action on $\Omega$, in which case it is one-dimensional.
\end{proposition}
\begin{proof}
  For $\lambda\in\mathbb{R}^d$, let $\eta$ be the constant form given by $\lambda$. As such for $f\in\bar\Delta^0_\mathcal{T}$,
  $$d_\eta f = \sum_{k=1}^d (\partial_k f - 2\pi \imath f\lambda_k)\, dx_k = 0$$
  implies that $\partial_k f = 2\pi \imath\lambda_k f $ for all $k$. For $k=j$, this differential equation has solution
  $$f(x_1,\dots, x_k) = C_j \exp\left(2\pi \imath \lambda_jx_j  \right),$$
  where $C_j$ does not depend on $x_j$. This implies that
  $$f(x) = C e^{2\pi \imath \langle \lambda,x\rangle},$$
which makes it an eigenfunction of the $\mathbb{R}^d$ action on $\Omega$ by the correspondence (\ref{eqn:imap}). Since the multiplicity of any dynamical eigenvalue is one, $H^0_\eta(\Omega;\mathbb{C})$ is one-dimensional.

  Alternatively, suppose $\eta  = \sum_k\eta_k\,dx_k$ is not constant and suppose $\eta_j$ is not constant. Let $f = R+iI$ be a smooth function. If $f$ is $d_\eta$-closed:
  $$d_\eta f = dR - 2\pi \imath R\eta + \imath(dI-2\pi \imath I \eta) = 0,$$
  then $dR = -2\pi \imath I\eta$ and $dI = 2\pi \imath R\eta$, from which it follows that
  \begin{equation}
    \begin{split}
      d_\eta\bar{f} &= dR - 2\pi \imath R\eta - \imath(dI-2\pi \imath I \eta)\\
      &= -2\pi \imath I\eta - 2\pi \imath R\eta - \imath(2\pi \imath R\eta-2\pi \imath I \eta) \\
      &= -2\pi \imath (I+R)\eta + 2\pi( R \eta -  I \eta) \\
      &= 2\pi (R-\imath I - I - \imath R)\eta = 2\pi (\bar{f}-\imath\bar{f})\eta \\
      &= 2\pi (1-\imath)\eta\bar{f}.
    \end{split}
  \end{equation}
Putting this together with the definition of $d_\eta$, it follows that $d\bar{f} = 2\pi \eta \bar{f} $, and therefore also that $\partial_k \bar{f} = 2\pi \bar{f} \eta_k$ for all $k$. Also, by definition,
  $$d_\eta f = \sum_{k=1}^d (\partial_k f - 2\pi \imath f\eta_k)\, dx_k = 0$$
  implies that $\partial_k f = 2\pi \imath f \eta_k$ for all $k$. For $k=j$, this differential equation has solution
  \begin{equation}
    \label{eqn:difEq}
    f(x_1,\dots, x_k) = C_j \exp\left(2\pi \imath \int_0^{x_j} \eta_j(x_1,\dots, z_j, \dots, x_d)\, dz_j\right),
  \end{equation}
  where $C_j$ does not depend on $x_j$. Similarly, the differential equation $\partial_j \bar{f} = 2\pi \bar{f} \eta_j$ has solution
  $$\bar{f}(x_1,\dots, x_k) = C'_j \exp\left(2\pi \int_0^{x_j} \eta_j(x_1,\dots, z_j, \dots, x_d)\, dz_j\right),$$
  where $C'_j$ does not depend on $x_j$. Taking the complex conjugate in (\ref{eqn:difEq}), it follows that
  $$C'_j \exp\left(2\pi \int_0^{x_j} \eta_j(x_1,\dots, z_j, \dots, x_d)\, dz_j\right)  = C_j \exp\left(-2\pi \imath \int_0^{x_j} \eta_j(x_1,\dots, z_j, \dots, x_d)\, dz_j\right) .$$
  It follows that, as a function of $x_j$,
  $$\mathrm{Const.} = \exp\left(2\pi (1+\imath) \int_0^{x_j} \eta_j(x_1,\dots, z_j, \dots, x_d)\, dz_j\right), $$
which contradicts that $\eta_j$ is not constant. So there is no solution to $d_\eta f = 0$ and $H^0_\eta(\Omega;\mathbb{C})$ is trivial.
\end{proof}
 Given a pattern equivariant $d$-form $\omega\in\bar\Delta^d_\mathcal{T}$, one would like to know when the cohomological equation
\begin{equation}
  \label{eqn:CEsolution}
  d_\lambda u = \omega,
\end{equation}
can be solved, that is, when we can find a $u\in\bar\Delta^{d-1}_\mathcal{T}$ which satisfies $d_\lambda u = \omega$.

Let $\sum_i \lambda_i \,dx_i \in \Delta^1$ be a constant 1-form and, for $R>0$, $\Upsilon^\lambda_R:\bar\Delta^d_\mathcal{T}\rightarrow \mathbb{C}$ the one-parameter family of currents defined by
\begin{equation}
  \label{eqn:current}
  \Upsilon^\lambda_R (\omega) := \mathcal{S}^\mathcal{T}_R(i^{-1}_\mathcal{T}\star \omega,\lambda) = \int_{C_R(0)}U_\lambda \omega.
\end{equation}
\begin{proposition}
\label{prop:CEsolution}
Let $u\in\bar\Delta^{d-1}_\mathcal{T}$ be a pattern-equivariant solution of (\ref{eqn:CEsolution}). There exists $C'>0$ such that the spectral measure $\mu_f$ of $f = i_\mathcal{T}^{-1}\star \omega$ satisfies
  $$\mu_f(B_r(\lambda))\leq C'r^{2}$$
  for all $r<1/2$. In particular, the lower local dimension satisfies $2\leq d^-_{i^{-1}_\mathcal{T}\star \omega}(\lambda)$. Therefore, if (\ref{eqn:CEsolution}) has a solution and $\lambda$ is an eigenvalue for the system, then $f=\star \omega$ lies in the orthogonal complement of $h_\lambda$, where $h_\lambda$ is an eigenfunction with eigenvalue $\lambda$:
  $$\star d_\lambda\Delta^{d-1}_\mathcal{T}\subset \langle h_\lambda\rangle^\perp.$$
\end{proposition}
\begin{proof}
  Recall by Proposition \ref{prop:Fact1} that $d_\lambda = U^{-1}_\lambda dU_\lambda$, so supposing that $d_\lambda\eta = \omega$,
  \begin{equation}
    \begin{split}
      \Upsilon_R^\lambda(\omega) &= \Upsilon_R^\lambda(d_\lambda\eta) = \Upsilon_R^\lambda(U^{-1}_\lambda d U_\lambda\eta) = \int_{C_R(0)}U_\lambda U_\lambda^{-1} d U_\lambda \eta \\
      &= \int_{C_R(0)} d U_\lambda \eta = \int_{\partial C_R(0)} U_\lambda \eta.
    \end{split}
  \end{equation}
  It follows that
  $\left| \Upsilon_R^\lambda(\omega) \right|\leq \|\eta\|_\infty (2R)^{d-1}$ for all $R>0$. The pointwise dimension estimates follow from (\ref{eqn:current}) and Lemma \ref{lem:dimension}.
\end{proof}


\bibliographystyle{amsalpha}
\bibliography{biblio}
\end{document}